\documentclass[11pt,a4paper]{article}

\usepackage{geometry}
\usepackage{hyperref} 

\usepackage{amsmath}
\usepackage{amssymb}
\usepackage{amscd}
\usepackage{textcomp}
\usepackage{dsfont}
\usepackage{mathrsfs}

\long\def\forget#1{}

\newcommand{\lang}[1]{\mbox{#1}}

\usepackage{color}

\newcommand{\comment}[1]{}{
\immediate\write16{}
\immediate\write16{Warning: There was still a comment . . . }
\immediate\write16{}}

\def\?{\ 
{\bf\color{red}???}\ 
\immediate\write16{}
\immediate\write16{Warning: There was still a question mark . . . }
\immediate\write16{}}

\usepackage{amsthm}
\theoremstyle{plain}
\newtheorem{theorem}{Theorem}[section]

\newtheorem{lemma}[theorem]{Lemma}

\newtheorem{corollary}[theorem]{Corollary}
\newtheorem{proposition}[theorem]{Proposition}

\theoremstyle{definition}
\newtheorem{definition}[theorem]{Definition}
\newtheorem{definition-theorem}[theorem]{Definition-Theorem}
\newtheorem{definition-remark}[theorem]{Definition-Remark}
\newtheorem{point}[theorem]{}
\newtheorem{example}[theorem]{Example}
\newtheorem{remark}[theorem]{Remark}

\theoremstyle{remark}

\usepackage{xy}
\xyoption{all}

\newdir^{ (}{{}*!/-3pt/\dir^{(}}    
\newdir^{  }{{}*!/-3pt/\dir^{}}    
\newdir_{ (}{{}*!/-3pt/\dir_{(}}    
\newdir_{  }{{}*!/-3pt/\dir_{}}

\newcounter{zahl}

\def\theenumi{(\alph{enumi})}

\def\p@enumii{\theenumi}

\newcommand{\DS}{\displaystyle}
\newcommand{\TS}{\textstyle}
\newcommand{\SC}{\scriptstyle}
\newcommand{\SSC}{\scriptscriptstyle}

\newcommand{\cG}{\mathcal{G}}

\DeclareMathOperator{\Aut}{Aut}

\DeclareMathOperator{\Et}{\acute{E}t}

\DeclareMathOperator{\Frob}{Frob}
\DeclareMathOperator{\Gal}{Gal}
\DeclareMathOperator{\GL}{GL}

\DeclareMathOperator{\Hom}{Hom}
\newcommand{\CHom}{{\cal H}om}

\DeclareMathOperator{\Ind}{Ind}
\DeclareMathOperator{\Int}{Int}
\DeclareMathOperator{\Isom}{Isom}

\DeclareMathOperator{\PGL}{PGL}

\DeclareMathOperator{\Quot}{Frac}

\DeclareMathOperator{\Rep}{Rep}
\DeclareMathOperator{\Res}{Res}

\DeclareMathOperator{\SL}{SL}

\DeclareMathOperator{\Spec}{Spec}
\DeclareMathOperator{\Spf}{Spf}

\DeclareMathOperator{\Var}{V}

\newcommand{\ad}{{\rm ad}}
\newcommand{\alg}{{\rm alg}}

\newcommand{\dom}{{\rm dom}}
\DeclareMathOperator{\diag}{diag}

\DeclareMathOperator{\equi}{equi}
\newcommand{\et}{{\rm \acute{e}t\/}}
\newcommand{\fppf}{{\it fppf\/}}
\newcommand{\fpqc}{{\it fpqc\/}}

\DeclareMathOperator{\id}{\,id}
\DeclareMathOperator{\im}{im}

\renewcommand{\mod}{{\rm\,mod\,}}

\newcommand{\red}{{\rm red}}

\DeclareMathOperator{\rk}{rk}

\newcommand{\topol}{{\rm top}}

\DeclareMathOperator{\whtimes}{\mathchoice
            {\wh{\raisebox{0ex}[0ex]{$\DS\times$}}}
            {\wh{\raisebox{0ex}[0ex]{$\TS\times$}}}
            {\wh{\raisebox{0ex}[0ex]{$\SC\times$}}}
            {\wh{\raisebox{0ex}[0ex]{$\SSC\times$}}}}

\renewcommand{\phi}{\varphi}
\renewcommand{\epsilon}{\varepsilon}

\usepackage{amsfonts}

\usepackage{xcolor}
\usepackage{graphicx}
\newcommand{\Bmu}{\mbox{$\raisebox{-0.59ex}{$l$}\hspace{-0.16em}\mu\hspace{-0.91em}\raisebox{-0.95ex}{\scalebox{2}{$\color{white}.$}}\hspace{-0.59em}\raisebox{+0.78ex}{\scalebox{2}{$\color{white}.$}}\hspace{0.46em}$}{}} 
\newcommand{\BOne} {{\mathchoice{\hbox{\rm1\kern-2.7pt l\kern.9pt}}
                              {\hbox{\rm1\kern-2.7pt l\kern.9pt}}
                              {\hbox{\scriptsize\rm1\kern-2.3pt l\kern.4pt}}
                              {\hbox{\scriptsize\rm1\kern-2.4pt l\kern.5pt}}}}

\newcommand{\BA}{{\mathbb{A}}}

\newcommand{\BD}{{\mathbb{D}}}

\newcommand{\BF}{{\mathbb{F}}}
\newcommand{\BG}{{\mathbb{G}}}

\newcommand{\BL}{{\mathbb{L}}}

\newcommand{\BN}{{\mathbb{N}}}
\newcommand{\BO}{{\mathbb{O}}}
\newcommand{\BP}{{\mathbb{P}}}

\newcommand{\BZ}{{\mathbb{Z}}}

\newcommand{\CA}{{\cal{A}}}

\newcommand{\CC}{{\cal{C}}}
\newcommand{\CalD}{{\cal{D}}}
\newcommand{\CE}{{\cal{E}}}
\newcommand{\CF}{{\cal{F}}}
\newcommand{\CG}{{\cal{G}}}
\newcommand{\CH}{{\cal{H}}}
\newcommand{\CI}{{\cal{I}}}
\newcommand{\CJ}{{\cal{J}}}

\newcommand{\CL}{{\cal{L}}}
\newcommand{\CM}{{\cal{M}}}
\newcommand{\CN}{{\cal{N}}}
\newcommand{\CO}{{\cal{O}}}

\newcommand{\CS}{{\cal{S}}}
\newcommand{\CT}{{\cal{T}}}
\newcommand{\CU}{{\cal{U}}}
\newcommand{\CV}{{\cal{V}}}
\newcommand{\CW}{{\cal{W}}}
\newcommand{\CX}{{\cal{X}}}
\newcommand{\CY}{{\cal{Y}}}
\newcommand{\CZ}{{\cal{Z}}}

\newcommand{\FF}{{\mathfrak{F}}}
\newcommand{\FG}{{\mathfrak{G}}}

\newcommand{\FM}{{\mathfrak{M}}}

\newcommand{\scrH}{{\mathscr{H}}}

\let\setminus\smallsetminus

\newcommand{\open}{^\circ}

\newcommand{\dual}{^{\SSC\lor}}

\newcommand{\mal}{^{\SSC\times}}

\newcommand{\ul}[1]{{\underline{#1}}}
\newcommand{\ol}[1]{{\overline{#1}}}
\newcommand{\wh}[1]{{\widehat{#1}}}
\newcommand{\wt}[1]{{\widetilde{#1}}}

\usepackage{ifthen}

\newcommand{\invlim}[1][]{\ifthenelse{\equal{#1}{}}
{\DS \lim_{\longleftarrow}}
{\DS \lim_{\underset{#1}{\longleftarrow}}}
}

\newcommand{\dirlim}[1][]{\ifthenelse{\equal{#1}{}}
{\DS \lim_{\longrightarrow}}
{\DS \lim_{\underset{#1}{\longrightarrow}}}
}

\newcommand{\dbl}{{\mathchoice{\mbox{\rm [\hspace{-0.15em}[}}
                              {\mbox{\rm [\hspace{-0.15em}[}}
                              {\mbox{\scriptsize\rm [\hspace{-0.15em}[}}
                              {\mbox{\tiny\rm [\hspace{-0.15em}[}}}}
\newcommand{\dbr}{{\mathchoice{\mbox{\rm ]\hspace{-0.15em}]}}
                              {\mbox{\rm ]\hspace{-0.15em}]}}
                              {\mbox{\scriptsize\rm ]\hspace{-0.15em}]}}
                              {\mbox{\tiny\rm ]\hspace{-0.15em}]}}}}
\newcommand{\dpl}{{\mathchoice{\mbox{\rm (\hspace{-0.15em}(}}
                              {\mbox{\rm (\hspace{-0.15em}(}}
                              {\mbox{\scriptsize\rm (\hspace{-0.15em}(}}
                              {\mbox{\tiny\rm (\hspace{-0.15em}(}}}}
\newcommand{\dpr}{{\mathchoice{\mbox{\rm )\hspace{-0.15em})}}
                              {\mbox{\rm )\hspace{-0.15em})}}
                              {\mbox{\scriptsize\rm )\hspace{-0.15em})}}
                              {\mbox{\tiny\rm )\hspace{-0.15em})}}}}

\newcommand{\dotBD}{\vbox{\hbox{\kern2pt\bf.}\vskip-4.5pt\hbox{$\BD$}}}

\DeclareMathOperator{\QIsog}{QIsog}
\DeclareMathOperator{\Nilp}{\CN \!{\it ilp}}

\DeclareMathOperator{\Sets}{\CS \!{\it ets}}
\DeclareMathOperator{\Gr}{\CG r}
\def\ulI{{{\underline{I\!}\,}{}}}

\def\ulHZ{{\underline{\hat Z\!}\,}{}}

\def\olB{{\,\overline{\!B}}}
\def\olT{{\overline{T}}}

\def\s{\sigma^\ast}

\def\longto{\longrightarrow}
\def\into{\hookrightarrow}
\def\onto{\mbox{$\kern2pt\to\kern-8pt\to\kern2pt$}}
\def\isoto{\stackrel{}{\mbox{\hspace{1mm}\raisebox{+1.4mm}{$\SC\sim$}\hspace{-3.5mm}$\longrightarrow$}}}
\def\longinto{\lhook\joinrel\longrightarrow}

\newbox\mybox
\def\arrover#1{\mathrel{
       \setbox\mybox=\hbox spread 1.4em{\hfil$\scriptstyle#1$\hfil}
       \vbox{\offinterlineskip\copy\mybox
             \hbox to\wd\mybox{\rightarrowfill}}}}

\newcommand{\ppsi}{\delta}

\newcommand{\RZ}{\CM_{\ul\BL}^{\hat{Z}}}

\newcommand{\BaseOfD}{\BF}
\newcommand{\BaseFldOfLocSht}{k}

\newcommand{\BaseFldInSectUnif}{k}
\newcommand{\AlgClFld}{k}
\newcommand{\ArbitraryFld}{k}
\newcommand{\genericG}{P}

\DeclareMathOperator{\Sht}{Sht}
\newcommand{\Vect}{V\!ect}
\DeclareMathOperator{\SpaceFl}{\CF\ell}
\newcommand{\tauGlob}{\tau}
\newcommand{\tauLoc}{\hat\tau}
\newcommand{\charsect}{s}

\DeclareMathOperator{\AbSh}{\CA {\it b}-\CS {\it h}}

\newcommand{\CorEtIsTrivial}{Corollary~2.9}
\newcommand{\PropRigidityLocal}{Proposition~2.11}

\newcommand{\ThmModuliSpX}{Theorem~4.4}
\newcommand{\EqDecency}{Equation~(4.4)}

\newcommand{\RemBdIsFormSch}{Remark~4.10}

\newcommand{\RemDecent}{Remark~4.15}
\newcommand{\RemJb}{Remark~4.16}
\newcommand{\ThmRRZSp}{Theorem~4.18}
\newcommand{\CorQC}{Corollary~4.26}
\newcommand{\quotisfstack}{Proposition~4.27}

\newcommand{\PropTateEquiv}{Proposition~3.4}
\newcommand{\PropTateEquivP}{Proposition~3.6}

\newcommand{\UnboundedUnif}{Chapter~5}
\newcommand{\SectGlobalLocalFunctor}{Section~5.2}
\newcommand{\DefGlobalLocalFunctor}{Definition~5.4}
\newcommand{\RemTateF}{Remark~5.6}
\newcommand{\PropLocalIsogeny}{Proposition~5.7}
\newcommand{\RemLocalIsogeny}{Remark~5.8}
\newcommand{\PropGlobalRigidity}{Proposition~5.9}

\begin{document}

\author{Esmail Arasteh Rad and Urs Hartl\footnote{Both authors acknowledge support by the Deutsche Forschungsgemeinschaft (DFG) in form of the research grant HA3002/2-1, the SFB's 478 and 878, and Germany's Excellence Strategy EXC 2044 –390685587, Mathematics M\"unster: Dynamics--Geometry--Structure.}}

\date{\today}

\title{Uniformizing The Moduli Stacks of Global $\FG$-Shtukas\\\comment{Do spellcheck!}}

\maketitle

\begin{abstract}
\noindent
This is the second in a sequence of articles, in which we explore moduli stacks of global $\FG$-shtukas, the function field analogs for Shimura varieties. Here $\FG$ is a flat affine group scheme of finite type over a smooth projective curve $C$ over a finite field. Global $\FG$-shtukas are generalizations of Drinfeld shtukas and analogs of abelian varieties with additional structure. We prove that the moduli stacks of global $\FG$-shtukas are algebraic Deligne-Mumford stacks separated and locally of finite type. They generalize various moduli spaces used by different authors to prove instances of the Langlands program over function fields. In the first article we explained the relation between global $\FG$-shtukas and local $\BP$-shtukas, which are the function field analogs of $p$-divisible groups. Here $\BP$ is the base change of $\FG$ to the complete local ring at a point of $C$. When $\BP$ is smooth with connected reductive generic fiber we proved the existence of Rapoport-Zink spaces for local $\BP$-shtukas. In the present article we use these spaces to (partly) uniformize the moduli stacks of global $\FG$-shtukas for smooth $\FG$ with connected fibers and reductive generic fiber. This is our main result. It has applications to the analog of the Langlands-Rapoport conjecture for our moduli stacks.

\end{abstract}

\tableofcontents

%
%

\thispagestyle{empty}

\bigskip
\section{Introduction}
\setcounter{equation}{0}

Since Drinfeld's celebrated work \cite{Drinfeld, Drinfeld77, Drinfeld1, Drinfeld89} it is clear that moduli spaces of shtukas are the function field analog of Shimura varieties and of similar importance in arithmetic algebraic geometry. To give their definition let $\BF_q$ be a finite field with $q$ elements, let $C$ be a smooth projective geometrically irreducible curve over $\BF_q$, and let $\FG$ be a flat affine group scheme of finite type over $C$. Fix an integer $n>0$. A \emph{global $\FG$-shtuka} $\ul\CG$ over an $\BF_q$-scheme $S$ is a tuple $(\CG,\charsect_1,\ldots,\charsect_n,\tauGlob)$ consisting of a $\FG$-torsor $\CG$ over $C_S:=C\times_{\BF_q}S$, an $n$-tuple of (characteristic) sections $(\charsect_1,\ldots,\charsect_n)\in C^n(S)$ called \emph{legs}, and a Frobenius connection $\tauGlob$ defined outside the graphs $\Gamma_{\charsect_i}$ of the legs $\charsect_i$, that is, an isomorphism $\tauGlob\colon\s \CG|_{C_S\setminus \cup_i \Gamma_{\charsect_i}}\isoto \CG|_{C_S\setminus \cup_i \Gamma_{\charsect_i}}$ where $\s=(\id_C \times \Frob_{q,S})^\ast$. More generally, we also consider iterated global $\FG$-shtukas; see Definition~\ref{Global Sht}.

In Theorem~\ref{nHisArtin} of this article we show that the moduli stack $\nabla_n^{\ul\omega}\scrH^1_D(C,\FG)$ of global $\FG$-shtukas, after imposing suitable boundedness conditions (by $\ul\omega$) and $D$-level structures, is a Deligne-Mumford stack separated and locally of finite type over $(C\setminus D)^n$. (See the end of this introduction for an explanation of our notation for the moduli stack of global $\FG$-shtukas.) Dropping the boundedness condition by $\ul\omega$ we obtain $\nabla_n\scrH^1_D(C,\FG)=\dirlim\nabla_n^{\ul\omega}\scrH^1_D(C,\FG)$ as an ind-algebraic Deligne-Mumford stack; see Definition~\ref{DefIndAlgStack}. Spelling out the Riemann-Hilbert correspondence for function fields, together with the Tannakian philosophy, one sees that $\nabla_n^{\ul\omega}\scrH^1_D(C,\FG)$ may play the same role that Shimura varieties play for number fields. More specifically one can hope that the Langlands correspondence for function fields is realized on its cohomology. On the other hand this analogy can be viewed as an attempt to build a bridge between the geometric Langlands program and the arithmetic Langlands program, where the role of global $\FG$-shtukas is played by abelian varieties (together with additional structures), respectively $D$-modules. 

Note that our moduli stack $\nabla_n^{\ul\omega}\scrH^1_D(C,\FG)$ generalizes the space of $F$-sheaves $FSh_{D,r}$ which was considered by Drinfeld~\cite{Drinfeld1} and Lafforgue~\cite{LafforgueL02} in their proof of the Langlands correspondence for $\FG=\GL_2$ (resp.\ $\FG=\GL_r$), and its generalization $FBun$ by Varshavsky~\cite{Var}. It likewise generalizes the moduli stacks $\CE\ell\ell_{C,\mathscr{D},I}$ of Laumon, Rapoport and Stuhler~\cite{LRS}, their generalizations by L.~Lafforgue~\cite{LafforgueL97}, Ng\^o~\cite{Ngo06} and Spie{\ss}~\cite{Spiess10}, the spaces $\text{Cht}_{\ul\lambda}$ of Ng\^o and Ng\^o Dac~\cite{NgoNgo,NgoDac13}, and the spaces $\AbSh^{r,d}_H$ of the second author~\cite{Har1}; see Remark~\ref{RemDrinfeldVarshavsky}. Varshavsky's stacks $FBun$ were recently used by V.~Lafforgue~\cite{Lafforgue12} to prove Langlands parameterization for split groups over function fields. Lafforgue~\cite[\S\,12.3.2]{Lafforgue12} also indicates the generalization to non-split groups and a corresponding generalization of Varshavsky's moduli stacks of shtukas. Our Theorem~\ref{nHisArtin} gives a way to justify some of his assertions; see Remark~\ref{RemDrinfeldVarshavsky}. The proof of Theorem~\ref{nHisArtin} is similar to Varshavsky's algebraicity result of his stack $FBun$. But beyond this, our focus in the present article is quite different from Varshavsky's.

Namely, our approach to study the moduli stack of global $\FG$-shtukas is to relate it to certain moduli spaces for local objects, called local $\BP$-shtukas. More precisely, let $A_\nu\cong\BF_\nu\dbl\zeta\dbr$ be the completion of the local ring $\CO_{C,\nu}$ at a closed point $\nu\in C$, let $Q_\nu$ be its fraction field, and let $\BP=\BP_\nu:=\FG\times_C \Spec A_\nu$ and $\genericG_\nu=\FG\times_C\Spec Q_\nu$. A \emph{local $\BP_\nu$-shtuka} over a scheme $S\in\Nilp_{A_\nu}$ is a pair $\ul\CL = (\CL_+,\tauLoc)$ consisting of an $L^+\BP_\nu$-torsor $\CL_+$ on $S$ and an isomorphism of the $L\genericG_\nu$-torsors $\tauLoc\colon  \hat{\sigma}^\ast \CL \isoto\CL$. Here $L\genericG_\nu$ (resp.\ $L^+\BP_\nu$) denotes the group of loops (resp.\ positive loops) of $\BP_\nu$ (see Chapter~\ref{LoopsAndSht}), $\CL$ denotes the $L\genericG_\nu$-torsor associated with $\CL_+$, and $\hat{\sigma}^\ast \CL$ denotes the pullback of $\CL$ under the absolute $\BF_\nu$-Frobenius endomorphism $\Frob_{(\#\BF_\nu),S} \colon  S \to S$. Moreover, $\Nilp_{A_\nu}$ denotes the category of $A_\nu$-schemes on which the uniformizer $\zeta$ of $A_\nu$ is locally nilpotent. Local $\BP_\nu$-shtukas can be viewed as function field analogs of $p$-divisible groups; see also \cite{H-V,HV2}.

If $\BP_\nu$ is smooth over $A_\nu$ we proved in \cite[\ThmModuliSpX]{AH_Local} for a fixed local $\BP_\nu$-shtuka $\ul\BL$ over a field $k\in\Nilp_{A_\nu}$ that the Rapoport-Zink functor
\begin{eqnarray*}
\CM_{\ul\BL}\colon  \Nilp_{k\dbl\zeta\dbr} &\longto&  \Sets\\
S &\longmapsto & \big\{\text{Isomorphism classes of }(\ul{\CL},\bar{\delta})\colon\;\text{where }\ul{\CL}~\text{is a local $\BP_\nu$-shtuka}\\ 
&&~~~~ \text{over $S$ and }\bar{\delta}\colon  \ul{\CL}_{\bar{S}}\to \ul\BL_{\bar{S}}~\text{is a quasi-isogeny  over $\bar{S}$}\big\},
\end{eqnarray*}
where $\bar S=\Var(\zeta)\subset S$, is representable by an ind-scheme, ind-quasi-projective over $\Spf k\dbl\zeta\dbr$; see also Theorem~\ref{ThmModuliSpX}. To obtain a formal scheme locally formally of finite type, as in the analog for $p$-divisible groups, one has to assume that the generic fiber $P_\nu$ is connected reductive, and one has to impose a \emph{bound $\hat Z$ on the Hodge polygon}, that is on the relative position of $\hat\sigma^*\CL_+$ and $\CL_+$ under $\tauLoc$; see Definition~\ref{DefBDLocal}. By \cite[\ThmRRZSp]{AH_Local} the bounded Rapoport-Zink functor $\RZ$ is representable by a formal scheme locally formally of finite type over $\Spf R_{\hat Z}$, where $R_{\hat Z}$ is the reflex ring of the bound $\hat{Z}$; see Definition~\ref{DefBDLocal} and Theorem~\ref{ThmRRZSp}.

Consider a fixed $n$-tuple of pairwise different characteristic places $\ul\nu=(\nu_1,\ldots,\nu_n)$ of $C$ and the formal stack $\nabla_n\scrH^1(C,\FG)^{\ul\nu}$, which is obtained by taking the formal completion of the stack $\nabla_n\scrH^1(C,\FG)$ at $\ul\nu\in C^n$. This means we let $A_{\ul\nu}$ be the completion of the local ring $\CO_{C^n,\ul\nu}$, and we consider global $\FG$-shtukas only over schemes $S$ whose characteristic morphism $S\to C^n$ factors through $\Nilp_{A_\ul\nu}$. Recall that with an abelian variety over a scheme in $\Nilp_{\BZ_p}$ one can associate its $p$-divisible group. In the analogous situation for global $\FG$-shtukas we associated in \cite[\DefGlobalLocalFunctor]{AH_Local} a tuple $(\wh\Gamma_{\nu_i}(\ul\CG))_i$ of local $\BP_{\nu_i}$-shtukas $\wh\Gamma_{\nu_i}(\ul\CG)$ with a global $\FG$-shtuka $\ul\CG$ in $\nabla_n\scrH^1(C,\FG)^{\ul\nu}(S)$. The relation between global $\FG$-shtukas and local $\BP$-shtukas was thoroughly explained in \cite[\UnboundedUnif]{AH_Local}.

As was pointed out in \cite{H-V} the true analogs of $p$-divisible groups are bounded local $\BP$-shtukas. Nevertheless in Chapter~\ref{UnboundedUnif} we prove that the product $\prod_i\CM_{\wh\Gamma_{\nu_i}(\ul\CG)}$ can be regarded as a uniformization space for $\nabla_n\scrH^1(C,\FG)^{\ul\nu}$ already in the unbounded situation. Namely, in Theorem~\ref{Uniformization1} we construct the uniformization morphism $\prod_i\CM_{\wh\Gamma_{\nu_i}(\ul\CG)}\,\to\,\nabla_n\scrH^1(C,\FG)^{\ul\nu}$ and show that it is ind-proper and formally \'etale. However, it will gain its full strength and provide a true (partial) uniformization only after we fix an $n$-tuple $\ulHZ=(\hat Z_i)_i$ of bounds and consider global $\FG$-shtukas $\ul\CG$ whose associated local $\BP_{\nu_i}$-shtuka $\wh\Gamma_{\nu_i}(\ul\CG)$ is bounded by $\hat Z_i$ for all $i=1,\ldots,n$. We let $R_{\hat Z_i}=\kappa_i\dbl\xi_i\dbr$ be the reflex ring of the bound $\hat Z_i$, and we set $R_{\ul{\hat Z}}:=\kappa\dbl\xi_1,\ldots,\xi_n\dbr$, where $\kappa$ is the compositum of all the $\kappa_i$ in an algebraic closure of $\BF_q$. We also introduce the notion of a \emph{rational $H$-level structure} for a compact open subgroup $H\subset \FG (\BA^{\ul\nu})$ on a global $\FG$-shtuka $\ul\CG$ using the Tannakian theory of Tate modules in Chapter~\ref{GalRepSht}. Here $\BA^{\ul\nu}$ is the ring of adeles of $C$ outside $\ul\nu$ and we assume that $\FG$ is smooth over $C$ with connected fibers. We denote by $\nabla_n^H\scrH^1(C,\FG)^{\ul\nu}$ the stack of global $\FG$-shtukas with $H$-level structure fibered in groupoids over $\Nilp_{A_{\ul\nu}}$, and by $\nabla_n^{H,\ulHZ}\scrH^1(C,\FG)^{\ul\nu}$ the closed ind-substack of $\nabla_n^H\scrH^1(C,\FG)^{\ul\nu}\whtimes_{A_{\ul\nu}}\Spf R_{\ul{\hat Z}}$ consisting of global $\FG$-shtukas bounded by $\ulHZ$. Both stacks are ind-algebraic Deligne-Mumford stacks, ind-separated and locally of ind-finite type; see Theorem~\ref{ThmLSGGsht1} and Remark~\ref{Rem5.2}. To describe our main uniformization result let $\ul\CG_0$ be a fixed global $\FG$-shtuka in $\nabla_n^{H,\ulHZ}\scrH^1(C,\FG)^{\ul\nu}(\BaseFldInSectUnif)$ where the field $\BaseFldInSectUnif\in\Nilp_{R_{\ul{\hat Z}}}$ is an algebraic closure of $\BF_q$. Let $(\ul\BL_i)_i:=\wh{\ul\Gamma}(\CG_0)$ be the associated $n$-tuple of local $\BP_{\nu_i}$-shtukas, and let $I_{\ul\CG_0}\!(Q)$ denote the group $\QIsog_\BaseFldInSectUnif(\ul\CG_0)$ of quasi-isogenies of $\ul\CG_0$; see Definition~\ref{quasi-isogenyGlob}. Let $\breve\CM_{\ul\BL_i}^{\hat{Z}_i}$ denote the base change of $\CM_{\ul\BL_i}^{\hat{Z}_i}$ to $\breve R_{\hat Z_i}:=\BaseFldInSectUnif\dbl\xi\dbr$, and set $\breve R_{\ul{\hat Z}}:=\BaseFldInSectUnif\dbl\xi_1,\ldots,\xi_n\dbr$. If $\FG$ is smooth over $C$ with connected fibers and reductive generic fiber, we construct in Theorem~\ref{Uniformization2} the uniformization morphism
\begin{equation}\label{EqUnifIntro}
\Theta\colon  I_{\ul\CG_0}\!(Q) \big{\backslash}\bigl(\prod_i \breve\CM_{\ul\BL_i}^{\hat{Z}_i}\times \Isom^{\otimes}(\omega^\circ,\check{\CV}_{\ul{\CG}_0})/H\bigr) \;\longto\; \nabla_n^{H,\ulHZ}\scrH^1(C,\FG)^{\ul\nu}\whtimes_{R_{\ulHZ}} \Spf\breve R_{\ulHZ}
\end{equation}
and in addition we prove that it induces an isomorphism after passing to the formal completion along its image. Note that the reduced subscheme of the Rapoport-Zink space for local $\BP$-shtukas is an affine Deligne-Lusztig variety; see Theorem~\ref{ThmRRZSp}. Thus as a consequence of the uniformization theorem one can relate the rational points in the quasi-isogeny class of $\ul\CG_0$ inside the moduli stack of global $\FG$-shtukas to the rational points of certain affine Deligne-Lusztig varieties. This has applications to the analog of the Langlands-Rapoport conjecture for $\nabla_n^{H,\ulHZ}\scrH^1(C,\FG)^{\ul\nu}$ which we discuss in a future article. The latter analog aims to give a description of the points of $\nabla_n^{H,\ulHZ}\scrH^1(C,\FG)^{\ul\nu}$ with values in finite fields, which is functorial in all the data, including the curve $C$. The general functoriality properties in dependence of $(C,\FG,H,\ulHZ)$ are studied by Breutmann~\cite{BreutmannFunctoriality}. For this reason we like to carry all the data along in the notation. Moreover, for the classifying space of $\FG$-torsors we prefer Behrend's~\cite{Beh} notation $\scrH^1(C,\FG)$ instead of the symbol $\text{Bun}_\FG$. Our notation $\nabla_n\scrH^1(C,\FG)$ (instead of $\FG\text{-Sht}$) for the space of global $\FG$-shtukas indicates that $\FG$-shtukas consist of a $\FG$-torsor together with a Frobenius connection. It is reminiscent of Drinfeld's original inspiration by the Riemann-Hilbert correspondence, and the relation to the geometric Langlands program. 

Let us add that similar uniformization results were previously obtained by Drinfeld~\cite{Drinfeld2} who uniformized the moduli spaces of Drinfeld modules at $\infty$, by the second author~\cite{Har1} who uniformized the moduli stacks $\AbSh^{r,d}_H$ of abelian $\tau$-sheaves at $\infty$, and by Blum and Stuhler~\cite{Blum-Stuhler} and Hausberger~\cite{Hausberger} who uniformized the moduli spaces $\CE\ell\ell_{C,\mathscr{D},I}$ of $\mathscr{D}$-elliptic sheaves from \cite{LRS} at $\infty$, respectively at finite places. Spie{\ss}~\cite{Spiess10} related the latter two uniformizations to each other. The analogous situation for Shimura varieties was developed by \v{C}erednik~\cite{Cerednik}, Drinfeld~\cite{Drinfeld2}, Rapoport and Zink~\cite{RZ}, Boutot and Zink~\cite{BZ95}, Varshavsky~\cite{Varshavsky98}, and Kudla and Rapoport~\cite{KudlaRapoport}.

\bigskip\noindent
{\bfseries\large Acknowledgements.}
We are grateful to Alain Genestier who suggested to the second author already many years ago to investigate the uniformization Theorem~\ref{Uniformization2} for smooth group schemes. We also thank Vincent Lafforgue for important remarks and his explanations on his current application of moduli spaces of $\FG$-shtukas to Langlands parameterization. Moreover, we thank Jochen Heinloth, Linus Kramer, Tu\^an Ng\^o Dac, Timo Richarz and Eva Viehmann for helpful discussions, and the referee for his careful reading and many valuable remarks.

\subsection*{Notation and Conventions}\label{Notation and Conventions}
Throughout this article we denote by
\begin{tabbing}
$\genericG_\nu:=\FG\times_C\Spec Q_\nu,$\; \=\kill
$\BF_q$\> a finite field with $q$ elements,\\[1mm]
$C$\> a smooth projective geometrically irreducible curve over $\BF_q$,\\[1mm]
$Q:=\BF_q(C)$\> the function field of $C$,\\[1mm]
$\nu$\> a closed point of $C$, also called a \emph{place} of $C$,\\[1mm]
$\BF_\nu$\> the residue field at the place $\nu$ on $C$,\\[1mm]
$A_\nu$\> the completion of the stalk $\CO_{C,\nu}$ at  $\nu$,\\[1mm]
$Q_\nu:=\Quot(A_\nu)$\> its fraction field,\\[1mm]
$n\in\BN_{>0}$\> a positive integer,\\[1mm]
$\ul\nu:=(\nu_i)_{i=1\ldots n}$\> an $n$-tuple of closed points of $C$,\\[1mm]
$\BO^{\ul\nu}:=\prod_{\nu\notin\ul\nu}A_\nu$\> \parbox[t]{0.775\textwidth}{the ring of integral adeles of $C$ outside $\ul\nu$,}\\[1mm]
$\BA^{\ul\nu}:=\BO^{\ul\nu}\otimes_{\CO_C}Q$\> \parbox[t]{0.775\textwidth}{the ring of adeles of $C$ outside $\ul\nu$,}\\[1mm]
$\BaseOfD$ \> a finite field containing $\BF_q$,\\[1mm]
$\BD_R:=\Spec R\dbl z \dbr$ \> \parbox[t]{0.775\textwidth}{the spectrum of the ring of formal power series in $z$ with coefficients in an $\BaseOfD$-algebra $R$,
}\\[2mm]
When $R= \BaseOfD$ we drop the subscript $R$ from the notation of $\BD_R$.
\end{tabbing}

\noindent
For a formal scheme $\wh S$ we denote by $\Nilp_{\wh S}$ the category of schemes over $\wh S$ on which an ideal of definition of $\wh S$ is locally nilpotent. We  equip $\Nilp_{\wh S}$ with the \'etale topology. We also denote by
\begin{tabbing}
$\genericG_\nu:=\FG\times_C\Spec Q_\nu,$\; \=\kill
$A_{\ul\nu}$\> the completion of the local ring $\CO_{C^n,\ul\nu}$ of $C^n$ at the closed point $\ul\nu=(\nu_i)$,\\[1mm]
$\Nilp_{A_{\ul\nu}}:=\Nilp_{\Spf A_{\ul\nu}}$\> \parbox[t]{0.775\textwidth}{the category of schemes over $C^n$ on which the ideal defining the closed point $\ul\nu\in C^n$ is locally nilpotent,}\\[2mm]
$\Nilp_{\BaseOfD\dbl\zeta\dbr}$\> \parbox[t]{0.775\textwidth}{the category of $\BD$-schemes $S$ for which the image of $z$ in $\CO_S$ is locally nilpotent. We denote the image of $z$ by $\zeta$ since we need to distinguish it from $z\in\CO_\BD$.}\\[2mm]
$\FG$\> a flat affine group scheme of finite type over $C$, \\
$\BP_\nu:=\FG\times_C\Spec A_\nu$ \> the base change of $\FG$ to $\Spec A_\nu$,\\[1mm]
$\genericG_\nu:=\FG\times_C\Spec Q_\nu$ \> the generic fiber of $\BP_\nu$ over $\Spec Q_\nu$,\\[1mm]
$\BP$\> a smooth affine group scheme of finite type over $\BD=\Spec\BaseOfD\dbl z\dbr$,\\[1mm] 
$\genericG:=\BP\times_{\BD}\Spec\BaseOfD\dpl z\dpr$\> the generic fiber of $\BP$ over $\Spec\BaseOfD\dpl z\dpr$.
\end{tabbing}

\noindent
Let $S$ be an $\BF_q$-scheme. We denote by $\sigma_S \colon  S \to S$ its $\BF_q$-Frobenius endomorphism which acts as the identity on the points of $S$ and as the $q$-power map on the structure sheaf. Likewise we let $\hat{\sigma}_S\colon S\to S$ be the $\BaseOfD$-Frobenius endomorphism of an $\BaseOfD$-scheme $S$. We set
\begin{tabbing}
$\genericG_\nu:=\FG\times_C\Spec Q_\nu,$\; \=\kill
$C_S := C \times_{\Spec\BF_q} S$, \> and \\[1mm]
$\sigma := \id_C \times \sigma_S$.
\end{tabbing}

Let $H$ be a sheaf of groups (for the \fppf-topology) on a scheme $X$. In this article a (\emph{right}) \emph{$H$-torsor} (also called an \emph{$H$-bundle}) on $X$ is a sheaf $\CG$ for the \fppf-topology on $X$ together with a (right) action of the sheaf $H$ such that $\CG$ is isomorphic to $H$ over an \fppf-covering of $X$. Here $H$ is viewed as an $H$-torsor by right multiplication.

%
%

\section{$\FG$-Bundles}\label{SectGBundles}
\setcounter{equation}{0}

Let $\BF_q$ be a finite field with $q$ elements, let $C$ be a smooth projective geometrically irreducible curve over $\BF_q$, and let $\FG$ be a flat affine group scheme of finite type over $C$. From Chapter~\ref{LoopsAndSht} on we are mainly interested in the situation where $\FG$ is a smooth group scheme. Furthermore, we assume in Chapter~\ref{GalRepSht} that all fibers of $\FG$ are connected (and smooth), and in Chapter~\ref{Uniformization Theorem} that its generic fiber is reductive. In the following we use the term \emph{vector bundle} for a locally free sheaf of finite rank. To construct useful representations of $\FG$ on vector bundles over $C$ we consider the following condition on a pair $(H,G)$ of group schemes of finite type over $C$ such that $H\into G$ is a closed subgroup scheme and flat over $C$.
\begin{eqnarray}\label{EqCondQuotient}
\text{There is a scheme $Y$ affine and of finite type over $C$ with an action}\\\text{$G\times_CY\to Y$ of $G$ and a $G$-equivariant open immersion $G/H\into Y$.}\nonumber
\end{eqnarray}
If $(H,G)$ satisfies \eqref{EqCondQuotient} then the quotient $G/H$ is a scheme quasi-affine and of finite type over $C$. We do not know whether the converse is true. Note that for a quasi-affine scheme $X$ of finite type over a field $k$ it is in general false that $\Gamma(X,\CO_{X})$ is a finitely generated $k$-algebra; see the discussion on \href{http://mathoverflow.net/questions/209443/is-the-affine-closure-of-a-quasi-affine-variety-again-a-variety}{http:/\!/mathoverflow.net/questions/209443/}. Therefore we do not know whether we can take $Y=\ul\Spec_C\,\CO_{G/H}$ in condition \eqref{EqCondQuotient}. For this reason we need the following

\begin{proposition}\label{PropQAffineQuot}
Let $H\into G$ be a closed immersion of group schemes of finite type over $C$ and let $H$ be flat over $C$. 
\begin{enumerate}
\item \label{PropQAffineQuot_A}
The quotient $G/H$ is representable by a scheme separated and of finite type over $C$. It is flat (respectively smooth) over $C$ if and only if $G$ is. 
\item \label{PropQAffineQuot_B}
Let $G'\into G$ be a closed immersion of group schemes such that $G'$ is flat over $C$ and $H\into G$ factors through $G'$. If the schemes $G/G'$ and $G'/H$ are (quasi-)affine over $C$ then also $G/H$ is (quasi-)affine over $C$. Moreover, if $(H,G')$ satisfies \eqref{EqCondQuotient} and $G/G'$ is affine over $C$ then also $(H,G)$ satisfies \eqref{EqCondQuotient}.
\item \label{PropQAffineQuot_D}
Let $G$ be affine and smooth over $C$ such that all its geometric fibers are reductive groups and $G/G\open\to C$ is finite, where $G\open\subset G$ is the open normal subgroup scheme whose fiber $(G\open)_x$ over any $x\in C$ is the connected component of $G_x$; see \cite[VI$_{\text{B}}$, Th\'eor\`eme~3.10]{SGA3}. Let $H$ be smooth over $C$. Then $G/H$ is affine over $C$ if and only if all geometric fibers of $H$ are reductive and $H/H\open\to C$ is finite.
\end{enumerate}
\end{proposition}

\begin{proof}
\ref{PropQAffineQuot_A} was proved in \cite[Th\'eor\`eme~4.C]{Anantharaman73} and \cite[VI$_{\rm B}$, Proposition~9.2(x,xi,xii)]{SGA3}. 

\medskip\noindent
\ref{PropQAffineQuot_B} The property of being \mbox{(quasi-)}affine is \fpqc-local on the base by \cite[IV$_2$, Proposition~2.7.1]{EGA}. We apply the base change $G\to G/G'$ which is faithfully flat by \cite[VI$_{\rm B}$, Proposition~9.2(xi)]{SGA3} to the morphism $G/H \to G/G'$ to obtain $(G/H)\times_{G/G'}G \to G$. This morphism equals the projection $(G'/H)\times_C G\to G$ under the isomorphism $\alpha\colon(G/H)\times_{G/G'}G\isoto (G'/H)\times_C G$, $(g_1H,g_2)\mapsto (g_2^{-1}g_1H,g_2)$ with inverse $(g'H,g_2)\mapsto(g_2g'H,g_2)$. Therefore $G/H \to G/G'$ is \mbox{(quasi-)}affine and by \cite[II, Proposition~5.1.10(ii)]{EGA} this proves that $G/H$ is (quasi-)affine over $C$.

Let moreover $Y'$ be a scheme affine and of finite type over $C$ with an action $G'\times_CY'\to Y'$ of $G'$ and a $G'$-equivariant open immersion $\iota\colon G'/H\into Y'$. Set $U:=G/G'$ and let $p\colon U':=G\onto G/G'=U$ be the quotient morphism. Then $G\times_CG'\isoto U'\times_UU'=:U''$ under the map $(g_1,g')\mapsto(g_1,g_1g')$ and $U''':=U'\times_UU'\times_UU'\isoto G\times_CG'\times_CG'$. Let $p_i\colon U''\to U'$, respectively $p_{ij}\colon U'''\to U''$ be the projections onto the $i$-th, respectively $i$-th and $j$-th component and set $q:=p_1\circ p=p_2\circ p$. That is $p_1(g_1,g')=g_1$ and $p_2(g_1,g')=g_1g'$. We consider the scheme $Y'\times_CG$ which is affine and of finite type over $G=U'$ with the induced open immersion
\[
\iota'\circ\alpha\colon\; p^*(G/H)\;:=\;(G/H)\times_{G/G'}G\;\isoto\; (G'/H)\times_C G\;\longinto\; Y'\times_CG\,, 
\]
where $\iota':=\iota\times\id_G$. Consider the $U''$-isomorphism
\begin{eqnarray}\label{EqDescendPhi}
\phi\colon\; p_2^*(Y'\times_CG) = (Y'\times_CG)\times_{G,p_2}U'' & \isoto & (Y'\times_CG)\times_{G,p_1}U'' = p_1^*(Y'\times_CG) \,,\nonumber\\[1mm]
(y',g_2,g_1,g_2) & \longmapsto & \bigl((g_1^{-1}g_2).\:\!y',g_1,g_1,g_2\bigr)\,. \nonumber
\end{eqnarray}
It satisfies the cocycle condition $p_{13}^*\phi=p_{12}^*\phi\circ p_{23}^*\phi$ over $U'''$. Therefore $Y'\times_CG$ descends to a $U$-scheme $Y\to U=G/G'$ with a $U'$-isomorphism $\gamma\colon Y\times_UU'\isoto Y'\times_CG$ satisfying $p_1^*\gamma=\phi\circ p_2^*\gamma$ by \cite[\S\,6.1, Theorem~6]{BLR}. Moreover, $Y$ is affine and of finite type over $G/G'$ by \cite[IV$_2$, Proposition~2.7.1]{EGA}. In particular, $Y\to C$ is affine and of finite type because $G/G'\to C$ is assumed to be affine. Since $\phi$ satisfies $p_1^*(\iota'\alpha)=\phi\circ p_2^*(\iota'\alpha)\colon\,q^*(G/H)\to Y'\times_CG\times_CG'$ the inclusion $\iota'\alpha$ descends to a morphism $G/H\to Y$ by \cite[\S\,6.1, Theorem~6]{BLR} which is an open immersion by \cite[IV$_2$, Proposition~2.7.1]{EGA}. 

To define a $G$-action on $Y$ we consider $G\times_C G/G'$ as a scheme over $U=G/G'$ via the morphism $\bar\mu\colon(g_1,g_2G')\mapsto g_1g_2G'$. In this way the $G\times_C G/G'$-scheme $G\times_CY$ becomes a $U$-scheme and we want to construct a $U$-morphism $G\times_CY\to Y$. We construct it over $U'$ as the composition in the top row of the following diagram
\[
\xymatrix  @R=0.4pc @C+0.6pc {
(g_1,y,g_2) \ar@{|->}[r] & (g_1,y,g_1^{-1}g_2) \ar@{|->}[r] & (g_1,y',g_1^{-1}g_2) \ar@{|->}[r] & (y',g_2)
\\
(G\underset{C}{\times}Y)\underset{\bar\mu,U}{\times}U' \ar[r]_\sim \ar[dddd]_{pr_{U'}} & G\underset{C}{\times}(Y\underset{U}{\times}U') \ar[r]_\sim^{\id_G\times\gamma} \ar[dddd]_{\mu\circ(pr_1,pr_3)} & G\underset{C}{\times}Y'\underset{C}{\times}G \ar[r]^{\quad\id_{Y'}\times\mu} \ar[dddd]_{\mu\circ(pr_1,pr_3)} & Y'\underset{C}{\times}G \ar[dddd]_{pr_G} \ar[r]_\sim^{\gamma^{-1}} & Y\underset{U}{\times}U' \ar[dddd]_{pr_{U'}} \\
\\
\\
\\
U' \ar@{=}[r] & G \ar@{=}[r] & G \ar@{=}[r] & G \ar@{=}[r] & U'.\!\!}
\]
One checks that this morphism is compatible with the descent datum $\phi$ over $U''$, because it is given by the action of $G$ on $Y'\times_CG$ on the second factor, whereas $\phi$ is given by the action of $G'$ on the first factor. By \cite[\S\,6.1, Theorem~6]{BLR} we obtain an action $G\times_CY\to Y$ of $G$ on $Y$ and one also checks over $U'$ that the inclusion $G/H\to Y$ is $G$-equivariant. Therefore $(H,G)$ satisfies condition \eqref{EqCondQuotient}.

\medskip\noindent
\ref{PropQAffineQuot_D} was proved by Alper~\cite[Theorems~9.4.1 and 9.7.6]{Alper14}.
\end{proof}

\begin{proposition}\label{PropGlobalRep}
Let $\FG$ be a flat affine group scheme of finite type over $C$.
\begin{enumerate}
\item \label{PropGlobalRep_A}
There is a faithful representation $\FG\into\GL(\CV)$ for a vector bundle $\CV$ on $C$ such that the quotient $\GL(\CV)/\FG$ is a quasi-affine scheme over $C$. Moreover, the pair $\bigl(\FG,\GL(\CV)\bigr)$ satisfies \eqref{EqCondQuotient}.
\item \label{PropGlobalRep_B}
There is a faithful representation $\rho\colon\FG\hookrightarrow\GL(\CV)$ for a vector bundle $\CV$ on $C$ together with an isomorphism $\alpha\colon\wedge^\topol\CV\isoto\CO_C$ such that $\rho$ factors through $\SL(\CV):=\ker\bigl(\det\colon\!\!\GL(\CV)\to\GL(\wedge^\topol\CV)\bigr)$ and the quotients $\SL(\CV)/\FG$ and $\GL(\CV)/\FG$ are quasi-affine schemes over $C$. Moreover, the pairs $\bigl(\FG,\SL(\CV)\bigr)$ and $\bigl(\FG,\GL(\CV)\bigr)$ satisfy \eqref{EqCondQuotient}.
\item \label{PropGlobalRep_C} 
Let $\FG$ be smooth. Then the quotients in \ref{PropGlobalRep_A} and \ref{PropGlobalRep_B} are affine over $C$ if and only if all geometric fibers of $\FG$ over $C$ are reductive and $\FG/\FG\open\to C$ is finite.
\end{enumerate}
\end{proposition}

\noindent
{\it Remark.} The assumption that $\FG$ is smooth is necessary for the conclusion in \ref{PropGlobalRep_C} that the fibers of $\FG$ are reductive. This can be seen from the flat affine non-reductive group scheme $\Bmu_q$ which is the kernel of the $q$-Frobenius endomorphism of the multiplicative group $\BG_m$ and its faithful representation on $\BG_m$ with $\BG_m/\Bmu_q=\BG_m$.

\begin{proof}
\ref{PropGlobalRep_A} can be proved as in \cite[Proposition~1.3]{PR2} and \cite[Example(1), page~504]{Heinloth}. Namely, the argument of \cite[Proposition~1.3]{PR2} works over any covering $C=U_1\cup\ldots\cup U_r$ by affine open subsets $U_i\subset C$. For each $i$ it first produces a faithful representation $\FG\times_CU_i\into\GL(\CN_i)$ on a finitely generated $\CO_{U_i}$-submodule $\CN_i\subset\CO_\FG|_{U_i}$ of the ``regular representation'' $\FG\to\GL(\CO_\FG)$. By Lemma~\ref{LemmaExtendRep}\ref{LemmaExtendRep_B} below we can extend this representation to a representation $\FG\to\GL(\CF_i)$ on a vector bundle $\CF_i\subset\CO_\FG$ over all of $C$ with $\CF_i|_{U_i}=\CN_i$. The direct sum $\CW:=\CF_1\oplus\ldots\oplus\CF_r$ is a faithful representation of $\FG$ over $C$.

Let $G=\GL(\CW)$. In the next step \cite[Proposition~1.3]{PR2} construct for each $i$ a representation $\rho_i\colon G\times_CU_i\to\GL(\CV_i)$ on a vector bundle $\CV_i\subset\wedge^{\ell_i}\CO_G|_{U_i}$ of some exterior power of $\CO_G$ and a line bundle $\CL_i\subset\CV_i$ such that for any $U_i$-scheme $\pi\colon S\to U_i$ one has 
\[
\FG(S)\;=\;\bigl\{\,g\in\GL(\CW)(S)\colon \exists\,\chi_i(g)\in \CO_S(S)\mal\text{ with }\rho_i(g)(v)=\chi_i(g)v\;\forall\,v\in\Gamma(S,\pi^*\CL_i)\,\bigr\}\,.
\]
By Lemma~\ref{LemmaExtendRep}\ref{LemmaExtendRep_B} below we extend $\CV_i$ and $\CL_i$ to representations $\rho_i\colon\GL(\CW)\to\GL(\wt\CV_i)$ and $\chi_i\colon\FG\to\GL(\wt\CL_i)=\BG_m$ on vector bundles $\wt\CV_i$ and $\wt\CL_i$ over all of $C$ with $\wt\CL_i\subset\wt\CV_i$ and $\wt\CV_i|_{U_i}=\CV_i$ and $\wt\CL_i|_{U_i}=\CL_i$. Next \cite{PR2} consider the representations $(\id,\chi_i)\colon\FG\into\GL(\CW)\times_C\BG_m$ and $\rho_i\otimes(-1)\colon\GL(\CW)\times_C\BG_m\to \GL(\wt\CV_i\otimes_{\CO_C}\wt\CL_i\dual)$, $(g,\lambda_i)\mapsto\rho_i(g)\otimes\lambda_i^{-1}$. Consider the inclusion map $h_i\colon\wt\CL_i\into\wt\CV_i\,\in\,\Gamma\bigl(C,\CHom_{\CO_C}(\wt\CL_i,\wt\CV_i)\bigr)=\Gamma(C,\wt\CV_i\otimes_{\CO_C}\wt\CL_i\dual)$ and let $H_i\subset\GL(\CW)\times_C\BG_m$ be the stabilizer of the section $h_i$, that is, $H_i(S)=\bigl\{\,(g,\lambda_i)\in\GL(\CW)(S)\times\BG_m(S)\colon(\rho_i(g)\otimes\lambda_i^{-1})\cdot h_i=h_i\,\bigr\}$ for any $C$-scheme $S$. It is a closed subgroup scheme with $\FG\subset H_i$ and $\FG\times_CU_i=H_i\times_CU_i$.

Now we consider the group morphisms $(\id,\chi_1,\ldots,\chi_r)\colon\FG\into\GL(\CW)\times_C\BG_m^r$ and 
\begin{eqnarray*}
{\TS\bigoplus\limits_i}\,\rho_i\otimes(-1)\colon\GL(\CW)\times_C\BG_m^r & \longto & \GL\bigl({\TS\bigoplus}_i\wt\CV_i\otimes_{\CO_C}\wt\CL_i\dual\bigr)\,,\\[-2mm]
(g,\lambda_1,\ldots,\lambda_r) & \longmapsto & {\TS\bigoplus}_i(\rho_i(g)\otimes\lambda_i^{-1})\,,
\end{eqnarray*}
and the section $\bigoplus_ih_i\in{\TS\bigoplus}_i\wt\CV_i\otimes_{\CO_C}\wt\CL_i\dual$. By working on the covering $C=U_1\cup\ldots\cup U_r$ we compute that the stabilizer of $\bigoplus_ih_i$ in $\GL(\CW)\times_C\BG_m^r$ equals $\FG$. Let $Y$ be the reduced closure of the orbit of $\bigoplus_ih_i$ in the geometric vector bundle $\wt V$ associated to the locally free sheaf $\bigoplus_i\wt\CV_i\otimes_{\CO_C}\wt\CL_i\dual$. Then $Y$ is an affine scheme of finite type over $C$, and the quotient $(\GL(\CW)\times_C\BG_m^r)/\FG$, which is reduced by Proposition~\ref{PropQAffineQuot}\ref{PropQAffineQuot_A}, is open in $Y$ by the usual argument that orbits are open in their closure.

We show that $Y$ carries an action of $\GL(\CW)\times_C\BG_m^r$ compatible with the action on the quotient $(\GL(\CW)\times_C\BG_m^r)/\FG$. The preimage of $Y$ under the morphism $d\colon\GL(\CW)\times_C\BG_m^r\times_C\wt V\to\wt V$ is closed and contains $\GL(\CW)\times_C\BG_m^r\times_C(\GL(\CW)\times_C\BG_m^r)/\FG$. To show that this preimage contains $\GL(\CW)\times_C\BG_m^r\times_CY$ it suffices to show that the latter is irreducible and reduced. Reducedness follows from the smoothness of $\GL(\CW)\times_C\BG_m^r$ over $C$ by \cite[\S\,2.3, Proposition~9]{BLR}. Since $Y$ is irreducible and reduced it is flat over $C$ by \cite[Proposition~III.9.7]{Hartshorne}. Conversely the flatness of $\GL(\CW)\times_C\BG_m^r\times_CY$ over $C$ shows that it equals the closure of its fiber over the generic point $\eta\in C$. This fiber is irreducible, because $Y_\eta$ is irreducible and $(\GL(\CW)\times_C\BG_m^r)_\eta$ is geometrically irreducible. This proves that $\GL(\CW)\times_C\BG_m^r\times_CY$ is irreducible and maps to $Y$ under the action $d$. In particular, the pair $\bigl(\FG,\GL(\CW)\times_C\BG_m^r\bigr)$ satisfies condition~\eqref{EqCondQuotient}.

So for $\CV:=\CW\oplus\CO_C^{\oplus r}$ the pair $\bigl(\FG,\GL(\CV)\bigr)$ also satisfies condition~\eqref{EqCondQuotient} by Proposition~\ref{PropQAffineQuot}\ref{PropQAffineQuot_D} and \ref{PropQAffineQuot_B} applied to the closed immersions $\FG\into\GL(\CW)\times_C\BG_m^r\into\GL(\CW\oplus\CO_C^{\oplus r})$.

\medskip\noindent
\ref{PropGlobalRep_B}
To pass to $\SL$, we take a representation $\FG\into\GL(\CV_0)$ as in \ref{PropGlobalRep_A} and we consider $\CV:=\CV_0\oplus \wedge^\topol\CV_0\dual $ and the closed immersion $\id\oplus\det\dual\colon \GL(\CV_0) \into \GL(\CV)$ which factors through $\SL(\CV)$. Applying Proposition~\ref{PropQAffineQuot}\ref{PropQAffineQuot_D} and \ref{PropQAffineQuot_B} to the closed immersions of group schemes $\FG\into\GL(\CV_0)\into\SL(\CV)\into\GL(\CV)$ shows that the pairs $\bigl(\FG,\SL(\CV)\bigr)$ and $\bigl(\FG,\GL(\CV)\bigr)$ satisfy condition~\eqref{EqCondQuotient}.

\medskip\noindent
\ref{PropGlobalRep_C} This follows from Proposition~\ref{PropQAffineQuot}\ref{PropQAffineQuot_D}.
\end{proof}

\begin{lemma}\label{LemmaExtendRep}
Let $\CE$ be a quasi-coherent sheaf of $\CO_C$-modules on $C$. Let $U\subset C$ be an open subscheme and let $\CN\subset\CE|_U$ be a finitely generated $\CO_U$-submodule.
\begin{enumerate}
\item \label{LemmaExtendRep_A}
There is a finitely generated $\CO_C$-submodule $\CH\subset\CE$ with $\CH|_U=\CN$.
\item \label{LemmaExtendRep_B}
Assume that $\CE$ is a flat $\CO_C$-module. Let $\FG\to\GL(\CE)$ be a representation (possibly of infinite dimension) and assume that $\CN$ is $\FG\times_CU$-invariant. Then there is a finitely generated flat $\FG$-invariant $\CO_C$-submodule $\CF\subset\CE$ with $\CN=\CF|_U$.
\end{enumerate}
\end{lemma}

\begin{proof}
If $U=C$ there is nothing to do. Otherwise we choose $U'=\Spec A\subset C$ open such that $C=U\cup U'$. By Riemann-Roch $\wt U:=U\cap U'$ is affine with $\CO_C(\wt U)=A[\tfrac{1}{a}]$ for an element $a\in A$.

\medskip\noindent
\ref{LemmaExtendRep_A} Let $(m_i)$ be a finite set of generators of the $A[\tfrac{1}{a}]$-module $\CN|_{\wt U}$. After multiplying the $m_i$ by suitable powers of $a$ they lie in $\CE|_{U'}$ and generate a finite $\CO_{U'}$-submodule $\CN'\subset\CE|_{U'}$ which by construction satisfies $\CN'|_{\wt U}=\CN|_{\wt U}$. Glueing $\CN$ with $\CN'$ over $\wt U$ produces the desired finitely generated submodule $\CH$.

\medskip\noindent
\ref{LemmaExtendRep_B} By \ref{LemmaExtendRep_A} we may choose a finitely generated $\CO_C$-submodule $\CH\subset\CE$ with $\CH|_U=\CN$. Let $d_\CE\colon\CE\to\CO_\FG\otimes_{\CO_C}\CE$ be the co-multiplication of the representation $\FG\to\GL(\CE)$. For each open affine subset $V\subset C$ let $\CF(V):=\{\,x\in\Gamma(V,\CE)\colon d_\CE(x)\in\Gamma(V,\CO_\FG\otimes_{\CO_C}\CH)\,\}$. Then $\CF(V)$ is $\FG\times_CV$-invariant and contained in $\CH(V)$ by \cite[\S\,1.5, Proposition~1]{Serre68}. Since $\CF(V)$ equals the kernel of the homomorphism $\Gamma(V,\CE)\xrightarrow{\,d_\CE}\Gamma(V,\CO_\FG\otimes_{\CO_C}\CE)\onto\Gamma(V,\CO_\FG\otimes_{\CO_C}\CE/\CH)$ the $\CF(V)$ form a sheaf $\CF\subset\CH$ of $\CO_C$-submodules. It follows that $\CF$ is finitely generated and $\FG$-invariant. Since $\CE$ is flat, whence $\CO_C$-torsion free, the same holds for $\CF$ because $C$ is a smooth curve. Moreover, $\CN\subset\CF|_U\subset\CH|_U=\CN$ by the $\FG\times_CU$-invariance of $\CN$.
\end{proof}

\begin{definition}\label{DefD-LevelStr}
We let $\scrH^1(C,\FG)$ denote the category fibered in groupoids over the category of $\BF_q$-schemes, such that the objects over $S$, $\scrH^1(C,\FG)(S)$, are $\FG$-torsors (also called $\FG$-bundles) over $C_S$ and morphisms are isomorphisms of $\FG$-torsors.

Let $D$ be a proper closed subscheme of $C$. A \emph{$D$-level structure} on a $\FG$-bundle $\CG$ on $C_S$ is a trivialization $\psi\colon \CG\times_{C_S}{D_S}\isoto \FG\times_C D_S$ along $D_S:=D\times_{\BF_q}S$. Let $\scrH^1_D(C,\FG)$ denote the stack classifying $\FG$-bundles with $D$-level structure, that is, $\scrH^1_D(C,\FG)$ is the category fibered in groupoids over the category of $\BF_q$-schemes, which assigns to an $\BF_q$-scheme $S$ the category whose objects are
$$
Ob\bigl(\scrH^1_D(C,\FG)(S)\bigr):=\left\lbrace (\CG,\psi)\colon \CG\in \scrH^1(C,\FG)(S),\, \psi\colon \CG\times_{C_S}{D_S}\isoto \FG\times_C D_S \right\rbrace,
$$ 
and whose morphisms are those isomorphisms of $\FG$-bundles that preserve the $D$-level structure.
\end{definition}

While we adopt the notation of Behrend~\cite{Beh}, the stack $\scrH^1(C,\FG)$ is also often denoted $\text{Bun}_\FG$. The following theorem is well known when $\FG$ is a parahoric group scheme over $C$; see Definition~\ref{DefParahoric} below. A proof is given for example by Heinloth~\cite[Proposition~1]{Heinloth}.

\begin{theorem}\label{Bun_G}
Let $\FG$ be a flat affine group scheme of finite type over the curve $C$. Then the stack $\scrH^1(C,\FG)$ is an Artin-stack locally of finite type over $\BF_q$. It admits a covering by connected open substacks of finite type over $\BF_q$. If $\FG$ is smooth over $C$ then $\scrH^1(C,\FG)$ is smooth over $\BF_q$.
\end{theorem}

The theorem can be proved following the argument of Behrend~\cite{Beh}. Since we will need the technique in Theorem~\ref{nHisArtin} below, we indicate Behrend's argument and thereby generalize his Propositions~4.4.4 and 4.4.1 and 4.4.5.

\begin{theorem}\label{r-q-a} Let $\CV$ be a vector bundle over $C$ and let $\rho\colon\FG\hookrightarrow \GL(\CV)$ be a closed subgroup scheme flat over $C$, such that $\bigl(\FG,\GL(\CV)\bigr)$ satisfies condition~\eqref{EqCondQuotient}. Then the natural morphism of stacks over $\BF_q$
$$
\rho_\ast\colon \scrH^1(C,\FG) \to \scrH^1(C,\GL(\CV))
$$ 
is representable by a morphism of schemes which is quasi-affine and of finite presentation. 
\end{theorem}

\begin{proof}
Let $p_S\colon C_S\to S$ be the projection map and view it as a morphism of big \'etale sites $\Et(C_S) \to \Et(S)$. For any scheme $Y$ over $C_S$ let $p_{S\ast} (Y)$ denote the sheaf which sends an $S$-scheme $T$ to $\Hom_{C_S}(C_T,Y)$. Let $\CG$ be a $\GL(\CV)$-bundle in $\scrH^1(C,\GL(\CV))(S)$. By \cite[Proposition~4.2.3]{Beh} we have the following 2-cartesian diagram of stacks
\begin{equation}\label{EqBehrend}
\CD
p_{S\ast}\left(\CG/\FG_S\right)@>>>S\\
@VVV @V{\CG}VV\\
\scrH^1(C,\FG)@>>> \scrH^1(C,\GL(\CV)).
\endCD 
\end{equation}
(See the proof of Proposition~\ref{HeckeisQP} for more details on this diagram.) We must show that $p_{S\ast}\left(\CG/\FG_S\right)$ is a quasi-affine $S$-scheme of finite presentation. We claim that there exists a scheme $Y$ affine and of finite presentation over $C_S$ and a quasi-compact open immersion of $C_S$-schemes $\CG/\FG_S\into Y$. We will prove this over a Zariski covering $U'\lang{$\to C_S$}$ of $C_S$ and then glue. In this way we can assume that there is a trivialization $\alpha\colon\CG\times_{C_S}U'\isoto\GL(\CV)\times_{C}U'$. Here a Zariski covering suffices, because every $\GL(\CV)$-torsor is trivial Zariski locally by Hilbert~90, see \cite[Proposition~III.4.9]{Milne}. Since $\bigl(\FG,\GL(\CV)\bigr)$ satisfies condition~\eqref{EqCondQuotient} there is a scheme $Y_0$ affine and of finite presentation over $C$ with an action $\GL(\CV)\times_CY_0\to Y_0$ of $\GL(\CV)$ and a $\GL(\CV)$-equivariant open immersion $\iota\colon\GL(\CV)/\FG\into Y_0$ which automatically is quasi-compact. We consider the quasi-compact open immersion $\iota\alpha\colon(\CG/\FG_S)\times_{C_S}U'\isoto(\GL(\CV)/\FG)\times_CU'\into Y_0\times_CU'$. Let $p_i\colon U'':=U'\times_{C_S}U'\to U'$ be the projection onto the $i$-th factor. Over $U''$ the automorphism $p_2^*\alpha\circ p_1^*\alpha^{-1}$ of the trivial $\GL(\CV)$-torsor is given by multiplication with an element $g\in\GL(\CV)(U'')$ satisfying the cocycle condition on $U''':=U'\times_{C_S}U'\times_{C_S}U'$. The isomorphism $g\colon p_1^*(Y_0\times_CU')\isoto p_2^*(Y_0\times_CU')$ allows to glue $Y_0\times_CU'$ to a $C_S$-scheme $Y$ which is affine and of finite presentation over $C_S$. By construction, the open immersion $\iota\alpha$ glues to a quasi-compact open immersion $\CG/\FG_S\into Y$ of $C_S$-schemes. Thus the theorem follows from the next lemma.
\end{proof}

\begin{lemma}\label{LemmaBehrend_Quasi-affine}
Let $X$ be a projective scheme over a field $\ArbitraryFld$. Let $p$ denote the structure morphism $p\colon  X \to \Spec \ArbitraryFld$. Let $S$ be a $\ArbitraryFld$-scheme, let $Y$ be a scheme affine and of finite presentation over $X_S:=X\times_\ArbitraryFld S$, and let $\wt Y\into Y$ a quasi-compact open immersion. Then ${p_{S}}_\ast\wt Y$ is a quasi-affine $S$-scheme of finite presentation.
\end{lemma}

\begin{proof}
Since $Y$ is of finite presentation over $X_S$ and $\wt Y\into Y$ is quasi-compact we may reduce to the case that $S$ is of finite type over $\ArbitraryFld$ and in particular noetherian; see \cite[IV$_3$, Proposition~8.9.1]{EGA}. By \cite[Proposition~4.4.1]{Beh} the sheaf $\CA:=p_{S\ast}Y$ is representable by an affine $S$-scheme of finite presentation. Consider the universal $X_S$-morphism $X_\CA \to Y$ corresponding to $\id_\CA\in \CA(\CA)$. Let $Z\hookrightarrow X_\CA$ be the complement of $X_\CA\times_Y \wt Y$ in $X_\CA$. By properness of $X$, $Z$ maps to a closed subscheme of $\CA$. Let $\CU$ denote the complement of $p_\CA(Z)$ in $\CA$. It is quasi-compact because we assumed that $S$ and hence $\CA$ are noetherian. In particular, $\CU$ is a quasi-affine $S$-scheme of finite presentation. 

We claim that $\CU$ represents $p_{S\ast}\wt Y$. To see this, first observe that the open immersion $X_\CU \hookrightarrow X_\CA\times_Y \wt Y$ gives a morphism $X_\CU\to \wt Y$ which induces a $\CU$-point in $(p_{S\ast}\wt Y)(\CU)$. Hence it is enough to check that for any scheme $T$ we have the inclusion $(p_{S\ast}\wt Y)(T)\subseteq \CU(T)$. Any point of $(p_{S\ast}\wt Y)(T)$ is a morphism $X_T \to \wt Y$. Composing with $i\colon \wt Y\to Y$ induces a $T$ valued point $\alpha$ of $\CA$. We have to show that $\alpha \colon  T\to \CA$ factors through the open subscheme $\CU$. One can easily check this on the level of topological spaces. Namely, if $T\times_\CA p_\CA(Z)$ is non-empty then so is $X_T\times_{X_\CA} Z$, which is a contradiction, since $Z$ is defined as the complement of $X_\CA\times_Y \wt Y$ in $X_\CA$. This shows that $\CU$ represents $p_{S\ast}\wt Y$.
\end{proof}

For the proof of Theorem~\ref{Bun_G} we use the following 

\begin{remark}\label{RemVect}
Let $\CV$ be a vector bundle on $C$ of rank $r$. The stack $\scrH^1(C,\GL(\CV))$ is isomorphic to the stack $\Vect_C^r$ whose $S$-valued points parameterize vector bundles of rank $r$ on $C_S$. For a $\GL(\CV)$-torsor $\CG$ over $C_S$ we let $\CV(\CG)$ denote the associated vector bundle over $C_S$. By \cite[Th\'eor\`eme 4.6.2.1]{L-M} the stack $\Vect^r_C$ is an Artin-stack locally of finite type over $\BF_q$ and by \cite[Theorem~1.0.1]{Wang} it admits a covering by connected open substacks of finite type over $\BF_q$. (Note that even though \cite[Th\'eor\`eme 4.6.2.1]{L-M} states that $\Vect^r_C$ is of finite type, all that is true and proved is that it is \emph{locally} of finite type.)
\end{remark}

\begin{proof}[Proof of Theorem~\ref{Bun_G}.]
By Proposition~\ref{PropGlobalRep}\ref{PropGlobalRep_A} there is a closed embedding $\FG\hookrightarrow\GL(\CV)$ such that $\bigl(\FG,\GL(\CV)\bigr)$ satisfies \eqref{EqCondQuotient}. By Theorem~\ref{r-q-a} and Remark~\ref{RemVect} the stack  $\scrH^1(C,\FG)$ is an Artin-stack locally of finite type over $\BF_q$ and admits a covering by connected open substacks of finite type over $\BF_q$. Finally the smoothness of $\scrH^1(C,\FG)$ follows from the vanishing of the second cohomology of coherent sheaves on a curve; compare \cite[Corollary~4.5.2]{Beh} or \cite[\S\,6]{Wang}.
\end{proof}

\begin{remark}\label{RemCoveringH1}
A covering of $\scrH^1(C,\FG)$ by connected open substacks of finite type over $\BF_q$ can be constructed as follows. Let $\rho\colon\FG\hookrightarrow \GL(\CV)$ be a representation with quasi-affine quotient which factors through $\SL(\CV)$ as in Proposition~\ref{PropGlobalRep}\ref{PropGlobalRep_B} such that $\bigl(\FG,\GL(\CV)\bigr)$ satisfies \eqref{EqCondQuotient}, and let $\mu=(\mu_1\ge\ldots\ge\mu_r)$ with $\mu_1+\ldots+\mu_r=0$ be a dominant cocharacter of $\SL_r$, where $r=\rk\CV$. We use the notation of Remark~\ref{RemVect}. Let $\scrH^1(C,\FG)^{\le\mu}$ be the sub-stack of $\scrH^1(C,\FG)$ consisting of those $\FG$-bundles $\CG$ on $C_S$ for $\BF_q$-schemes $S$ such that for any geometric point $s\in S$ and for any filtration $(0)=\CF_0\subset\ldots\subset\CF_r=\CV(\rho_*\CG)_s$ of the vector bundle $\CF:=\CV(\rho_*\CG)_s$ on $C_s$ by sub-vector bundles $\CF_i$ with line bundles as subquotients (such a filtration is called a \emph{$B$-structure}) the degree of the line bundle $\CF_i/\CF_{i-1}$ is at most $\mu_i$. Since $\rho$ factors through $\SL(\CV)$ there is a canonical isomorphism $\alpha_\CG\colon\wedge^r\CV(\rho_*\CG)\isoto\CO_{C_S}$ and therefore the condition that $\deg(\CF_i/\CF_{i-1})\le\mu_i$ for all $i$ is equivalent to the condition 
\begin{equation}\label{EqBunMu}
\TS\deg\bigotimes\limits_{i=1}^r(\CF_i/\CF_{i-1})^{\otimes\lambda_i}\;\le\;\sum\limits_{i=1}^r\mu_i\lambda_i\quad\text{for all integers }\lambda_1\ge\ldots\ge\lambda_r\text{ with }\lambda_1+\ldots+\lambda_r=0\,.
\end{equation}
It follows from \cite[Lemma~A.3]{Var}, that the stack $\scrH^1(C,\GL_r)^{\le\mu}$ defined by condition \eqref{EqBunMu} (which for arbitrary vector bundles $\CF$ is not equivalent to $\deg(\CF_i/\CF_{i-1})\le\mu_i$ for all $i$) is an open substack of $\scrH^1(C,\GL_r)\cong Vect_C^r$. Its intersections with the connected components of $\scrH^1(C,\GL_r)$ are Artin stacks of finite type over $\BF_q$ by \cite[Lemma~3.1(a)]{Var}. Thus also $\scrH^1(C,\FG)^{\le\mu}$ is an open substack of $\scrH^1(C,\FG)$ and an Artin stack of finite type over $\BF_q$, being the preimage of $\scrH^1(C,\GL_r)^{\le\mu}$ under the morphism of stacks $\rho_\ast\colon \scrH^1(X,\FG) \to \scrH^1(X,\GL(\CV))$ from Theorem~\ref{r-q-a}, which factors through the connected component on which the vector bundles have degree zero; compare \cite[\S\,2.1.1]{B-D} or \cite[Lemma~2.2]{Var}.

Moreover, if the closed subscheme $D\subset C$ is large enough compared to $\mu$ then the preimage $\scrH^1_D(C,\GL_r)^{\le\mu}$ of $\scrH^1(C,\GL_r)^{\le\mu}$ in $\scrH^1_D(C,\GL_r)$ is a scheme over $\BF_q$ with quasi-projective connected components by \cite[Lemma~3.1(a)]{Var}. Therefore also $\scrH^1_D(C,\FG)^{\le\mu}$ is a quasi-projective scheme over $\BF_q$ by Theorem~\ref{r-q-a}.
\end{remark}

\section{Moduli Stacks of Global $\FG$-Shtukas}\label{SectGlobalGShtukas}
\setcounter{equation}{0}

\begin{definition}\label{Hecke}
Let $I$ be a finite set and let $I=I_1\cup\ldots\cup I_k$ with $I_j\cap I_{j'}=\emptyset$ for $j\ne j'$ be a partition of $I$. We write $\ulI:=(I_1,\ldots,I_k)$. We let $Hecke_{\FG,D,\,\ulI}$ be the stack fibered in groupoids over the category of $\BF_q$-schemes, whose $S$ valued points are tuples $\bigl((\charsect_i)_{i\in I},\,(\CG_j,\psi_j)_{j=0,\ldots,k},\,(\tauGlob_j)_{j=1,\ldots,k}\,\bigr)$ where
\begin{itemize}
\item[] $\charsect_i \in (C\setminus D)(S)$ for all $i\in I$ are sections, called \emph{legs} (or \textit{characteristic sections}), 
\item[] $(\CG_j,\psi_j)$ for all $j=0,\ldots,k$ are objects in $\scrH^1_D(C,\FG)(S)$, and
\item[] $\tauGlob_j\colon  \CG_{j-1}|_{{C_S}\setminus\bigcup_{i\in I_j}\Gamma_{\charsect_i}}\isoto \CG_j|_{{C_S}\setminus\bigcup_{i\in I_j}\Gamma_{\charsect_i}}$ for all $j=1,\ldots k$ are isomorphisms preserving the $D$-level structures, that is, $\psi_j\circ\tauGlob_j|_{D_S}=\psi_{j-1}$. Here $\Gamma_{\charsect_i}\subset C_S$ denotes the graph of the section $\charsect_i$.
\end{itemize}
If $D=\emptyset$ we will drop it from the notation. For each $j=0,\ldots,k$ we denote by $pr_j\colon Hecke_{\FG,D,\,\ulI}\to \scrH^1_D(C,\FG)$ the morphism sending $\bigl((\charsect_i)_{i\in I},\,(\CG_j,\psi_j)_{j=0,\ldots,k},\,(\tauGlob_j)_{j=1,\ldots,k}\,\bigr)$ to $(\CG_j,\psi_j)$. If $k=1$ and $I=I_1=\{1,\ldots,n\}$ we also write $Hecke_{\FG,D,n}:=Hecke_{\FG,D,(I)}:=Hecke_{\FG,D,\,\ulI}$. In general there is a canonical isomorphism
\begin{equation}\label{EqHeckeProduct}
Hecke_{\FG,D,\,\ulI}\isoto Hecke_{\FG,D,(I_1)}\;\underset{pr_1,\scrH^1_D(C,\FG),pr_0}{\times}\;\ldots\;\underset{pr_1,\scrH^1_D(C,\FG),pr_0}{\times}\;Hecke_{\FG,D,(I_k)}
\end{equation}
sending the tuple $\bigl((\charsect_i)_{i\in I},\,(\CG_j,\psi_j)_{j=0,\ldots,k},\,(\tauGlob_j)_{j=1,\ldots,k}\,\bigr)\in Hecke_{\FG,D,\,\ulI}(S)$ to the tuple in the $j$-th factor $\bigl((\charsect_i)_{i\in I_j},\,(\CG_{j-1},\psi_{j-1}),\,(\CG_j,\psi_j),\,\tauGlob_j\,\bigr)\in Hecke_{\FG,D,(I_j)}(S)$.
\end{definition}

\begin{definition}\label{ker}
Assume that we have two morphisms $f,g\colon X\to Y$ of schemes or stacks. We denote by $\equi(f,g\colon X\rightrightarrows Y)$ the pull back of the diagonal under the morphism $(f,g)\colon X\to Y\times_\BZ Y$, that is $\equi(f,g\colon X\rightrightarrows Y)\,:=\,X\times_{(f,g),Y\times Y,\Delta}Y$ where $\Delta=\Delta_{Y/\BZ}\colon Y\to Y\times_\BZ Y$ is the diagonal morphism.
\end{definition}

\begin{definition}\label{Global Sht}
Keep the notation of Definition~\ref{Hecke}. We define the \emph{moduli stack} $\nabla_\ulI\scrH^1_D(C,\FG)$ \emph{of (iterated) global $\FG$-shtukas with $D$-level structure} to be the preimage on $Hecke_{\FG,D,\,\ulI}$ of the graph of the Frobenius morphism on $\scrH^1(C,\FG)$. In other words
$$
 \nabla_\ulI\scrH^1_D(C,\FG):=\equi(\sigma_{\scrH^1_D(C,\FG)} \circ pr_k,\,pr_0\colon  Hecke_{\FG,D,\,\ulI}\rightrightarrows \scrH^1_D(C,\FG)).
$$ 
If $k=1$ and $I=I_1=\{1,\ldots,n\}$ we also write $\nabla_n\scrH^1_D(C,\FG):=\nabla_\ulI\scrH^1_D(C,\FG)$. Each object $\ul{\cG}$ in $Ob(\nabla_\ulI\scrH^1_D(C,\FG)(S))$ is called an \emph{(iterated) global $\FG$-shtuka with $D$-level structure over $S$} and the corresponding sections $(\charsect_i)_{i\in I}$ are called the \textit{legs} (or \textit{characteristic sections}) of $\ul\CG$, or of $S$. 

More explicitly, a global $\FG$-shtuka $\ul\CG\in\nabla_\ulI\scrH^1_D(C,\FG)(S)$ with $D$-level structure over an $\BF_q$-scheme $S$ is an object $\ul\CG=\bigl((\charsect_i)_{i\in I},\,(\CG_j,\psi_j)_{j=0,\ldots,k},\,(\tauGlob_j)_{j=1,\ldots,k}\,\bigr)$ of $Hecke_{\FG,D,\,\ulI}(S)$ together with an isomorphism $\tauGlob_0\colon\sigma^*(\CG_k,\psi_k)\isoto(\CG_0,\psi_0)$ in $\scrH^1_D(C,\FG)$, that is an isomorphism $\tau_0\colon\sigma^*\CG_k\isoto\CG_0$ of $\FG$-bundles satisfying $\psi_0\circ\tau_0|_{D_S}=\sigma^*\psi_k$.

If $k=1$ and $I=I_1=\{1,\ldots,n\}$ a global $\FG$-shtuka $\ul\CG\in\nabla_n\scrH^1_D(C,\FG)(S)$ with $D$-level structure over an $\BF_q$-scheme $S$ is a tuple $\ul\CG=(\charsect_1,\ldots,\charsect_n,\CG,\psi,\tauGlob)$ consisting of an $n$-tuple of legs $(\charsect_1,\ldots,\charsect_n)\in (C\setminus D)^n(S)$, a $\FG$-bundle $\CG:=\CG_1$ over $C_S$, a trivialization $\psi:=\psi_1\colon \CG\times_{C_S}{D_S}\isoto \FG\times_C D_S$, and an isomorphism $\tauGlob:=\tau_1\circ\tau_0\colon  \s \CG|_{C_S\setminus \cup_i \Gamma_{\charsect_i}}\isoto \CG|_{C_S\setminus \cup_i \Gamma_{\charsect_i}}$ with $\psi\circ\tauGlob|_{D_S}=\s(\psi)$. 

If $D=\emptyset$ we drop $\psi$ from $\ul\CG$ and write $\nabla_\ulI\scrH^1(C,\FG)$, resp.\ $\nabla_n\scrH^1(C,\FG)$ for the \emph{stack of global $\FG$-shtukas}. If $k=1$ and $I=I_1=\{1,\ldots,n\}$ we will sometimes fix the legs $(\charsect_1,\ldots,\charsect_n)\in C^n(S)$ and simply call $\ul\CG=(\CG,\tauGlob)$ a global $\FG$-shtuka over $S$.
\end{definition}

Global $\FG$-shtukas can be viewed as function field analogs of abelian varieties with extra structure. This inspires the following notions of quasi-isogenies.

\begin{definition}\label{quasi-isogenyGlob}
Let $k=1$ and $I=I_1=\{1,\ldots,n\}$. Consider a scheme $S$ together with legs $\charsect_i\colon S\to C$ for $i=1,\ldots,n$ and let $\ul{\CG}=(\CG,\tau)$ and $\ul{\CG}'=(\CG',\tau')$ be two global $\FG$-shtukas over $S$ with the same legs $\charsect_i$. A \emph{quasi-isogeny} from $\ul\CG$ to $\ul\CG'$ is an isomorphism $f\colon\CG|_{C_S \setminus E_S}\isoto \CG'|_{C_S \setminus E_S}$ satisfying $\tau'\sigma^{\ast}(f)=f\tau$, where $E$ is some effective divisor on $C$. We denote the \emph{group of quasi-isogenies} from $\ul\CG$ to itself by $\QIsog_S(\ul\CG)$.
\end{definition}

\bigskip

In order to obtain algebraic substacks locally of finite type of $Hecke_{\FG,D,\,\ulI}$ and $\nabla_\ulI\scrH^1_D(C,\FG)$ one has to bound the relative position of $\CG_{j-1}$ and $\CG_j$ under the isomorphism $\tauGlob_j$. For the rest of this section we fix a faithful representation $\rho\colon  \FG \to \SL(\CV_0)\subset\GL(\CV_0)$ for some vector bundle  $\CV_0$ of rank $r$, with quasi-affine quotient $\SL(\CV_0)\slash \FG$ as in Proposition~\ref{PropGlobalRep}\ref{PropGlobalRep_B} such that the pairs $\bigl(\FG,\SL(\CV_0)\bigr)$ and $\bigl(\FG,\GL(\CV_0)\bigr)$ satisfy \eqref{EqCondQuotient}. In Remark~\ref{RemIndepOfRep} we explain to what extent our results depend on the choice of $\rho$.

\begin{remark}\label{RemRelAffineGrass}
We consider the induced morphisms of stacks
$$
\xymatrix @R=0pc {
\scrH^1(C,\FG) \ar[r]^{\TS\rho_*\quad} & \scrH^1(C,\SL(\CV_0)) \ar[r] & \scrH^1(C,\GL(\CV_0)) \ar[r]^{\qquad\sim} & \Vect_C^r\\
\CG \ar@{|->}[r] & \rho_*\CG \ar@{|->}[rr] & & \CV(\rho_*\CG)\,,
}
$$ 
see Remark~\ref{RemVect}. Since $\rho$ factors through $\SL(\CV_0)$ there is a canonical isomorphism $\alpha_\CG\colon\wedge^r\CV(\rho_*\CG)\isoto\CO_{C_S}$. Conversely, every pair $(\CV,\alpha)$ where $\CV$ is a rank $r$ vector bundle on $C_S$ and $\alpha\colon\wedge^r\CV\isoto\CO_{C_S}$ is an isomorphism is induced from an $\SL(\CV_0)$-torsor over $C_S$.
\end{remark}

As an auxiliary tool to prove the representability of $Hecke_{\FG,D,\,\ulI}$ and $\nabla_\ulI\scrH^1_D(C,\FG)$ we introduce the following stack.

\begin{definition}\label{DefRelAffineGrass}
Keep the notations of Remark~\ref{RemRelAffineGrass}. We define the \emph{relative affine Grassmannian} $\CG r_{\FG,I,r}$ as the stack over $C^I\times_{\BF_q}\scrH^1(C,\FG)$ which parametrizes tuples $\bigl((\charsect_i)_{i\in I},\CG, \CV',\alpha',\phi\bigr)$, where
$$
\bigl((\charsect_i)_{i\in I},\CG, \CV'\bigr)\in C^I\times_{\BF_q} \scrH^1(C,\FG)\times_{\BF_q} \Vect_C^r,
$$ 
$\alpha'\colon\wedge^r\CV'\isoto\CO_{C_S}$ is an isomorphism and $\phi\colon\CV(\rho_*\CG)|_{C_S\setminus\bigcup_{i\in I}\Gamma_{\charsect_i}}\isoto\CV'|_{C_S\setminus\bigcup_{i\in I}\Gamma_{\charsect_i}}$ is an isomorphism between the vector bundle $\CV(\rho_*\CG)$ associated with $\rho_\ast \CG$ and $\CV'$ outside the graphs $\bigcup_{i\in I}\Gamma_{\charsect_i}$ such that $\alpha_\CG=\alpha'\circ\wedge^r\phi$ on $C_S\setminus\cup\Gamma_{\charsect_i}$. In particular $\wedge^r\phi=(\alpha')^{-1}\circ\alpha_\CG$ extends to an isomorphism $\wedge^r\phi\colon\wedge^r\CV(\rho_*\CG)\isoto\wedge^r\CV'$ on all of $C_S$. If $\CS l'$ is the $\SL(\CV_0)$-torsor associated with $(\CV',\alpha')$ then $\phi$ induces an isomorphism $\phi\colon\rho_*\CG|_{C_S\setminus\bigcup_{i\in I}\Gamma_{\charsect_i}}\isoto\CS l'|_{C_S\setminus\bigcup_{i\in I}\Gamma_{\charsect_i}}$.

Note that for every $j\in\{1,\ldots,k\}$ the morphism  $\rho_\ast$ yields a morphism $Hecke_{\FG,\,\ulI}\to \CG r_{\FG,I_j,r}$, sending the tuple $\bigl((\charsect_i)_{i\in I},\,(\CG_j)_j,\,(\tauGlob_j)_j\bigr)$ to the tuple $\bigl((\charsect_i)_{i\in I_j},\CG_{j-1},\CV(\rho_\ast\CG_j),\alpha_{\CG_j},\CV(\rho_*\tauGlob_j)\bigr)$ where $\alpha_{\CG_j}\colon\wedge^r\CV(\rho_*\CG_j)\isoto\CO_{C_S}$ is the canonical isomorphism induced from the fact that $\rho$ factors through $\SL(\CV_0)$. The morphism $Hecke_{\FG,\,\ulI}\to \CG r_{\FG,I_j,r}$ factors through the $j$-th component in \eqref{EqHeckeProduct}.
\end{definition}

Let $\ul{\omega}:=(\omega_i)_{i\in I}$ be a tuple of coweights of $\SL_r$ given as $\omega_i\colon x\mapsto\diag(x^{\omega_{i,1}},\ldots,x^{\omega_{i,r}})$ for integers $\omega_{i,1}\ge\ldots\ge\omega_{i,r}$ with $\omega_{i,1}+\ldots+\omega_{i,r}=0$ for all $i$. (The inequality means that all $\omega_i$ are dominant with respect to the Borel subgroup of upper triangular matrices.) Let $\CG r_{\FG,I,r}^{\ul\omega}$ denote the substack of $\CG r_{\FG,I,r}$ defined by the condition that the isomorphism $\phi$ is \emph{bounded} by $\ul\omega$, that is satisfies
\begin{eqnarray}\label{EqBounded1}
&\bigwedge^\ell_{C_S}\phi\bigl(\CV(\rho_*\CG)\bigr)\;\subset\;\Bigl(\bigwedge^\ell_{C_S} \CV'\Bigr)\bigl(\sum_{i\in I}(-\omega_{i,r-\ell+1}-\ldots-\omega_{i,r})\!\cdot\!\Gamma_{s_i}\bigr)\\
&\text{for all $1\le \ell\le r$ with equality for $\ell=r$} \nonumber
\end{eqnarray}
where the notation $\bigl(\wedge^\ell_{C_S} \CV'\bigr)(\sum_{i\in I}(-\omega_{i,r-\ell+1}-\ldots-\omega_{i,r})\cdot\Gamma_{s_i})$ means that we allow poles of order $-\omega_{i,r-\ell+1}-\ldots-\omega_{i,r}$ along the Cartier divisor $\Gamma_{s_i}$ on $C_S$; compare \cite[Lemma~4.3]{H-V}. Note that the condition for $\ell=r$ is equivalent to the requirement that $\wedge^r\phi$ is an isomorphism on all of $C_S$, which in turn is equivalent to the condition that $\alpha_\CG=\alpha'\circ\wedge^r\phi$ for an isomorphism $\alpha'\colon\wedge^r\CV'\isoto\CO_{C_S}$. By Cramer's rule (e.g.~\cite[III.8.6, Formulas (21) and (22)]{BourbakiAlgebra}) condition \eqref{EqBounded1} is equivalent to 
\begin{eqnarray}\label{EqBounded2}
&\bigwedge^\ell_{C_S}\phi^{-1}(\CV')\;\subset\;\Bigl(\bigwedge^\ell_{C_S} \CV(\rho_*\CG)\Bigr)\bigl(\sum_i(\omega_{i,1}+\ldots+\omega_{i,\ell})\!\cdot\!\Gamma_{s_i}\bigr)\\
&\text{for all $1\le \ell\le r$ with equality for $\ell=r$}\nonumber
\end{eqnarray}
Again the condition for $\ell=r$ is equivalent to the condition that $\alpha_\CG=\alpha'\circ\wedge^r\phi$ for an isomorphism $\alpha'\colon\wedge^r\CV'\isoto\CO_{C_S}$.

\begin{proposition}\label{RelativGrProj/BunG}
The relative affine Grassmannian $\CG r_{\FG,I,r}^{\ul\omega}$ is relatively representable over $C^I\times_{\BF_q}\scrH^1(C,\FG)$ by a projective morphism of schemes.
\end{proposition}

\begin{proof}
We look at the fiber of $\CG r_{\FG,I,r}^{\ul\omega}\to C^I\times_{\BF_q}\scrH^1(C,\FG)$ over an $S$-valued point $((\charsect_i)_{i\in I},\CG)$ in $(C^I\times_{\BF_q}\scrH^1(C,\FG))(S)$. Then \eqref{EqBounded1} and \eqref{EqBounded2} imply that \mbox{$\CV(\rho_*\CG)(\sum_{i\in I}\omega_{i,1}\!\cdot\!\Gamma_{s_i})/\phi^{-1}(\CV')$} must be a quotient of the sheaf
$$
\TS\CF\;:=\;\CV(\rho_*\CG)(\sum_{i\in I}\omega_{i,1}\!\cdot\!\Gamma_{s_i})\big/\CV(\rho_*\CG)(\sum_{i\in I}\omega_{i,r}\!\cdot\!\Gamma_{s_i})
$$
supported on the effective relative Cartier divisor $X:=\sum_{i\in I}(\omega_{i,1}-\omega_{i,r})\cdot\Gamma_{s_i}$ inside $C_S$. Note that $X$ is a finite flat $S$-scheme. From the case $\ell=r$ in \eqref{EqBounded1} and \eqref{EqBounded2} we also obtain the isomorphism \mbox{$\alpha':=\alpha_\CG\circ(\wedge^r\phi)^{-1}\colon\wedge^r\CV'\isoto\CO_{C_S}$}. Therefore $\CG r_{\FG,I,r}^{\ul\omega}\times_{(C^I\times_{\BF_q}\scrH^1(C,\FG))}S$ is represented by a closed subscheme of Grothendieck's Quot-scheme ${\rm Quot}^\Phi_{\CF/X/S}$, see \cite[n$\open$221, Th\'eor\`eme~3.1]{FGA} or \cite[Theorem~2.6]{AltmanKleiman}, for constant Hilbert polynomial $\Phi=r\cdot\sum_{i\in I}\omega_{i,1}$; compare \cite[p.~5]{HartlHellmann}. 
\end{proof}

\begin{definition}
The stack $Hecke_{\FG,\,\ulI}^{\ul\omega}$ is defined by the cartesian diagram
$$
\xymatrix {
Hecke_{\FG,\,\ulI}^{\ul\omega} \ar[r]\ar[d] & \CG r_{\FG,I_1,r}^{\ul\omega}\times_{\BF_q}\ldots\times_{\BF_q}\CG r_{\FG,I_k,r}^{\ul\omega} \ar[d] \\
Hecke_{\FG,\,\ulI} \ar[r] & \CG r_{\FG,I_1,r}\times_{\BF_q}\ldots\times_{\BF_q}\CG r_{\FG,I_k,r}\;.
}
$$
\end{definition}

\begin{proposition}\label{HeckeisQP}
Let $\rho\colon  \FG \to \SL(\CV_0)$ be a faithful representation as above with quasi-affine (resp.\ affine) quotient $\SL(\CV_0)\slash \FG$. Then the morphism $Hecke_{\FG,(I)}\to\CG r_{\FG,I,r}$ is represented by a locally closed and quasi-compact (resp.\ a closed) immersion. In particular, the morphism of stacks $Hecke_{\FG,\,\ulI}^{\ul\omega}\,\to\,C^I\times_{\BF_q}\scrH^1(C,\FG)$ sending $\bigl((\charsect_i)_{i\in I},\,(\CG_j)_{j=0,\ldots,k},\,(\tauGlob_j)_{j=1,\ldots,k}\,\bigr)$ to $\bigl((\charsect_i)_{i\in I},\CG_0\bigr)$ is relatively representable by a morphism of schemes which is quasi-compact and quasi-projective, and even projective if there is a representation $\rho$ with affine quotient $\SL(\CV_0)\slash \FG$. 
\end{proposition}

\begin{proof}
Note that for any $\SL(\CV_0)$-torsor $\CS l$ over a $C$-scheme $U$ there is an isomorphism of stacks 
\[
[C/\FG]\times_{\rho_*,[C/\SL(\CV_0)],\,\CS l}\,U \isoto\CS l/\FG
\]
where $[C/\FG]$ denotes the stack classifying $\FG$-torsors over $C$-schemes. This isomorphism was already used in diagram \eqref{EqBehrend} for $\GL(\CV)$ instead of $\SL(\CV_0)$. It is given as follows. Let $\CG$ be a $\FG$-torsor over a $U$-scheme $T$ and let $\wt\phi\colon\rho_*\CG\isoto\CS l_T$ be an isomorphism of $\SL(\CV_0)$-torsors over $T$. We choose an \fppf-covering $T'\to T$ and a trivialization $\alpha\colon\FG_{T'}\isoto\CG_{T'}$. We consider the induced trivialization $\rho_*\alpha\colon\SL(\CV_0)_{T'}\isoto\rho_*\CG_{T'}$ and let $1_{T'}\in\SL(\CV_0)(T')$ be the identity section. Then the section $\ol{\wt\phi_{T'}\circ\rho_*\alpha(1_{T'})}\in (\CS l/\FG)(T')$ descends to a section in $(\CS l/\FG)(T)$ which is independent of our choices.

We now consider a morphism $S\to\CG r_{\FG,I,r}$ given by a tuple $((\charsect_i)_{i\in I},\CG,\CV',\alpha',\phi)$ over a scheme $S$. We let $\CS l'$ be the $\SL(\CV_0)$-torsor over $C_S$ associated with the pair $(\CV',\alpha')$ by Remark~\ref{RemRelAffineGrass}, and we let $\wt\phi\colon\rho_*\CG|_{C_S\setminus\cup\Gamma_{\charsect_i}}\isoto\CS l'|_{C_S\setminus\cup\Gamma_{\charsect_i}}$ be the isomorphism induced by $\phi$ on $C_S\setminus\cup_{i\in I}\Gamma_{\charsect_i}$. By the above, the pair $(\CG,\wt\phi)$ defines a section $x\in(\CS l'/\FG)(C_S\setminus\cup_{i\in I}\Gamma_{s_i})$. Since $\CS l'$ is isomorphic to $\SL(\CV_0)$ \fppf-locally on $C_S$ and $\SL(\CV_0)/\FG$ is \mbox{(quasi-)}affine over $C_S$, also $\CS l'\slash \FG$ is \mbox{(quasi-)}affine over $C_S$ by \cite[IV$_2$, Proposition~2.7.1]{EGA}. So by Lemma~\ref{LemmaExtendSect} below there exists a (locally) closed (and quasi-compact) immersion $S'\into S$ such that $x$ extends over $S'$ to a section $x'\in(\CS l'/\FG)(C_{S'})$. Again the section $x'$ corresponds to a $\FG$-torsor $\CG'$ over $C_{S'}$ and an isomorphism $\psi\colon\rho_*\CG'\isoto\CS l'$ of $\SL(\CV_0)$-torsors over $C_{S'}$. The equality of sections $x|_{C_{S'}\setminus\cup\Gamma_{s_i}}=x'|_{C_{S'}\setminus\cup\Gamma_{s_i}}$ induces an isomorphism $\tauGlob\colon\CG|_{C_{S'}\setminus\cup\Gamma_{s_i}}\isoto\CG'|_{C_{S'}\setminus\cup\Gamma_{s_i}}$ with $\rho_*\tauGlob=\psi^{-1}\wt\phi$. Therefore $S'$ represents the fiber product $Hecke_{\FG,(I)}\times_{\CG r_{\FG,I,r}}S$. This implies that the morphism $Hecke_{\FG,(I)}\to\CG r_{\FG,I,r}$ is represented by a (locally) closed (quasi-compact) immersion. 

Finally, Proposition~\ref{RelativGrProj/BunG} implies that $Hecke_{\FG,(I_j)}^{\ul\omega}$ is \mbox{(quasi-)}projective over $C^{I_j}\times_{\BF_q}\scrH^1(C,\FG)$ for all $j=1,\ldots,k$. So the \mbox{(quasi-)}projectivity of $Hecke_{\FG,n}^{\ul\omega}\,\to\,C^I\times_{\BF_q}\scrH^1(C,\FG)$ follows from the fact that
\begin{equation}\label{EqBoundedHeckeProduct}
Hecke_{\FG,\,\ulI}^{\ul\omega}\isoto Hecke_{\FG,(I_1)}^{\ul\omega}\;\underset{pr_1,\scrH^1(C,\FG),pr_0}{\times}\;\ldots\;\underset{pr_1,\scrH^1(C,\FG),pr_0}{\times}\;Hecke_{\FG,(I_k)}^{\ul\omega}
\end{equation}
under the isomorphism \eqref{EqHeckeProduct}.
\end{proof}

\begin{lemma}\label{LemmaExtendSect}
Let $Y$ be a (quasi-)affine scheme over $C_S$. Let $D$ be an effective relative Cartier divisor on $C_S$ over $S$ and set $U:=C_S\setminus D$. Let $x\colon U\to Y$ be a section. Then there is a (locally) closed (quasi-compact) immersion $S'\into S$ such that for any $S$-scheme $T$ the section $x_{_T}\colon U\times_ST\to Y\times_ST$ extends uniquely to a section $x'_{_T}\colon C_T\to Y\times_ST$ if and only if the structure morphism $T\to S$ factors through $S'$.
\end{lemma}

\begin{proof}
We first assume that $Y$ is affine over $C_S$. By the uniqueness the question is local on $S$. So we can assume that $S$ and $T$ are affine and that there is an affine open subset $\Spec A\subset C_S$ containing $D$. Then $Y\times_{C_S}\Spec A$ and $Y\times_{C_S}\Spec A\times_S T$ are affine, say of the form $\Spec B$ and $\Spec B_T$, respectively. Let $\CI\subset A$ be the invertible ideal defining $D$, which is of finite presentation by \cite[I$_{\rm new}$, Corollaire~1.4.3]{EGA}. Then the section $x$ corresponds to an $A$-morphism $x^*\colon B\to \Gamma(\Spec A\setminus D,\CO_{C_S})$ by \cite[I, Proposition~2.2.4]{EGA}. For each element $b\in B$ there is a positive integer $m_b$ such that $\CJ_b:=x^*(b)\cdot \CI^{m_b}\subset A$ by \cite[I, Th\'eor\`eme~1.4.1]{EGA}. The section $x_{_T}$ extends uniquely to a section $x'_{_T}\colon C_T\to Y\times_ST$ if and only if $x^*(B_T)\subset A\otimes_{\CO_S}\CO_T\subset \Gamma(\Spec A\setminus D,\CO_{C_S})\otimes_{\CO_S}\CO_T$. Since $\CI$ is an invertible ideal and $B_T=B\otimes_{\CO_S}\CO_T$, this is the case if and only if $\CJ_b\otimes_{\CO_S}\CO_T\subset \CI^{m_b}\otimes_{\CO_S}\CO_T$ for all $b\in B$, that is, if and only if the image of $\CJ_b\otimes_{\CO_S}\CO_T$ in $(A/\CI^{m_b})\otimes_{\CO_S}\CO_T$ is zero for all $b\in B$. Now $A/\CI^{m_b}$ is the structure sheaf of the relative Cartier divisor $m_b\cdot D$. Since the scheme $m_b\cdot D$ is finite flat and of finite presentation over $S$ the $\CO_S$-module $A/\CI^{m_b}$ is finite locally free by \cite[I$_{\rm new}$, Proposition~6.2.10]{EGA}. Therefore by \cite[I$_{\rm new}$, Lemma~9.7.9.1]{EGA} there exists a closed subscheme $S'_b\into S$ such that the image of $\CJ_b\otimes_{\CO_S}\CO_T$ in $(A/\CI^{m_b})\otimes_{\CO_S}\CO_T$ is zero if and only if the structure morphism $T\to S$ factors through $S'_b$. Taking $S'$ as the scheme theoretic intersection of the $S'_b$ for all $b\in B$ proves the lemma when $Y$ is affine over $C_S$.

If $Y$ is quasi-affine let $\pi\colon Y\to C_S$ be the structure morphism and set $\ol Y:=\ul\Spec_{C_S}\pi_*\CO_{Y}$. Then $Y\into\ol Y$ is a quasi-compact open subscheme by \cite[II, Proposition 5.1.2(b)]{EGA} and hence the complement $\ol Y\setminus Y$ is a finitely presented closed subscheme of $\ol Y$. By the above there is a closed subscheme $S'$ of $S$ over which the section $x\in Y(U)$ extends uniquely to a section $x'\in\ol Y(C_{S'})$, because $\ol Y$ is affine over $C_S$. The preimage $Z:=(\ol Y\setminus Y)\times_{\ol Y\!,\,x'}C_{S'}$ is a finitely presented closed subscheme of $C_{S'}$ and its image in $S'$ is closed and of finite presentation. The latter can be seen as follows. Since $C_{S'}$ and $Z$ are of finite presentation over $S'$ we may compute the image of $Z$ in $S'$ by reducing to the case that $S'$ is noetherian, see \cite[IV$_3$, Proposition~8.9.1]{EGA}, and then every closed subscheme of $S'$ is of finite presentation. Let $S''$ be the complement of this image. Then $S''\into S'$ is a quasi-compact open immersion and for any $S'$-scheme $T$ the structure morphism $T\to S'$ factors through $S''$ if and only if the section $x'_{_T}\colon C_T\to \ol Y\times_ST$ factors through $Y\times_ST$. In other words, $S''$ represents the question whether $x$ extends uniquely to a section $x''\in Y(C_S)$.
\end{proof}

The criterion for $Hecke_{\FG,\,\ulI}^{\ul\omega}$ to be projective over $C^I\times_{\BF_q}\scrH^1(C,\FG)$ given in Proposition~\ref{HeckeisQP} is only sufficient. To give a necessary and sufficient criterion we recall the following

\begin{definition}\label{DefParahoric}
A smooth affine group scheme $\FG$ over $C$ is called a \emph{parahoric} \lang{(\emph{Bruhat-Tits})} \emph{group scheme} if
\begin{enumerate}
\item
 all geometric fibers of $\FG$ are connected and the generic fiber of $\FG$ is reductive over $\BF_q(C)$,
\item 
for any ramification point $\nu$ of $\FG$ (i.e. those points $\nu$ of $C$, for which the fiber above $\nu$ is not reductive) the group scheme $\BP_\nu :=\FG\times_C\Spec A_\nu$ is a parahoric group scheme over $A_\nu$, as defined by Bruhat and Tits \cite[D\'efinition~5.2.6]{BT84}; see also \cite{H-R}.
\end{enumerate}
\end{definition}

Note that by \cite[\S\S\,4.6 and 5.1.9]{BT84} every connected reductive group over $Q$ has (in general many) integral models over $C$ which are parahoric group schemes.

\begin{proposition}\label{HeckeisProper}
The group scheme $\FG$ is parahoric if and only if the representable morphism of stacks $Hecke_{\FG,\,\ulI}^{\ul\omega}\,\to\,C^I\times_{\BF_q}\scrH^1(C,\FG)$ sending $\bigl((\charsect_i)_{i\in I},\,(\CG_j)_{j=0,\ldots,k},\,(\tauGlob_j)_{j=1,\ldots,k}\,\bigr)$ to $\bigl((\charsect_i)_{i\in I},\CG_0\bigr)$ is projective. 
\end{proposition}

\begin{proof}
To prove that the morphism is projective it suffices by Proposition~\ref{HeckeisQP} to show that it is proper. For this we use the valuative criterion for properness \cite[Th\'eor\`eme~7.10]{L-M}. So let $R$ be a discrete valuation ring with fraction field $K$. Let $\bigl((s_i)_i,\CG_0\bigr)$ be an $R$-valued point of $C^I\times_{\BF_q}\scrH^1(C,\FG)$ and let $\bigl((s_i)_i,(\CG_j)_j,(\tau_j)_j\bigr)$ be a $K$-valued point of $Hecke_{\FG,\,\ulI}^{\ul\omega}$ mapping to $\bigl((s_i)_i,\CG_0\bigr)$. We must extend this $K$-valued point to an $R$-valued point. By looking at the decomposition \eqref{EqBoundedHeckeProduct} it suffices to treat the case with $k=1$ and $I=I_1=\{1,\ldots,n\}$. If $s_i=s_{i'}$ for some $i,i'\in I$ we can replace $I$ by $I\setminus\{i'\}$ and $\omega_i$ by $\omega_i+\omega_{i'}$. In this way we can assume that all $s_i\in C(R)$ are pairwise different. Since $C$ is separated over $\BF_q$ this implies that all the induced $s_{i,K}\in C(K)$ are pairwise different. We now consider $Hecke_{\FG,\,\ulI'}^{\ul\omega}$ with $I'=I=\{1,\ldots,n\}$ and $I'_j=\{j\}$ for $j=1,\ldots,n$ and the morphism of schemes
\begin{eqnarray*}
Hecke_{\FG,\,\ulI'}^{\ul\omega}\,\times_{C^I\times_{\BF_q}\scrH^1(C,\FG)}\,\Spec R & \longto & Hecke_{\FG,\,(I)}^{\ul\omega}\,\times_{C^I\times_{\BF_q}\scrH^1(C,\FG)}\,\Spec R\,\\[2mm]
\bigl((\charsect_i)_i,\,(\CG_{j})_{j=0,\ldots,n},\,(\tau_{j})_{j=1,\ldots,n}\bigr) & \longmapsto & \bigl((\charsect_i)_i,\,(\CG_0,\CG_n),\,\tau:=\tau_n\circ\ldots\circ\tau_1\bigr).
\end{eqnarray*}
Since all $s_{i,K}\in C(K)$ are pairwise different, our $K$-valued point $(\CG_0,\CG_n,\tau)$ of $Hecke_{\FG,\,(I)}^{\ul\omega}$ lifts to a unique $K$-valued point of $Hecke_{\FG,\,\ulI'}^{\ul\omega}$ by letting $\CG_{j}$ be the $\FG$-bundle that coincides with $\CG_0$ on $C_K\setminus\bigcup_{i\le j}\Gamma_{\charsect_{i,K}}$ and with $\CG_n$ on $C_K\setminus\bigcup_{i> j}\Gamma_{\charsect_{i,K}}$. To prove the properness of $Hecke_{\FG,\,(I)}^{\ul\omega}$ it therefore suffices to prove the properness of $Hecke_{\FG,\,\ulI'}^{\ul\omega}$, and by looking at the decomposition \eqref{EqBoundedHeckeProduct} again, it suffices to prove the properness of $Hecke_{\FG,(I'_j)}^{\ul\omega}$ which has only one leg. By \cite[Proposition~2.0.9]{EsmailSomayeh_LocModel} (respectively \cite[Lemma~4.1]{Var} for constant split reductive $\FG$) the base change $Hecke_{\FG,(I'_j)}^{\ul\omega}\times_{C\times_{\BF_q}\scrH^1(C,\FG)}\,\Spec R$ is isomorphic to the fiber over $s_j\in C(R)$ of the \emph{BD-Grassmannian} $\Gr(\FG,C)$ from \cite[Definition~1.4]{Richarz13}. The latter is proper over $C$ by \cite[Theorem~1.18]{Richarz13} if $\FG$ is parahoric.

Conversely, if the morphism of stacks is projective we fix a point $x\in C(L)$ for $L=\BF_q^\alg$, and we consider the fiber of $Hecke_{\FG,\,\ulI}^{\ul\omega}$ over the $L$-valued point $\bigl((\charsect_i)_i,\FG_{C_L}\bigr)$ of $C^I\times_{\BF_q}\scrH^1(C,\FG)$ where all the $\charsect_i$ are equal to $x$ and $\FG_{C_L}$ is the trivial $\FG$-bundle. Then $(\CG_0=\ldots=\CG_{k-1}=\FG_{C_L},\tau_1=\ldots=\tau_{k-1}=\id)$ is a closed $L$-valued point of the first $k-1$ factors of the decomposition \eqref{EqBoundedHeckeProduct}. It defines a closed subscheme isomorphic to the fiber of $Hecke_{\FG,(I_k)}^{\ul\omega}$ over $\bigl((\charsect_i)_i,\FG_{C_L}\bigr)$. Therefore the latter is proper over $\Spec L$. The limit of this fiber for varying $\ul\omega$ is ind-proper and equals the fiber of the BD-Grassmannian $\Gr(\FG,C)$ over $x$. It now follows from \cite[Theorem~1.18]{Richarz13} that $\FG$ is parahoric.
\end{proof}

\begin{definition}\label{DefNablaOmegaH}
We denote by $\nabla_{\ulI}^{\ul\omega}\scrH^1(C,\FG)$ the pull back of $Hecke_{\FG,\,\ulI}^{\ul\omega}$ under the morphism of stacks $\nabla_\ulI\scrH^1(C,\FG) \to Hecke_{\FG,\,\ulI}$. We also denote by $\nabla_{\ulI}^{\ul\omega}\scrH^1_D(C,\FG)$ the pull back of $\nabla_\ulI^{\ul\omega}\scrH^1(C,\FG)$ under the (forgetful) morphism $\nabla_\ulI\scrH^1_D(C,\FG) \to \nabla_\ulI\scrH^1(C,\FG)$.
\end{definition}

We want to prove that the stack $\nabla_\ulI\scrH^1_D(C,\FG)$ is an ind-algebraic stack. We do not intend to give a general treatment of ind-algebraic stacks here. We just make the following tentative 

\newcommand{\SSS}{S}
\newcommand{\TTT}{T}
\begin{definition}\label{DefIndAlgStack}
Let $\TTT$ be a scheme.
\begin{enumerate}
\item \label{DefIndAlgStack_A}
By an \emph{inductive system of Deligne-Mumford stacks over $\TTT$} we mean an inductive system $(\CC_a, i_{ab})$ indexed by a countable directed set $A$, such that each $\CC_a$ is a Deligne-Mumford stack over $\TTT$ and $i_{ab}\colon\CC_a\into\CC_b$ is a closed immersion of stacks for all $a\le b$ in $A$.
\item \label{DefIndAlgStack_B}
A stack $\CC$ over $\TTT$ is an \emph{ind-DM-stack over $\TTT$} (or \emph{ind-algebraic Deligne-Mumford stack over $\TTT$}) if there is an inductive system of Deligne-Mumford stacks $(\CC_a, i_{ab})$ over $\TTT$ together with morphisms $j_a:\CC_a\to\CC$ satisfying $j_b\circ i_{ab}=j_a$ for all $a\le b$, such that for all quasi-compact $\TTT$-schemes $S$ and all objects $c\in \CC(S)$ there is an $a\in A$, an object $c_a\in\CC_a(S)$ and an isomorphism $j_a(c_a)\cong c$ in $\CC$. In this case we say that \emph{$\CC$ is the inductive limit of $(\CC_a,i_{ab})$} and we write $\CC=\dirlim\CC_a$.
\item \label{DefIndAlgStack_C}
If in \ref{DefIndAlgStack_B} all $\CC_a$ are locally of finite type (resp.\ separated) over $\TTT$ we say that $\CC$ is \emph{locally of ind-finite type} (resp.\ \emph{ind-separated}) \emph{over $\TTT$}.
\end{enumerate}
\end{definition}

\begin{theorem}\label{nHisArtin}
Let $\FG$ be a flat affine group scheme of finite type over the curve $C$ and let $D$ be a proper closed subscheme of $C$. Then the stack $\nabla_\ulI^{\ul\omega}\scrH^1_D(C,\FG)$ is a Deligne-Mumford stack locally of finite type and separated over $(C\setminus D)^I$. It is relatively representable over $(C\setminus D)^I\times_{\BF_q}\scrH^1(C,\FG)$ by a separated morphism of finite type of algebraic spaces. In particular $\nabla_\ulI\scrH^1_D(C,\FG)=\dirlim\nabla_\ulI^{\ul\omega}\scrH^1_D(C,\FG)$ is an ind-DM-stack over $(C\setminus D)^I$ which is ind-separated and locally of ind-finite type. The forgetful morphism $F\colon\nabla_\ulI\scrH^1_D(C,\FG)\rightarrow \nabla_\ulI\scrH^1(C,\FG)\times_{C^I}(C\setminus D)^I$ is relatively representable by an affine \'etale morphism of schemes and above its image it is a torsor under the finite group $\FG(D)$. If $\FG\times_CD\to D$ is smooth, then $F$ is finite. If the fibers of $\FG$ over the points in $D$ are smooth and connected, then $F$ is surjective.
\end{theorem}

\noindent
{\itshape Remark.} If the fibers of $\FG$ over the points in $D$ are not smooth and connected, $F$ may fail to be surjective; see Example~\ref{ExFNotSurj}.

\begin{proof}[Proof of Theorem~\ref{nHisArtin}]
We first show that the morphism 
\[
F\colon\nabla_\ulI\scrH^1_D(C,\FG)\;\longto\; \nabla_\ulI\scrH^1(C,\FG)\times_{C^I}(C\setminus D)^I
\]
is relatively representable by an affine \'etale morphism of schemes, which is in addition finite if $\FG\times_CD\to D$ is smooth. Let $S$ be a scheme and consider a morphism $S\to\nabla_\ulI\scrH^1(C,\FG)\times_{C^I}(C\setminus D)^I$ given by a global $\FG$-shtuka $\bigl((\charsect_i)_{i\in I},(\CG_j)_j,(\tauGlob_j)_j\bigr)$. For an $S$-scheme $T$ the $T$-valued points of the fiber product $\nabla_\ulI\scrH^1_D(C,\FG)\times_{\nabla_\ulI\scrH^1(C,\FG)}S$ equals the set of isomorphisms of $\FG$-torsors $(\psi_j\colon \CG_j\times_{C_S}{D_T}\isoto \FG\times_C D_T)_j$ satisfying $\psi_j\circ\tauGlob_j|_{D_T}=\psi_{j-1}$ for $j=1,\ldots,k$ and $\psi_0\circ\tauGlob_0|_{D_T}=\sigma^*(\psi_k)$. Since the characteristic morphism $S\to C^I$ factors through $(C\setminus D)^I$, the isomorphism $\tauGlob:=\tauGlob_k\circ\ldots\circ\tauGlob_0$ induces an isomorphism $\tauGlob|_{D_S}\colon\sigma^*(\CG_k\times_{C_S}D_S)\isoto\CG_k\times_{C_S}D_S$. In particular $\psi_k\circ\tau|_{D_T}=\sigma^*\psi_k$ and the other $\psi_j$ are uniquely determined by $\psi_k$. Let $\pi_S\colon D_S\to S$ be the structure morphism. Then $\psi_k$ is a section over $T$ of the the sheaf $Y:=\pi_{S*}\ul\Isom(\CG_k\times_{C_S}{D_S}\,,\,\FG\times_C D_S)$, which is representable by an affine scheme of finite presentation over $S$ by \cite[Corollary~4.4.2]{Beh}. The condition $\psi_k\circ\tau|_{D_T}=\sigma^*\psi_k$, and thus the fiber product $\nabla_\ulI\scrH^1_D(C,\FG)\times_{\nabla_\ulI\scrH^1(C,\FG)}S$ is representable by the scheme $\equi(\id_Y\circ\tau,\sigma_Y\colon Y\rightrightarrows Y)$, see Definition~\ref{ker}, which is a closed subscheme of finite presentation of $Y$, because $Y$ is affine and of finite presentation over $S$. This shows that $F$ is relatively representable by an affine morphism of schemes of finite presentations.

To show that $F$ is \'etale let $T$ be an affine $S$-scheme and let $j\colon\bar T\into T$ be a closed subscheme defined by an ideal $J$ with $J^2=(0)$. Then the morphism $\sigma_T$ factors through $j\colon \bar{T}\to T$ as
\begin{equation}\label{EqSigmaFactors}
\sigma_T = j \circ \sigma'\colon  T \xrightarrow{\:\,\sigma'\,} \bar{T} \xrightarrow{\;\; j\:\,} T\,,
\end{equation}
where $\sigma'$ is the identity on the underlying topological space $|\bar T|=|T|$ and on the structure sheaf this factorization is given by
\begin{eqnarray*}
\TS\CO_T \enspace \xrightarrow{\enspace j^\ast\;} & \CO_{\bar T} & \TS\xrightarrow{\enspace \sigma'{}^\ast\;} \enspace \CO_T\\
\TS b\quad \mapsto\;\: & b \mod J& \TS\;\:\mapsto\quad b^q\,.
\end{eqnarray*}
Note that $\sigma_{\bar T}=\sigma'\circ j\colon \bar T\to\bar T$. Any $\bar\psi_k$ defined over $\bar T$ with $\bar\psi_k\circ\tau|_{D_{\bar T}}=\sigma^*\bar\psi_k$ lifts uniquely to $\psi_k:=\sigma'{}^*\bar\psi_k\circ\tau|_{D_T}^{-1}$ with $j^*\psi_k=j^*\sigma'{}^*\bar\psi_k\circ j^*\tau|_{D_T}^{-1}=\sigma_{\bar T}^*\bar\psi_k\circ\tau|_{D_{\bar T}}^{-1}=\bar\psi_k$ and $\psi_k\circ\tau|_{D_T}=\sigma'{}^*\bar\psi_k=\sigma'{}^*j^*\psi_k=\sigma^*\psi_k$. This shows that $F$ is \'etale, and hence faithfully flat over its image. 

If $\FG_D:=\FG\times_CD \to D$ is smooth we show that $\nabla_\ulI\scrH^1_D(C,\FG)\times_{\nabla_\ulI\scrH^1(C,\FG)}S$ is finite over $S$. By \cite[IV$_2$, Proposition~2.7.1]{EGA} this may be checked over an \'etale covering $\wt S\to S$. By \fpqc-descent \cite[IV$_4$, Corollaire 17.7.3]{EGA} the scheme $\CG_k\times_{C_S}D_S$ is smooth over $D_S$ and the Weil restriction $\Res_{D_S/S}(\CG_k\times_{C_S}D_S)$ is a smooth $S$-scheme by \cite[\S\,7.6, Theorem~4 and Proposition~5]{BLR}. Therefore, there exists an \'etale covering $\wt S\to S$ and a section in $\Res_{D_S/S}(\CG_k\times_{C_S}D_S)(\wt S)=(\CG_k\times_{C_S}D_S)(D_{\wt S})$. This section yields a trivialization $\alpha\colon\FG\times_CD_{\wt S}\isoto\CG_k\times_{C_S}D_{\wt S}$. The Weil restriction $\Res_{D/\BF_q}\FG_D$ is a smooth affine group scheme over $\BF_q$ by \cite[\S\,7.6, Theorem~4 and Proposition~5]{BLR}. Therefore, the morphism $-\id+\Frob_q\colon\Res_{D/\BF_q}\FG_D\to\Res_{D/\BF_q}\FG_D$ given by $g\mapsto g^{-1}\cdot\sigma(g)$ is finite \'etale. The composition $\alpha^{-1}\circ\tau|_{D_{\wt S}}\circ\sigma^*\alpha$ equals multiplication with an element $\tilde g\in\FG(D_{\wt S})=(\Res_{D/\BF_q}\FG_D)(\wt S)$. There is an isomorphism 
\[
\xymatrix @C+1pc @R=0.5pc {
\nabla_\ulI\scrH^1_D(C,\FG)\times_{\nabla_\ulI\scrH^1(C,\FG)}\wt S \ar[r]^{\TS\sim\qquad\quad} & \Res_{D/\BF_q}\FG_D\underset{\TS-\id+\Frob_q,\Res_{D/\BF_q}\FG_D}{\times}\wt S\,,\\
\psi_k \text{ with }\psi_k\circ\tau|_{D_{\wt S}}=\sigma^*\psi_k\ar@{|->}[r] & g:=\psi_k\circ\alpha\text{ with }g^{-1}\sigma^*g=\tilde g\qquad\qquad\quad
}
\]
This proves that $F$ is finite.

If in addition the fibers of $\FG$ over the points of $D$ are connected (and smooth), the Weil restriction $\Res_{D/\BF_q}\FG_D$ is connected by \cite[Proposition~A.5.9]{CGP}. Then Lang's Theorem \cite[Corollary on p.~557]{Lang} for $\Res_{D/\BF_q}\FG_D$ shows that $-\id+\Frob_q$ is surjective, and hence $F$ is surjective.

We show that $F\colon\nabla_\ulI\scrH^1_D(C,\FG)\rightarrow \nabla_\ulI\scrH^1(C,\FG)\times_{C^I}(C\setminus D)^I$ is a torsor over its image under the finite group $\FG(D)$, which we consider as a finite \'etale group scheme $\ul{\FG(D)}$ over $\Spec\BF_q$. Indeed, $g\in\FG(D)$ acts on $\nabla_\ulI\scrH^1_D(C,\FG)$ by sending $\psi_j$ to $g\circ\psi_j$. Here we observe that $g$ is invariant under $\sigma$, and hence $\sigma^*(g\circ\psi_k)=g\circ\sigma^*\psi_k$. The product $\nabla_\ulI\scrH^1_D(C,\FG)\times_{\nabla_\ulI\scrH^1(C,\FG)}\nabla_\ulI\scrH^1_D(C,\FG)$ classifies data $\bigl((\charsect_i)_{i\in I},(\CG_j)_j,(\tauGlob_j)_j,(\psi_j)_j,(\psi'_j)_j\bigr)$ over $\BF_q$-schemes $S$ where $(\psi_j)_j$ and $(\psi'_j)_j$ are two $D$-level structures on the global $\FG$-shtuka $\bigl((\charsect_i)_{i\in I},(\CG_j)_j,(\tauGlob_j)_j\bigr)$. This means that $\psi_j,\psi'_j\colon\CG_j\times_{C_S}D_S\isoto\FG\times_CD_S$ are isomorphisms of torsors with $\psi_j\circ\tauGlob_j|_{D_S}=\psi_{j-1}$ for $j=1,\ldots,k$ and $\psi_0\circ\tauGlob_0|_{D_S}=\sigma^*(\psi_k)$ and similarly for $\psi'_j$. In particular, $\psi'_k\psi_k^{-1}$ is an automorphism of $\FG\times_C D_S$ given by left translation with an element $g\in\FG(D_S)$. Since $\psi'_k\psi_k^{-1}=\ldots=\psi'_0\psi_0^{-1}=\sigma^*(\psi'_k\psi_k^{-1})$ we have $g=\sigma^*(g)\in\ul{\FG(D)}(S)$. This shows that $\nabla_\ulI\scrH^1_D(C,\FG)\times_{\nabla_\ulI\scrH^1(C,\FG)}\nabla_\ulI\scrH^1_D(C,\FG)\;\cong\; \ul{\FG(D)}\times_{\BF_q}\nabla_\ulI\scrH^1_D(C,\FG)$ and $\nabla_\ulI\scrH^1_D(C,\FG)\rightarrow \nabla_\ulI\scrH^1(C,\FG)\times_{C^I}(C\setminus D)^I$ is a $\FG(D)$-torsor over its image. 

To prove that the stack $\nabla_\ulI^{\ul\omega}\scrH^1(C,\FG)$ is an Artin stack locally of finite type over $C^I$ we observe that its defining morphism \mbox{$\nabla_\ulI^{\ul\omega}\scrH^1(C,\FG)\to Hecke_{\FG,\,\ulI}^{\ul\omega}$} is representable, separated and of finite type, because it arises by base change from the diagonal morphism $\Delta_{\scrH^1(C,\FG)\,/\BF_q}\colon\scrH^1(C,\FG)\to\scrH^1(C,\FG)\times_{\BF_q}\scrH^1(C,\FG)$. The latter is representable, separated and of finite type by \cite[Lemma~4.2]{L-M}, because $\scrH^1(C,\FG)$ is algebraic (Theorem~\ref{Bun_G}). Note that in general $\Delta_{\scrH^1(C,\FG)\,/\BF_q}$, and hence also \mbox{$\nabla_\ulI^{\ul\omega}\scrH^1(C,\FG)\to Hecke_{\FG,\,\ulI}^{\ul\omega}$} is not a closed immersion, and not even proper. Since $Hecke_{\FG,\,\ulI}^{\ul\omega}$ is quasi-projective and quasi-compact over $C^I\times_{\BF_q}\scrH^1(C,\FG)$ by Proposition~\ref{HeckeisQP}, it follows that $\nabla_\ulI^{\ul\omega}\scrH^1_D(C,\FG)$ is relatively representable over $(C\setminus D)^I\times_{\BF_q}\scrH^1(C,\FG)$ by a separated morphism of finite type of algebraic spaces. In particular it is an Artin stack locally of finite type over $(C\setminus D)^I$. 

We prove that $\nabla_\ulI^{\ul\omega}\scrH^1(C,\FG)$ is Deligne-Mumford and separated over $C^I$. This means that the diagonal morphism $\Delta\colon\nabla_\ulI^{\ul\omega}\scrH^1(C,\FG)\longto\nabla_\ulI^{\ul\omega}\scrH^1(C,\FG)\times_{C^I}\nabla_\ulI^{\ul\omega}\scrH^1(C,\FG)$ is unramified and proper; see \cite[Th\'eor\`eme~8.1 and D\'efinition~7.6]{L-M}. By the algebraicity of $\nabla_\ulI^{\ul\omega}\scrH^1(C,\FG)$ over $C^I$ the diagonal $\Delta$ is already relatively representable by algebraic spaces, separated and of finite type; see \cite[Lemme~4.2]{L-M}. Consider an $S$-valued point of \mbox{$\nabla_\ulI^{\ul\omega}\scrH^1(C,\FG)\times_{C^I}\nabla_\ulI^{\ul\omega}\scrH^1(C,\FG)$} given by two global $\FG$-shtukas $\ul\CG$ and $\ul\CG'$ over a $C^I$-scheme $S$. The base change of $\Delta$ to $S$ is represented by the algebraic space $\ul\Isom_S(\ul\CG,\ul\CG')$. Being unramified means that whenever $S$ is affine and $j\colon\bar S\into S$ is a closed subscheme defined by an ideal $J$ with $J^2=(0)$, then every isomorphism $\bar f$ between $\ul\CG_{\bar S}$ and $\ul\CG'_{\bar S}$ over $\bar S$ comes from at most one isomorphism $f$ between $\ul\CG$ and $\ul\CG'$ over $S$. Such an isomorphism is a tuple $f=(f_0,\ldots,f_k)$ of isomorphisms $f_j\colon\CG_j\isoto\CG_j'$ compatible with the maps $\tau_j$ of $\ul\CG$ and $\tau'_j$ of $\ul\CG'$. So assume that there are two isomorphisms $f,f'$ over $S$ with $j^*f=\bar f=j^*f'$. Since $J^2 = (0)$ the morphism $\sigma_S$ factors through $j\colon \bar{S}\to S$ as $\sigma_S=j\circ\sigma'$ like in \eqref{EqSigmaFactors}. From the diagram
$$
\CD
\CG_k|_{C_S\setminus\cup\Gamma_{\charsect_i}} @>{f_k}>\cong>\CG_k'|_{C_S\setminus\cup\Gamma_{\charsect_i}}\\
@A{\tau_k\circ\ldots\circ\tau_0}A{\cong}A @A{\tau'_k\circ\ldots\circ\tau'_0}A{\cong}A\\
\sigma_S^*\CG_k|_{C_S\setminus\cup\Gamma_{\charsect_i}} @>{\sigma'^*(j^*f_k)}>\cong>\sigma_S^*\CG_k'|_{C_S\setminus\cup\Gamma_{\charsect_i}}
\endCD
$$
we obtain that $f_k$ and $f_k'$ coincide on $C_S\setminus \bigcup_{i\in I}\Gamma_{\charsect_i}$ and likewise $f_j=(\tau'_k\circ\ldots\circ\tau'_{j+1})^{-1}\circ f_k\circ\tau_k\circ\ldots\circ\tau_{j+1}$ and $f'_j$ coincide on $C_S\setminus \bigcup_{i\in I}\Gamma_{\charsect_i}$ for all $j$. Since $\CG_j'$ is separated over $C_S$ and $\bigcup_{i\in I}\Gamma_{\charsect_i}$ is a Cartier divisor on $C_S$ we have $f_j=f_j'$ for all $j$ on all of $C_S$ by \cite[Lemma~3.11]{BreutmannFunctoriality}. This shows that the diagonal $\Delta$ is unramified and $\nabla_\ulI^{\ul\omega}\scrH^1(C,\FG)$ is Deligne-Mumford.

To prove that $\nabla_\ulI^{\ul\omega}\scrH^1(C,\FG)$ is separated over $C^I$ we choose a representation $\rho:\FG\into\GL(\CV)$ as in Proposition~\ref{PropGlobalRep}\ref{PropGlobalRep_A} such that the pair $\bigl(\FG,\GL(\CV)\bigr)$ satisfies \eqref{EqCondQuotient},  and consider the commutative diagram 
\[
\xymatrix {
\nabla_\ulI^{\ul\omega}\scrH^1(C,\FG) \ar[r] \ar[d] & \nabla_\ulI^{\ul\omega}\scrH^1(C,\GL(\CV)) \ar[d] \\
C^I\times_{\BF_q}\scrH^1(C,\FG) \ar[r] & C^I\times_{\BF_q}\scrH^1(C,\GL(\CV))\,.
}
\]
The vertical morphisms and the lower horizontal morphism are separated by what we have proved above and by Theorem~\ref{r-q-a}. By \cite[Remarque 7.8.1(3)]{L-M} also the upper horizontal morphism is separated, and it suffices to prove the statement for $\FG=\GL(\CV)$. We use the equivalence $\scrH^1(C,\GL(\CV))\isoto\Vect_C^r$ from Remark~\ref{RemVect} under which $\ul\CG$ corresponds to $(M_j,\tauGlob_j)_{j=0,\ldots,k}$ consisting of locally free sheaves $M_j$ of rank $r$ on $C_S$ and isomorphisms $\tauGlob_j\colon M_{j-1}\isoto M_j$ on $C_S\setminus\bigcup_{i\in I_j}\Gamma_{\charsect_i}$ for all $j=1,\ldots,k$ and an isomorphism $\tauGlob_0\colon\sigma^*M_k\isoto M_0$ on $C_S$. Similarly $\ul\CG'$ corresponds to $(M'_j,\tauGlob'_j)_j$. We now use the valuative criterion for properness \cite[Th\'eor\`eme~7.10]{L-M} to show that $\ul\Isom_S\bigl((M_j,\tauGlob_j)_j,(M'_j,\tauGlob'_j)_j\bigr)$ is proper over $S$. We may assume that $S$ is the spectrum of a discrete valuation ring $R$ with fraction field $K$, uniformizer $\pi$, and residue field $k$. Let $f_j\colon M_j\otimes_RK\isoto M'_j\otimes_RK$ with $\tauGlob'_j\circ f_{j-1}=f_j\circ\tauGlob_j$ for all $j$ and $\tauGlob'_0\circ\sigma^*f_k=f_0\circ\tauGlob_0$ be isomorphisms over $C_K$. We view $f_j$ and its inverse as sections over $C_K$ of the locally free sheaves $M'_j\otimes M_j\dual$ and $M_j\otimes(M'_j)\dual$. We have to show that both extend uniquely to $C_R$. The local ring $\CO_{C_R,\eta}$ at the generic point $\eta$ of the special fiber $C_k$ is a discrete valuation ring with uniformizer $\pi$. It suffices to extend $f_j$ to $\CO_{C_R,\eta}$, because then $f_j$ is defined on an open set whose complement has codimension $2$, and therefore $f_j$ extends to all of $C_R$ by \cite[Discussion after Corollary~11.4]{Eisenbud}. To extend $f_k$ to $\CO_{C_R,\eta}$ we choose bases of $M_{k,\eta}$ and $M'_{k,\eta}$ and write $f_k = \pi^m\tilde f_k$ with an $r\times r$ matrix $\tilde f_k\in(\CO_{C_R,\eta})^{r\times r}\setminus(\pi\CO_{C_R,\eta})^{r\times r}$. Now the equation $\pi^m\tilde f_k \tauGlob_k\ldots\tauGlob_0 = \tauGlob'_k\ldots\tauGlob'_0 \,\sigma^*f_k = \pi^{qm} \tauGlob'_k\ldots\tauGlob'_0 \,\sigma^*\tilde f_k$ shows that $m$ must be zero, because all $\tauGlob_j$ and $\tauGlob'_j$ belong to $\GL_r(\CO_{C_R,\eta})$. Therefore $f_k$ and with it all $f_j$ extend uniquely to $C_R$ and this proves that $\nabla_\ulI^{\ul\omega}\scrH^1(C,\FG)$ is separated over $C^I$.

To prove that $\nabla_\ulI\scrH^1_D(C,\FG)$ is an ind-DM-stack, we consider the countable set of tuples $\ul\omega=(\omega_i)_{i\in I}$ of coweights of $\SL_r$ which are dominant with respect to the Borel subgroup of upper triangular matrices. We make this set into a directed set by equipping it with the partial order $\ul\omega\preceq\ul\omega'$ whenever $\omega_i\preceq\omega'_i$ for all $i$ in the Bruhat order, that is $\omega'_i-\omega_i$ is a positive linear combination of positive coroots of $\SL_r$. When $\ul\omega$ runs through this directed set, the $\nabla_\ulI^{\ul\omega}\scrH^1_D(C,\FG)$ form an inductive system of algebraic stacks which are separated and locally of finite type over $(C\setminus D)^I$. Indeed, we have to show that $\nabla_\ulI^{\ul\omega}\scrH^1_D(C,\FG)\to\nabla_\ulI^{\ul\omega'}\scrH^1_D(C,\FG)$ is a closed immersion for $\ul\omega\preceq\ul\omega'$. This follows from the fact that condition~\eqref{EqBounded1} for $\ul\omega$ on $\Gr^{\ul\omega'}_{\FG,I_j,r}$ is equivalent to the vanishing of the image of $\bigwedge^\ell_{C_S}\phi\bigl(\CV(\rho_*\CG)\bigr)$ inside the finite locally free sheaf 
\[
\Bigl(\bigwedge^\ell_{C_S} \CV'\Bigr)\bigl(\sum_{i\in I_j}(-\omega'_{i,r-\ell+1}-\ldots-\omega'_{i,r})\!\cdot\!\Gamma_{s_i}\bigr)\Big/\Bigl(\bigwedge^\ell_{C_S} \CV'\Bigr)\bigl(\sum_{i\in I_j}(-\omega_{i,r-\ell+1}-\ldots-\omega_{i,r})\!\cdot\!\Gamma_{s_i}\bigr)
\]
on $\Gr^{\ul\omega'}_{\FG,I_j,r}$. Therefore condition~\eqref{EqBounded1} defines a closed immersion $\Gr^{\ul\omega}_{\FG,I_j,r}\into\Gr^{\ul\omega'}_{\FG,I_j,r}$ by \cite[I$_{\rm new}$, Lemma~9.7.9.1]{EGA} for all $j$. Therefore, its base change $\nabla_\ulI^{\ul\omega}\scrH^1_D(C,\FG)\to\nabla_\ulI^{\ul\omega'}\scrH^1_D(C,\FG)$ is a closed immersion.

We prove that $\nabla_\ulI\scrH^1_D(C,\FG)=\dirlim\nabla_\ulI^{\ul\omega}\scrH^1_D(C,\FG)$. For this purpose let $\bigl((\charsect_i)_i,(\CG_j)_j,(\tauGlob_j)_j\bigr)$ be a global $\FG$-shtuka in $\nabla_\ulI\scrH^1_D(C,\FG)(S)$ for a quasi-compact scheme $S$ and let $\bigl((\charsect_i)_{i\in I_j},\CG_{j-1},\CV'_j,\alpha'_j,\phi_j\bigr)\in\CG r_{\FG,I_j,r}(S)$ be the induced $S$-valued point of $\CG r_{\FG,I_j,r}$ for all $j$. Since $S$ is quasi-compact there is for each $i\in I_j$ an integer $N_i\ge0$ with $\phi_j(\CV(\rho_*\CG_{j-1}))\subset\CV'_j(\sum_{i\in I_j}N_i\!\cdot\!\Gamma_{\charsect_i})$ by \cite[I, Th\'eor\`eme~1.4.1]{EGA}. We set $\omega_{i,\ell}:=-N_i$ for all $i\in I$ and $\ell=2,\ldots,r$ and $\omega_{i,1}:=(r-1)N_i$. Then the tuples $\bigl((\charsect_i)_{i\in I_j},\CG_{j-1},\CV'_j,\alpha'_j,\phi_j\bigr)$ satisfy condition~\eqref{EqBounded1}. It follows that $\bigl((\charsect_i)_i,(\CG_j)_j,(\tauGlob_j)_j\bigr)\in\nabla_\ulI^{\ul\omega}\scrH^1_D(C,\FG)(S)$ and the theorem is proved.
\end{proof}

In the course of the proof we have established the following corollary which we formulate for later reference.

\begin{corollary}\label{CorAutFinite}
Let $\ul\CG$ and $\ul\CG'$ be global $\FG$-shtukas of the same characteristic over an $\BF_q$-scheme $S$. Then the sheaf of sets on $S_\fpqc$ given by $\ul\Isom_S(\ul\CG,\ul\CG')\colon T\mapsto\Isom_T(\ul\CG_T,\ul\CG'_T)$ is representable by a scheme which is finite and unramified over $S$. In particular the group of automorphisms $\Aut_S(\ul\CG)$ of $\ul\CG$ is finite.
\end{corollary}

\begin{proof} We have seen in the proof of Theorem~\ref{nHisArtin} that $\ul\Isom_S(\ul\CG,\ul\CG')$ is an algebraic space which is unramified and proper over $S$. In particular it is finite and affine over $S$ and hence a scheme; compare \cite[Lemma~4.2]{L-M}.
\end{proof}

\begin{remark}\label{RemnHNotAScheme}
Note that the corollary says that the diagonal morphism 
\[
\Delta\colon\nabla_\ulI^{\ul\omega}\scrH^1(C,\FG)\longto\nabla_\ulI^{\ul\omega}\scrH^1(C,\FG)\times_{C^I}\nabla_\ulI^{\ul\omega}\scrH^1(C,\FG)
\]
is representable by a finite and unramified morphism of schemes. However, this does not imply that for large enough level $D$ the stack $\nabla_\ulI^{\ul\omega}\scrH^1_D(C,\FG)$ is a scheme. In general it is not even an algebraic space. By \cite[Proposition~4.4]{L-M} the latter holds if and only if the diagonal $\Delta$ is a closed immersion on $\nabla_\ulI^{\ul\omega}\scrH^1_D(C,\FG)$. This is not the case because $\nabla_\ulI^{\ul\omega}\scrH^1(C,\FG)$ is not quasi-compact and on it the fibers of $\Delta$ may grow arbitrarily. In particular, the order of the finite group $\Aut_S(\ul\CG)$ from the corollary is not bounded on $\nabla_\ulI^{\ul\omega}\scrH^1(C,\FG)$.

For example let $C=\BP^1$, let $\infty\in C(\BF_q)$ with $C\setminus\{\infty\}=\Spec \BF_q[t]$, and let $0$ be the point where $t=0$. Let $\FG=\GL_2$ and consider the global $\FG$-shtuka over $S=\Spec\BF_q$ with $I=I_1=\{1,2\}$ and legs $\charsect_1\colon S\to\{0\}\subset C$ and $\charsect_2\colon S\to\{\infty\}\subset C$, which is given by the vector bundle $M=M_1=\CO_{\BP^1}(m)\oplus\CO_{\BP^1}(0)$ and the Frobenius $\tau=\tau_1\circ\tau_0=t\cdot\id_M\colon\sigma^*M=M\isoto M$ over $C_S\setminus\{0,\infty\}$. For any fixed level $D=\Spec\BF_q[t]/(a)$ with $a\in\BF_q[t]\setminus(t)$ we can choose $m\ge\deg_t(a)$ and the automorphism $f=\left(\begin{smallmatrix} 1 & a\\ 0 & 1\end{smallmatrix}\right)$ of $(M,\tau)$ which also is compatible with any $D$-level structure $\psi$ on $(M,\tau)$ over a field extension of $\BF_q$. This shows that there is no level $D$ that is able to kill the automorphisms of every global $\FG$-shtuka.

Nevertheless $\nabla_\ulI^{\ul\omega}\scrH^1_D(C,\FG)$ is locally the quotient of a scheme separated and of finite type over $\BF_q$ by a finite group whose order, however, grows arbitrarily as one moves on $\nabla_\ulI^{\ul\omega}\scrH^1_D(C,\FG)$. Indeed, by Remark~\ref{RemCoveringH1} the stack $\scrH^1(C,\FG)$ can be covered by connected open substacks $\scrH^1(C,\FG)^{\le\mu}$ of finite type over $\BF_q$, and the preimage $\nabla_\ulI^{\ul\omega}\scrH^1_D(C,\FG)^{\le\mu}$ of $\scrH^1(C,\FG)^{\le\mu}$ in $\nabla_\ulI^{\ul\omega}\scrH^1_D(C,\FG)$ is separated and of finite type over $(C\setminus D)^I$ by Theorem~\ref{nHisArtin}. On this preimage the automorphism group is bounded and can be killed by suitably enhancing the level $D$ to some $D'$. Then $\nabla_\ulI^{\ul\omega}\scrH^1_D(C,\FG)^{\le\mu}$ is the quotient by the finite group $\ker(\FG(D')\to\FG(D))$ of $\nabla_\ulI^{\ul\omega}\scrH^1_{D'}(C,\FG)^{\le\mu}$. The latter is an algebraic space separated and of finite type over $\BF_q$ by \cite[Corollaire~8.1.1]{L-M}. It is even a quasi-projective scheme over $(C\setminus D')^I$, because it is constructed as in the proof of Theorem~\ref{nHisArtin} from the quasi-projective scheme $\scrH^1_{D'}(C,\FG)^{\le\mu}$; see Remark~\ref{RemCoveringH1}.
\end{remark}

\begin{example}\label{ExFNotSurj}
We show that the morphism $F\colon\nabla_\ulI\scrH^1_D(C,\FG)\rightarrow \nabla_\ulI\scrH^1(C,\FG)\times_{C^I}(C\setminus D)^I$ does not have to be surjective for non-connected groups $\FG$. For example let $G_0=\bigl\{\left(\begin{smallmatrix} * & 0 \\ 0 & * \end{smallmatrix}\right)\bigr\}\coprod\bigl\{\left(\begin{smallmatrix} 0 & * \\ * & 0 \end{smallmatrix}\right)\bigr\}\subset\GL_2$ be the normalizer of the diagonal torus and let $\FG:=G_0\times_{\BF_q}C$. Let $C=\BP^1$ and consider two points $\infty\in\BP^1$ and $0=\Var(t)\in\BP^1$, where $\BP^1\setminus\{\infty\}=\Spec\BF_q[t]$, and let $D=\Var(t-1)$. For an algebraically closed field $k$ we consider the $k$-valued point $x=\bigl(0,\infty,\FG_{C_k},\tau= \left(\begin{smallmatrix} 0 & 1 \\ t & 0 \end{smallmatrix}\right)\bigr)$ of $\nabla_2\scrH^1(C,\FG)\times_{C^2}(C\setminus D)^2$. Then there is no $\psi_2\colon\FG\times_C D_k\isoto\FG\times_C D_k$ with $\psi_2\circ\tau|_{D_k}=\sigma^*\psi_2$. Namely $\psi_2$ equals multiplication with an element $g\in(\FG\times_C D)(k)=G_0(k)$ and then $\psi_2^{-1}\circ\sigma^*\psi_2$ equals multiplication with $g^{-1}\sigma^*g$ which lies in $\bigl\{\left(\begin{smallmatrix} * & 0 \\ 0 & * \end{smallmatrix}\right)\bigr\}\subset G_0$. On the other hand, $\tau|_{D_k}=\left(\begin{smallmatrix} 0 & 1 \\ 1 & 0 \end{smallmatrix}\right)$. This proves the non-existence of $\psi_2$, and hence the non-existence of a point of $\nabla_\ulI\scrH^1_D(C,\FG)$ above $x$.

\end{example}

\begin{remark}\label{RemIndepOfRep}
Let us explain to what extent the results of Theorem~\ref{nHisArtin} depend on the choice of the representation $\rho$. The stacks $\CG r_{\FG,I_j,r}$ and all stacks involving a bound by $\ul\omega$ depend in their definition on $\rho$. The others are independent. However, the structure of $\nabla_\ulI\scrH^1_D(C,\FG)=\dirlim\nabla_\ulI^{\ul\omega}\scrH^1_D(C,\FG)$ as an ind-DM-stack over $(C\setminus D)^I$ depends on $\rho$. Nevertheless, one can choose a covering of $\scrH^1_D(C,\FG)$ by quasi-compact open substacks $U$; see Theorem~\ref{Bun_G}. If we choose a different representation $\rho'$ and corresponding coweights $\ul\omega'$, the restriction of $\nabla_\ulI^{\ul\omega'}\scrH^1_D(C,\FG)$ to $U$ is quasi-compact by Theorem~\ref{nHisArtin}. So above $U$ the morphism $\nabla_\ulI^{\ul\omega'}\scrH^1_D(C,\FG)\,\longto\,\nabla_\ulI\scrH^1_D(C,\FG)\,=\,\dirlim\nabla_\ulI^{\ul\omega}\scrH^1_D(C,\FG)$ factors through some $\nabla_\ulI^{\ul\omega}\scrH^1_D(C,\FG)$ by Definition~\ref{DefIndAlgStack}\ref{DefIndAlgStack_B}. In that sense the structure of $\nabla_\ulI\scrH^1_D(C,\FG)$ as an ind-DM-stack over $(C\setminus D)^I$ is independent of the representation $\rho$.

Another more intrinsic way to obtain the strcture of ind-DM-stack on $\nabla_\ulI\scrH^1_D(C,\FG)$ would be to use an analog of $\ul\omega$ for the group $\FG$. Since the fibers of $\FG$ over $C$ are not reductive in general this analog cannot be given by coweights. Instead one should work in the affine flag variety of $\FG$. To be more precise, in the local situation in Chapter~\ref{LoopsAndSht} we consider bounds for local $\BP_\nu$-shtukas, see Definition~\ref{DefBDLocal}, and we apply these in Chapter~\ref{Uniformization Theorem} to global $\FG$-shtukas whose characteristic sections $(s_1,\ldots,s_n)\colon S\to C^n$ factor through the (formal spectrum of a) complete local ring of a closed point $\ul\nu\in C^n$; see Definition~\ref{bdglobal}. These bounds are given by closed formal subschemes of the local affine flag varieties; see after Definition~\ref{RZ space}. If one wants to avoid the restriction on the caracteristic sections, one can work with the global affine flag variety and use it to define bounds that then provide the structure of ind-DM-stack on $\nabla_\ulI\scrH^1_D(C,\FG)$ in a way intrinsic to the group $\FG$. For more details we refer to \cite{EsmailSomayeh_LocModel}.
\end{remark}

\begin{remark}\label{RemDrinfeldVarshavsky}
Our moduli spaces for global $\FG$-shtukas generalize Varshavsky's \cite{Var} moduli stacks $FBun$, which in turn are a generalization of the moduli spaces of $F$-sheaves $FSh_{D,r}$ considered by Drinfeld~\cite{Drinfeld1} and $\text{Cht}^r_N$ considered by Lafforgue~\cite{LafforgueL02} in their proof of the Langlands correspondence for $\FG=\GL_2$ (resp.\ $\FG=\GL_r$). Namely Varshavsky considers the situation where $G$ is a split reductive group over $\BF_q$ and $\FG=G\times_{\BF_q}C$ is constant. Let $T$ be a maximal split torus in $G$ and let $\Lambda$ be a finite generating system of the monoid of dominant weights $X^*(T)_\dom$, dominant with respect to the choice of a Borel subgroup $B\subset G$ containing $T$. Let $\olB$ be the opposite Borel. For all $\lambda\in\Lambda$ let $V_\lambda:=\bigl(\Ind_\olB^G(-\lambda)_\dom\bigr)\dual$ be the Weyl module of highest weight $\lambda$ and set \mbox{$V:=\bigoplus_{\lambda\in\Lambda}V_\lambda$}. The representation $\rho\colon G\to\GL(V)\into\SL(V\oplus\wedge^\topol V\dual)$ is faithful by \cite[Proposition~3.14]{H-V} and we take it as the representation fixed before Remark~\ref{RemRelAffineGrass} using Proposition~\ref{PropQAffineQuot}\ref{PropQAffineQuot_D}. Varshavsky considers an $n$-tuple of dominant coweights $\ul{\wt\omega}=(\wt\omega_1,\ldots,\wt\omega_n)$ of $G$ and the induced coweights $\rho\ul{\wt\omega}=(\rho\circ\wt\omega_1,\ldots,\rho\circ\wt\omega_n)$ of $\SL(V\oplus\wedge^\topol V\dual)$. Then his stack $FBun_{D,n,\ul{\wt\omega}}$, see \cite[Proposition~2.16]{Var}, coincides with our stack $\nabla_n^{\rho\ul{\wt\omega}}\scrH^1_D(C,\FG)$. The stacks $FSh_{D,r}$ and $\text{Cht}^r_N$ of Drinfeld and Lafforgue are obtained by taking $G=\GL_r$, $n=2$, the dominant coweights $\wt\omega_1=(1,0,\ldots,0)$ and $\wt\omega_2=(0,\ldots,0,-1)$ of $\GL_r$, and using the correspondence between $\GL_r$-torsors and locally free sheaves of rank $r$; see Remark~\ref{RemVect}.

Varshavsky's moduli stacks were further generalized by Bao Ch\^au Ng\^o and Ng\^o Dac Tu\^an~\cite{NgoNgo,NgoDac13} who dropped the assumption that $\FG$ is constant split and allowed $\FG$ to be a smooth group scheme over $C$ all of whose fibers are reductive. Other variants considered before include the moduli spaces $\CE\ell\ell_{C,\mathscr{D},I}$ of $\mathscr{D}$-elliptic sheaves of Laumon, Rapoport and Stuhler~\cite{LRS}, and their generalizations by L.~Lafforgue~\cite{LafforgueL97}, Ng\^o~\cite{Ngo06} and Spie{\ss}~\cite{Spiess10}, who take $\FG$ as the group of units in a maximal (resp.\ hereditary) order $\mathscr{D}$ of a central division algebra over $Q$, as well as the moduli spaces $\AbSh^{r,d}_H$ of abelian $\tau$-sheaves of the second author~\cite{Har1}, who takes $\FG=\GL_r$, $n=2$ and $\wt\omega_1=(d,0,\ldots,0)$ and $\wt\omega_2=(0,\ldots,0,-d)$. Both \cite{LRS} and \cite{Har1} require in addition that the characteristic section $\charsect_2\colon S\to C$ factors through a fixed place $\infty\in C$ and that the associated local $\GL_r$-shtuka at $\infty$ is isoclinic in the sense of \cite{Har1}, that is, basic in the sense of \cite{H-V}. All these moduli spaces are special cases of our $\nabla_n^{\ul\omega}\scrH^1_D(C,\FG)$.

Varshavsky's stacks $FBun$ were recently used by V.~Lafforgue~\cite{Lafforgue12} to prove Langlands parameterization for split connected reductive groups over the function field $Q$. In the non-split case, Lafforgue~\cite[\S\,12.3.2]{Lafforgue12} chooses a parahoric model $\FG$ of the connected reductive group over $Q$ and generalizes Varshavsky's moduli stack to a moduli stack ${\rm Cht}^{(I_1,\ldots,I_k)}_{N,I}$ which equals our $\nabla_\ulI\scrH^1_D(C,\FG)$ when $N=D$, and hence is an ind-DM-stack over $C^I$ locally of ind-finite type by Theorem~\ref{nHisArtin}. Lafforgue~\cite[\S\,12.1]{Lafforgue12} truncates the stacks by bounding the degree of the $\FG$-bundles as follows. He fixes a representation $\theta\colon\FG\onto\FG^\ad\into\GL(\CV_0)$ and a dominant cocharacter $\mu=(\mu_1\ge\ldots\ge\mu_r)$ of $\SL_r$, where $r=\rk\CV_0$, and considers the sub-stack $\scrH^1(C,\FG)^{\le\mu}$ from Remark~\ref{RemCoveringH1}; see \cite[\S\,11.1 and (1.3)]{Lafforgue12}. Since $\FG^\ad$ is semi-simple the representation $\theta$ factors through $\FG\onto\FG^\ad\into\SL(\CV_0)\into\GL(\CV_0)$, because its geometric generic fiber does; use \cite[27.5~Theorem~(d)]{Humphreys75}. Therefore the stacks $\scrH^1(C,\FG^\ad)^{\le\mu}$ are of finite type over $\BF_q$ by Remark~\ref{RemCoveringH1}.

In addition, V.~Lafforgue considers representations $W$ of $({}^L\FG)^I$, where ${}^L\FG$ is the Langlands dual group of $\FG_Q:=\FG\times_C\Spec Q$. For every such representation he defines the stack ${\rm Cht}_{N,I,W}^{\ulI,\le\mu}\,=\,\bigcup_{\ul\omega}\nabla^{\ul\omega}_\ulI\scrH^1_D(C,\FG)^{\le\mu}$, again with $N=D$, where $\rho\colon\FG\into\SL(\CV)$ is a fixed faithful representation (as before Remark~\ref{RemRelAffineGrass}) and the union runs over a finite family of $n$-tuples $\ul\omega$ of coweights of $\SL_{\rk\CV}$ related to $W$. More precisely, when $\FG$ is constant split the $\ul\omega$ are of the form $\ul\omega=\rho\ul{\wt\omega}$, where the $\ul{\wt\omega}$ run over those $n$-tuples of dominant coweights of $\FG$ whose Weyl modules occur in $W$. When $\FG$ is non-split this set is described in \cite[\S\,12.3.1]{Lafforgue12}. In this case the perverse sheaves associated with $W$ are not necessarily IC-sheaves. So Lafforgue does not need to care about the exact support of these sheaves and can use any bound to get Artin stacks locally of finite type. Our Theorem~\ref{nHisArtin} (or Remark~\ref{RemCoveringH1}) shows that $\nabla^{\ul\omega}_\ulI\scrH^1_D(C,\FG^\ad)^{\le\mu}$ is a Deligne-Mumford stack of finite type and separated over $(C\setminus D)^I$. Finally Lafforgue fixes a cocompact lattice $\Xi\subset Z(\BA)/Z(Q)$, where $Z$ is the center of $\FG_Q$, and shows that ${\rm Cht}_{N,I,W}^{\ulI,\le\mu}/\Xi$ is a Deligne-Mumford stack of finite type; see \cite[Lemma~12.19]{Lafforgue12}.
\end{remark}

\begin{remark}
Varshavsky~\cite[Proposition~2.16]{Var} proves our Theorem~\ref{nHisArtin} for a constant split reductive group $\FG=G\times_{\BF_q}C$ by a different technique. He proceeds by augmenting the level $D$ and introducing bounds on the degree of $B$-structures; see Remark~\ref{RemCoveringH1}. In this way he obtains open substacks of $\scrH^1_D(C,\FG)$ which are representable by schemes, and he can prove that $\nabla_n^{\ul\omega}\scrH^1_D(C,\FG)$ is locally a quotient of a scheme by a finite group, so in particular a separated Deligne-Mumford stack. For general flat group schemes $\FG$ the existence of $B$-structures has been established so far only when $\FG$ is parahoric with quasi-split and semi-simple generic fiber; see Heinloth~\cite[Corollary~26]{Heinloth}. For this reason we use a different argument to prove Theorem~\ref{nHisArtin}.
\end{remark}

\begin{remark}
Like in the case of Shimura varieties the question whether $\nabla_n^{\ul\omega}\scrH^1_D(C,\FG)$ is proper over $(C\setminus D)^I$ depends on the group $\FG$, but unlike to that case it also depends on additional aspects. For example Drinfeld's and Lafforgue's moduli stack for $\FG=\GL_r$ is not proper and this results in considerable efforts to compactify it; see \cite[Chapter~III]{LafforgueL02}. On the other hand the stack $\CE\ell\ell_{C,\mathscr{D},I}$ of Laumon, Rapoport and Stuhler, is proper outside the ramification locus of $\mathscr{D}$; see \cite[Theorem~6.1]{LRS}. Here $\FG$ is the group of units of a maximal order $\mathscr{D}$ in a central division algebra, $n=2$, and the bound on the relative position of $\CG$ and $\sigma^*\CG$ under $\tau$ is non-trivial but as small as possible. If one allows larger bounds as considered by \cite{LafforgueL97,Ngo06,LauDiss,Lau07} one has to increase the ramification of $\mathscr{D}$ to retain the properness of the moduli space. Otherwise it can actually happen that the moduli spaces are not proper; see \cite[Theorem~A]{Lau07} for the precise statement. This phenomenon is unknown from the theory of Shimura varieties.
\end{remark}

\section{Loop Groups and Local $\BP$-Shtukas}\label{LoopsAndSht}
 \setcounter{equation}{0}

In this chapter we let $\BP$ be a smooth affine group scheme over $\BD:=\Spec\BaseOfD\dbl z\dbr$. We let $\dot\BD:=\Spec\BaseOfD\dpl z\dpr$ and $\genericG:=\BP\times_\BD\dot{\BD}$ be the generic fiber of $\BP$. We are mainly interested in the situation where we have an isomorphism $\BD\cong\Spec A_\nu$ for a place $\nu$ of $C$ and where $\BP=\BP_\nu:=\FG\times_C\Spec A_\nu$, assuming that $\FG$ is smooth above $\nu$.

The \emph{group of positive loops associated with $\BP$} is the infinite dimensional affine group scheme $L^+\BP$ over $\BaseOfD$ whose $R$-valued points for an $\BaseOfD$-algebra $R$ are 
\[
L^+\BP(R):=\BP(R\dbl z\dbr):=\BP(\BD_R):=\Hom_\BD(\BD_R,\BP)\,.
\]
The \emph{group of loops associated with $\genericG$} is the $\fpqc$-sheaf of groups $L\genericG$ over $\BaseOfD$ whose $R$-valued points for an $\BaseOfD$-algebra $R$ are 
\[
L\genericG(R):=\genericG(R\dpl z\dpr):=\genericG(\dot{\BD}_R):=\Hom_{\dot\BD}(\dot\BD_R,\genericG)\,,
\]
where we write $R\dpl z\dpr:=R\dbl z \dbr[\frac{1}{z}]$ and $\dot{\BD}_R:=\Spec R\dpl z\dpr$. Let $\scrH^1(\Spec\BaseOfD,L^+\BP)\,:=\,[\Spec\BaseOfD/L^+\BP]$ (resp.\ $\scrH^1(\Spec\BaseOfD,L\genericG)\,:=\,[\Spec\BaseOfD/L\genericG]$) denote the classifying space of $L^+\BP$-torsors (resp.\ $L\genericG$-torsors). It is a stack fibered in groupoids over the category of $\BaseOfD$-schemes $S$ whose category $\scrH^1(\Spec\BaseOfD,L^+\BP)(S)$ of $S$-valued points consists of all $L^+\BP$-torsors (resp.\ $L\genericG$-torsors) on $S$. The inclusion of sheaves $L^+\BP\subset L\genericG$ gives rise to the natural morphism 
\begin{equation}\label{EqLoopTorsor}
L\colon\scrH^1(\Spec\BaseOfD,L^+\BP)\longto \scrH^1(\Spec\BaseOfD,L\genericG),~\CL_+\mapsto \CL:=L\CL_+\,.
\end{equation}

\begin{definition}\label{localSht}
Let $S\in\Nilp_{\BaseOfD\dbl\zeta\dbr}$ and let $\hat{\sigma}:=\hat{\sigma}_S$ be the $\BaseOfD$-Frobenius of $S$. A \emph{local $\BP$-shtuka over $S\in \Nilp_{\BaseOfD\dbl\zeta\dbr}$} is a pair $\ul\CL = (\CL_+,\tauLoc)$ consisting of an $L^+\BP$-torsor $\CL_+$ on $S$ and an isomorphism of the associated $LP$-torsors $\tauLoc\colon  \hat{\sigma}^\ast \CL \to\CL$ from \eqref{EqLoopTorsor}. 
A local $\BP$-shtuka $(\CL_+,\tauLoc)$ is called \emph{\'etale} if $\tauLoc$ comes from an isomorphism of $L^+\BP$-torsors $\hat{\sigma}^\ast\CL_+\isoto \CL_+$.
We denote the stack fibered in groupoids over $\Nilp_{\BaseOfD\dbl\zeta\dbr}$ which classifies local $\BP$-shtukas by $\Sht_{\BP}^{\BD}$.
\end{definition}

\begin{definition}\label{quasi-isogeny L}
A \emph{quasi-isogeny} $f\colon\ul\CL\to\ul\CL'$ between two local $\BP$-shtukas $\ul{\CL}:=(\CL_+,\tauLoc)$ and $\ul{\CL}':=(\CL_+' ,\tauLoc')$ over $S$ is an isomorphism of the associated $L\genericG$-torsors $f \colon  \CL \to \CL'$ satisfying $f\circ\tauLoc=\tauLoc'\circ\hat{\sigma}^{\ast}f$. We denote by $\QIsog_S(\ul{\CL},\ul{\CL}')$ the set of quasi-isogenies between $\ul{\CL}$ and $\ul{\CL}'$ over $S$, and we write $\QIsog_S(\ul\CL):=\QIsog_S(\ul\CL,\ul\CL)$ for the quasi-isogeny group of $\ul\CL$. 
\end{definition}

As in the theory of $p$-divisible groups, also our quasi-isogenies are rigid; see \cite[\PropRigidityLocal]{AH_Local}. Moreover, local $\BP$-shtukas possess moduli spaces in the following sense. For a scheme $\SSS$ in $\Nilp_{\BaseOfD\dbl\zeta\dbr}$ let $\bar{\SSS}$ denote the closed subscheme $\Var_\SSS(\zeta)\subseteq \SSS$. On the other hand for a scheme $\bar \TTT$ over $\BaseOfD$ we set $\wh{\TTT}:=\bar \TTT\whtimes_{\Spec\BaseOfD}\Spf\BaseOfD\dbl\zeta\dbr$. Then $\wh \TTT$ is a $\zeta$-adic formal scheme with underlying topological space $\bar \TTT=\Var_{\wh \TTT}(\zeta)$.

\begin{definition}\label{RZ space}
With a given local $\BP$-shtuka $\ul\BL$ over an $\BaseOfD$-scheme $\bar{\TTT}$ we associate the functor
\begin{eqnarray*}
\CM_{\ul\BL}\colon  \Nilp_{\wh{\TTT}} &\to&  \Sets\\
\SSS &\mapsto & \big\{\text{Isomorphism classes of }(\ul{\CL},\bar{\delta})\colon\;\text{where }\ul{\CL}~\text{is a local $\BP$-shtuka}\\ 
&&~~~~ \text{over $\SSS$ and }\bar{\delta}\colon  \ul{\CL}_{\bar{\SSS}}\to \ul\BL_{\bar{\SSS}}~\text{is a quasi-isogeny  over $\bar{\SSS}$}\big\}. 
\end{eqnarray*}
Here we say that $(\ul\CL,\bar\ppsi)$ and $(\ul\CL',\bar\ppsi')$ are isomorphic if $\bar\ppsi^{-1}\circ\bar\ppsi'$ lifts to an isomorphism $\ul\CL'\to\ul\CL$. The group $\QIsog_{\bar \TTT}(\ul\BL)$ of quasi-isogenies of $\ul\BL$ acts on the functor $\CM_{\ul\BL}$ via $j\colon(\ul\CL,\bar\delta)\mapsto(\ul\CL,j\circ\bar\delta)$ for $j\in\QIsog_{\bar \TTT}(\ul\BL)$.
\end{definition}

We proved in \cite[\ThmModuliSpX]{AH_Local} that $\CM_{\ul\BL}$ is representable by an ind-scheme, called an \emph{unbounded Rapoport-Zink space}. For this recall that the affine flag variety $\SpaceFl_\BP$ is defined to be the $fpqc$-sheaf associated with the presheaf
$$
R\;\longmapsto\; L\genericG(R)/L^+\BP(R)\;=\;\BP\left(R\dpl z \dpr \right)/\BP\left(R\dbl z\dbr\right)
$$ 
on the category of $\BaseOfD$-algebras. It is represented by an ind-scheme which is ind-quasi-projective over $\BaseOfD$, and hence ind-separated and of ind-finite type over $\BaseOfD$; see \cite[Theorem~1.4]{PR2}. If $\BP$ is parahoric in the sense of Bruhat and Tits, see \cite[D\'efinition~5.2.6]{B-T} and \cite{H-R}, then $\SpaceFl_\BP$ is ind-projective over $\BaseOfD$ by \cite[Corollary~1.3]{Richarz13}. Viewing the formal scheme $\wh\TTT=\bar\TTT\whtimes_\BaseOfD\Spf \BaseOfD\dbl\zeta\dbr$ as the ind-scheme $\dirlim\bar\TTT\times_\BaseOfD\Spec\BaseOfD[\zeta]/(\zeta^{m})$ we may form the fiber product $\wh{\SpaceFl}_{\BP,\wh\TTT}:=\SpaceFl_{\BP}\whtimes_\BaseOfD\wh\TTT$ in the category of ind-schemes (see \cite[7.11.1]{B-D}).

\begin{theorem}[{\cite[\ThmModuliSpX]{AH_Local}}] \label{ThmModuliSpX}
The functor $\CM_{\ul\BL}$ from Definition~\ref{RZ space} is represented by an ind-scheme, ind-quasi-projective, hence ind-separated and of ind-finite type over $\wh \TTT=\bar \TTT\whtimes_{\BaseOfD}\Spf\BaseOfD\dbl\zeta\dbr$. If $\BP$ is parahoric in the sense of Bruhat and Tits, see \cite[D\'efinition~5.2.6]{B-T} and \cite{H-R}, then $\CM_{\ul\BL}$ is ind-projective over $\wh \TTT$.

If $\ul\BL$ is trivialized by an isomorphism $\alpha\colon\ul\BL\isoto\bigl((L^+\BP)_{\bar \TTT},b\hat\sigma^*\bigr)$ over $\bar\TTT$ with $b\in L\genericG(\bar \TTT)$ then $\CM_{\ul\BL}$ is represented by the ind-scheme $\wh\SpaceFl_{\BP,\wh\TTT}:=\SpaceFl_\BP\whtimes_\BaseOfD\wh \TTT$.
\end{theorem}

If one in addition imposes a boundedness condition on the Hodge polygon of the local $\BP$-shtukas, and assumes $\bar\TTT=\Spec\BaseFldOfLocSht$ for a field $\BaseFldOfLocSht$, then one even obtains a formal scheme locally formally of finite type over $\Spf\BaseFldOfLocSht\dbl\zeta\dbr$. Here a formal scheme over $\Spf\BaseFldOfLocSht\dbl\zeta\dbr$ in the sense of \cite[I$_{new}$, 10]{EGA} is called \emph{locally formally of finite type} if it is locally noetherian and adic and its reduced subscheme is locally of finite type over $\BaseFldOfLocSht$. To give more details we recall the following definition from \cite[Definitions~4.5 and 4.8]{AH_Local}.

\begin{definition}\label{DefBDLocal}
\begin{enumerate}
\item We fix an algebraic closure $\BaseOfD\dpl\zeta\dpr^\alg$ of $\BaseOfD\dpl\zeta\dpr$, and consider pairs $(R,\hat Z_R)$, where $R/\BaseOfD\dbl\zeta\dbr$ is a finite extension of discrete valuation rings such that $R\subset\BaseOfD\dpl\zeta\dpr^\alg$, and where $\hat Z_R\subset \wh\SpaceFl_{\BP,R}:=\SpaceFl_\BP\whtimes_{\BaseOfD}\Spf R$ is a closed ind-subscheme. Two such pairs $(R,\hat Z_R)$ and $(R',\hat Z'_{R'})$ are \emph{equivalent} if for some finite extension of discrete valuation rings $\wt R/\BaseOfD\dbl\zeta\dbr$ with $R,R'\subset \wt R$ the two closed ind-subschemes $\hat{Z}_R\whtimes_{\Spf R}\Spf\wt R$ and $\hat{Z}'_{R'}\whtimes_{\Spf R'}\Spf\wt R$ of $\wh\SpaceFl_{\BP,\wt R}$ are equal. By \cite[Remark~4.6]{AH_Local} this then holds for all such rings $\wt R$.
\item \label{DefBDLocal_A}
A \emph{bound} is an equivalence class $\hat Z:=[(R,\hat Z_R)]$ of pairs $(R,\hat Z_R)$ as above such that all $\hat{Z}_R\subset\wh{\SpaceFl}_{\BP,R}$ are stable under the left $L^+\BP$-action, and the \emph{special fiber} $Z_R:=\hat{Z}_R\whtimes_{\Spf R}\Spec \kappa_R$ is a  quasi-compact subscheme of $\SpaceFl_\BP\whtimes_{\BaseOfD}\kappa_R$ where $\kappa_R$ is the residue field of $R$. By \cite[Remark~4.10]{AH_Local} this implies that the $\hat Z_R$ are formal schemes in the sense of \cite[I$_{\rm new}$]{EGA}.

\item \label{DefBDLocal_B}
\newcommand{\BaseFld}{\ol\BaseOfD}
The \emph{reflex ring} $R_{\hat Z}$ of a bound $\hat Z=[(R,\hat Z_R)]$ is the intersection of the fixed field of the group $\{\gamma\in\Aut_{\BaseOfD\dbl\zeta\dbr}(\BaseOfD\dpl\zeta\dpr^\alg)\colon \gamma(\hat{Z})=\hat{Z}\,\}$ in $\BaseOfD\dpl\zeta\dpr^\alg$ with all the finite extensions $R\subset\BaseOfD\dpl\zeta\dpr^\alg$ of $\BaseOfD\dbl\zeta\dbr$ over which a representative $\hat{Z}_R$ of $\hat{Z}$ exists. We write $R_{\hat Z}=\kappa\dbl\xi\dbr$ and call its fraction field $E:=E_{\hat Z}=\kappa\dpl\xi\dpr$ the \emph{reflex field} of $\hat{Z}$. We let $\breve R_{\hat Z}:=\BaseFld\dbl\xi\dbr$ and $\breve E:=\breve E_{\hat Z}:=\BaseFld\dpl\xi\dpr$ be the completions of their maximal unramified extensions, where $\BaseFld$ is an algebraic closure of the finite field $\kappa$.
 
\item \label{DefBDLocal_C}
Let $\hat Z=[(R,\hat Z_R)]$ be a bound with reflex ring $R_{\hat Z}$. Let $\CL_+$ and $\CL_+'$ be $L^+\BP$-torsors over a scheme $S\in \Nilp_{R_{\hat Z}}$ and let $\delta\colon L\CL_+\isoto L\CL_+'$ be an isomorphism of the associated $L\genericG$-torsors. We consider an \'etale covering $S'\to S$ over which trivializations $\alpha\colon\CL_+\isoto(L^+\BP)_{S'}$ and $\alpha'\colon\CL_+'\isoto(L^+\BP)_{S'}$ exist. Then the automorphism $\alpha'\circ\delta\circ\alpha^{-1}$ of $(L\genericG)_{S'}$ corresponds to a morphism $S'\to L\genericG\whtimes_{\BaseOfD}\Spf R_{\hat Z}$. We say that $\delta$ is \emph{bounded by $\hat{Z}$} if for every such trivialization and for every finite extension $R$ of $\BaseOfD\dbl\zeta\dbr$ over which a representative $\hat Z_R$ of $\hat Z$ exists the induced morphism
\[
S'\whtimes_{R_{\hat Z}}\Spf R\longto L\genericG\whtimes_{\BaseOfD}\Spf R\longto \wh{\SpaceFl}_{\BP,R}
\]
factors through $\hat{Z}_R$. Furthermore, we say that a local $\BP$-shtuka $\ul\CL=(\CL_+,\hat\tau)$ is \emph{bounded by $\hat{Z}$} if $\hat\tau$ is bounded by $\hat{Z}$.
\end{enumerate}
\end{definition}

\begin{remark}
(a) The reflex field $E_{\hat{Z}}$ of $\hat Z$ is always a finite extension of $\BaseOfD\dpl\zeta\dpr$. For a detailed explanation of our definition of reflex ring and a comparison with the number field case see \cite[Remark~4.7]{AH_Local}. We do not know whether in general $\hat Z$ has a representative over the reflex ring. In contrast, the equivalence class of the special fibers $Z_R:=\hat{Z}_R\whtimes_{\Spf R}\Spec \kappa_R$ always has a representative $Z\subset\SpaceFl_\BP\whtimes_{\BF_q}\Spec\kappa$ over the residue field $\kappa$ of the reflex ring $R_{\hat Z}$, because the Galois descent for closed ind-subschemes of $\SpaceFl_\BP$ is effective. We call $Z$ the \emph{special fiber} of $\hat Z$. It is a quasi-projective scheme over $\kappa$ by \cite[Lemma~5.4]{H-V} because $\SpaceFl_\BP$ is ind-quasi-projective. It is even projective over $\kappa$ if $\BP$ is parahoric, because then $\SpaceFl_\BP$ is ind-projective.

\medskip\noindent 
(b) The condition of Definition~\ref{DefBDLocal}\ref{DefBDLocal_C} is satisfied for \emph{all} trivializations $\alpha$ and $\alpha'$ and for \emph{all} such finite extensions $R$ of $\BF_q\dbl\zeta\dbr$ if and only if it is satisfied for \emph{one} trivialization and for \emph{one} such finite extension. Indeed, by the $L^+\BP$-invariance of $\hat Z$ the definition is independent of the trivializations. That one finite extension suffices follows from \cite[Remark~4.6]{AH_Local}.
\end{remark}

We recall the following

\begin{proposition}[{\cite[Proposition~4.11]{AH_Local}}]\label{PropBoundedClosed}
Let $\CL_+$ and $\CL'_+$ be $L^+\BP$-torsors over a scheme $S\in\Nilp_{\BaseOfD\dbl\zeta\dbr}$ and let $\delta\colon \CL\isoto\CL'$ be an isomorphism of the associated $L\genericG$-torsors. Let $\hat Z$ be a bound. Then the condition that $\delta$ is bounded by $\hat Z$ is represented by a closed subscheme of $S$.
\end{proposition}

To formulate the ind-representability in the bounded case we consider the following important special situation. Let $\hat{Z}$ be a bound with reflex ring $R_{\hat{Z}}=\kappa\dbl\xi\dbr$ and special fiber $Z\subset\SpaceFl_\BP\whtimes_\BaseOfD\Spec\kappa$. Let $\ul\BL=\bigl((L^+\BP)_\BaseFldOfLocSht,b\hat{\sigma}^*\bigr)$ be a trivialized local $\BP$-shtuka over a field $\BaseFldOfLocSht$ in $\Nilp_{\BaseOfD\dbl\zeta\dbr}$ with $b\in L\genericG(\BaseFldOfLocSht)$. Assume that $b$ is decent with integer $s$ in the sense of \cite[\EqDecency]{AH_Local}, and let $\ell\subset\BaseFldOfLocSht^\alg$ be the compositum of the residue field $\kappa$ of $R_{\hat{Z}}$ and the finite field extension of $\BaseOfD$ of degree $s$. Then $b\in L\genericG(\ell)$ by \cite[\RemDecent]{AH_Local}. So $\ul\BL$ is defined over $\bar T=\Spec\ell$ and we may replace $\BaseFldOfLocSht$ by $\ell$. Then $\CM_{\ul\BL}$ is defined over $\Spf\ell\dbl\zeta\dbr$. We remark that $\CM_{\ul\BL}$ only depends on the isogeny class of $\ul\BL$. Notice that if $\BaseFldOfLocSht$ is algebraically closed any local $\BP$-shtuka over $\BaseFldOfLocSht$ is trivial and isogenous to a decent one by \cite[\RemDecent]{AH_Local}. Also $\ell\dbl\xi\dbr$ is the finite unramified extension of $R_{\hat{Z}}$ with residue field $\ell$. We recall the following

\renewcommand{\BaseFldOfLocSht}{\ell}

\begin{theorem}[{\cite[\ThmRRZSp]{AH_Local}}] \label{ThmRRZSp}
In the above situation if $\BP$ is a smooth affine group scheme over $\BD$ with connected reductive generic fiber, the functor 
\begin{eqnarray*}
\CM_{\ul\BL}^{\hat{Z}}\colon (\Nilp_{\BaseFldOfLocSht\dbl\xi\dbr})^o &\longto & \Sets\\
S&\longmapsto & \Bigl\{\,\text{Isomorphism classes of }(\ul\CL,\bar\ppsi)\in\CM_{\ul\BL}(S)\colon\ul\CL\text{~is bounded by $\hat{Z}$}\Bigr\}
\end{eqnarray*}
is ind-representable by a formal scheme over $\Spf \BaseFldOfLocSht\dbl\xi\dbr$ which is locally formally of finite type and separated. It is called a \emph{bounded Rapoport-Zink space for local $\BP$-shtukas}. Its underlying reduced subscheme equals the associated \emph{affine Deligne--Lusztig variety}, which is the reduced closed ind-subscheme $X_Z(\ul\BL)\subset\SpaceFl_\BP\whtimes_\BaseOfD\Spec\BaseFldOfLocSht$ whose $K$-valued points (for any field extension $K$ of $\BaseFldOfLocSht$) are given by
\begin{equation}
\label{EqADLV}
X_Z(\ul\BL)(K)\;:=\;X_Z(b)(K)\;:=\;\big\{ g\in \SpaceFl_\BP(K)\colon g^{-1}\,b\,\hat\sigma^\ast(g) \in Z(K)\big\}.
\end{equation}
In particular $X_Z(\ul\BL)$ is a scheme locally of finite type over $\BaseFldOfLocSht$. 
\end{theorem}

\begin{remark}\label{RemJActsOnRZ}
The group $\QIsog_{\BaseFldOfLocSht}(\ul\BL)$ of quasi-isogenies of $\ul\BL$ acts on $\RZ$ via $j\colon(\ul\CL,\bar\delta)\mapsto(\ul\CL,j\circ\bar\delta)$. Since $\ul\BL=\bigl((L^+\BP)_\BaseFldOfLocSht,b\hat{\sigma}^*\bigr)$ is trivialized and decent, $\QIsog_{\BaseFldOfLocSht}(\ul\BL)=J_b(\BaseOfD\dpl z\dpr)$ where $J_b$ is the connected algebraic group over $\BaseOfD\dpl z\dpr$ which is defined by its functor of points that assigns to an $\BaseOfD\dpl z\dpr$-algebra $R$ the group
\begin{equation}\label{EqGroupJ}
J_b(R):=\bigl\{\,j \in \genericG(R\otimes_{\BaseOfD\dpl z\dpr} {\BaseFldOfLocSht\dpl z\dpr})\colon j^{-1}b\hat{\sigma}(j)=b\,\bigr\}\,,
\end{equation}
see \cite[\RemJb]{AH_Local}.
In particular, $\QIsog_{\BaseFldOfLocSht}(\ul\BL)=J_b(\BaseOfD\dpl z\dpr)\subset L\genericG(\ell)$ and this group acts on $X_Z(b)$ via $j\colon g\mapsto j\cdot g$ for $j\in J_b(\BaseOfD\dpl z\dpr)$.
\end{remark}

%
%

\section{Unbounded Uniformization}\label{UnboundedUnif}
\setcounter{equation}{0}

Analogously to the functor which assigns to an abelian variety $A$ over a $\BZ_p$-scheme its $p$-divisible group $A[p^\infty]$ we introduced in \cite[\SectGlobalLocalFunctor]{AH_Local} a global-local functor which assigns to a global $\FG$-shtuka an $n$-tuple of local $\BP_{\nu_i}$-shtukas. The reason for the $n$-tuple is that our global $\FG$-shtukas have $n$ characteristic places, whereas abelian varieties only have one characteristic place. We consider the following situation.

\begin{definition}\label{DefGlobalLocalFunctor}
Fix a tuple $\ul\nu:=(\nu_i)_{i=1\ldots n}$ of places on $C$ with $\nu_i\ne\nu_j$ for $i\ne j$. We require throughout this chapter that $\FG$ is smooth above every $\nu_i$. Let $A_{\ul\nu}$ be the completion of the local ring $\CO_{C^n,\ul\nu}$ of $C^n$ at the closed point $\ul\nu$, and let $\BF_{\ul\nu}$ be the residue field of the point $\ul\nu$. Then $\BF_{\ul\nu}$ is the compositum of the finite fields $\BF_{\nu_i}$ inside $\BF_q^\alg$, and $A_{\ul\nu}\cong\BF_{\ul\nu}\dbl\zeta_1,\ldots,\zeta_n\dbr$ where $\zeta_i$ is a uniformizing parameter of $C$ at $\nu_i$. Let the stack 
\[
\nabla_n\scrH^1(C,\FG)^{\ul\nu}\;:=\;\nabla_n\scrH^1(C,\FG)\whtimes_{C^n}\Spf A_{\ul\nu}
\]
be the formal completion of the ind-DM-stack $\nabla_n\scrH^1(C,\FG)$ along $\ul\nu\in C^n$; see \cite[Appendix~A]{Har1} for an explanation of this concept. It is an ind-DM-stack over $\Spf A_{\ul\nu}$, that is, an ind-DM-stack fibered over the category $\Nilp_{A_{\ul\nu}}$. It is ind-separated and locally of ind-finite type over $\BF_q$ (and over $\Spf A_{\ul\nu}$) by Theorem~\ref{nHisArtin}, with structure of ind-DM-stack given as follows. The formal spectrum $\Spf A_{\ul\nu}$ can be viewed as the ind-scheme $\dirlim\Spec A_{\ul\nu}/(\zeta_1,\ldots,\zeta_n)^j$. Then $\nabla_n\scrH^1(C,\FG)^{\ul\nu}$ is the limit of separated Deligne-Mumford stacks locally of finite type over $\BF_q$
\[
\nabla_n\scrH^1(C,\FG)^{\ul\nu}\;=\;\lim_{\underset{\ul\omega,j}{\longrightarrow}}\nabla_n^{\ul\omega}\scrH^1_D(C,\FG)\times_{C^n}\Spec A_{\ul\nu}/(\zeta_1,\ldots,\zeta_n)^j.
\]
Note that a stack fibered over $\Nilp_{A_{\ul\nu}}$ is the natural generalization of a scheme over $\Spf A_{\ul\nu}$; see for example \cite{Hakim72}.

The \emph{global-local functor} $\wh\Gamma_{\nu_i}\colon \nabla_n\scrH^1(C,\FG)^{\ul\nu}(S)\to \Sht_{\BP_{\nu_i}}^{\Spec A_{\nu_i}}(S)$ from \cite[\DefGlobalLocalFunctor]{AH_Local} assigns to a global $\FG$-shtuka $\ul\CG$ over $S$ its local $\BP_{\nu_i}$-shtuka $\wh\Gamma_{\nu_i}(\ul\CG)$ at $\nu_i$. The \emph{global-local functor} 
\[
\ul{\wh\Gamma}:=\prod_i\wh\Gamma_{\nu_i}\colon\;\nabla_n\scrH^1(C,\FG)^{\ul\nu}(S)\;\longto\; \prod_i \Sht_{\BP_{\nu_i}}^{\Spec A_{\nu_i}}(S)
\]
assigns to $\ul\CG$ the $n$-tuple $\bigl(\wh\Gamma_{\nu_i}(\ul\CG)\bigr)_i$.
Both functors also transform quasi-isogenies into quasi-isogenies.
\end{definition}

Note that at a place $\nu$ outside the characteristic places $\nu_i$ we also associated with $\ul\CG$ an \'etale local $\wt\BP_\nu$-shtuka $L^+_\nu(\ul\CG)$ in \cite[\RemTateF]{AH_Local}. Here $\wt\BP_\nu:=\Res_{\BF_\nu/\BF_q}\BP_\nu$ is the Weil restriction. We proved the following result in \cite[\PropLocalIsogeny{} and \RemLocalIsogeny]{AH_Local}.

\begin{proposition}\label{PropLocalIsogeny}
Let $\ul\CG\in\nabla_n\scrH^1(C,\FG)^{\ul\nu}(\SSS)$ be a global $\FG$-shtuka over $\SSS$ and let $\nu\in C$ be a place. If $\nu\in\ul\nu$ consider a quasi-isogeny of local $\BP_\nu$-shtukas $ f\colon\ul{\CL}{}'_\nu\to\wh\Gamma_\nu(\ul\CG)$. If $\nu\notin\ul\nu$ consider a quasi-isogeny of local $\wt\BP_\nu$-shtukas $ f\colon\ul{\CL}{}'_\nu\to L^+_\nu(\ul\CG)$ and assume that $\ul{\CL}{}'_\nu$ is \'etale. Then there exists a unique global $\FG$-shtuka $\ul\CG'\in\nabla_n\scrH^1(C,\FG)^{\ul\nu}(\SSS)$ and a unique quasi-isogeny $g\colon\ul\CG'\to\ul\CG$ which is an isomorphism outside $\nu$, such that the local $\BP_\nu$-shtuka (resp.\ $\wt\BP_\nu$-shtuka) associated with $\ul\CG'$ is $\ul{\CL}{}'_\nu$ and the quasi-isogeny of local $\BP_\nu$-shtukas (resp.\ $\wt\BP_\nu$-shtukas) induced by $g$ is $ f$. We denote $\ul\CG'$ by $ f^*\ul\CG$.
\end{proposition}

As we shall see in the next theorem, the unbounded Rapoport-Zink space appears as the uniformization space for $\nabla_n\scrH^1(C,\FG)$. In the present article this is only an auxiliary result for us, in which we construct the uniformization morphism. No bounds are needed for this purpose and the construction even works in the relative situation. This uniformization morphism will be used in our main result (Theorem~\ref{Uniformization2}) on the uniformization in the bounded case. We retain the notation of Definition~\ref{DefGlobalLocalFunctor}.

\begin{theorem}\label{Uniformization1}
Let $\bar \TTT$ be a scheme over $\Spec A_{\ul\nu}/(\zeta_1,\ldots,\zeta_n)$, and let $\ul{\CG}_0\in\nabla_n\scrH^1(C,\FG)^{\ul\nu}(\bar \TTT)$ be a global $\FG$-shtuka. Let $(\ul{\BL}_i)_i:=\wh{\ul\Gamma}(\ul{\CG}_0)$ be the tuple of local $\BP_{\nu_i}$-shtukas over $\bar\TTT$ where $\wh{\ul\Gamma}$ is the global-local functor from Definition~\ref{DefGlobalLocalFunctor}. Consider the unbounded Rapoport-Zink spaces $\CM_{\ul{\BL}_i}$ which are ind-quasi-projective ind-schemes over $\bar\TTT\whtimes_{\BF_{\nu_i}}\Spf A_{\nu_i}$ by Theorem~\ref{ThmModuliSpX} and the product $\prod_i \CM_{\ul{\BL}_i}:=\CM_{\ul{\BL}_1}\whtimes_{\bar \TTT}\ldots\whtimes_{\bar \TTT}\CM_{\ul{\BL}_n}$ which is an ind-quasi-projective ind-scheme over $\wh \TTT:=\bar \TTT\whtimes_{\BF_{\ul\nu}}\Spf A_{\ul\nu}$. Then there is a natural morphism of ind-DM-stacks over $\wh \TTT$
\begin{align}
\Psi_{\ul\CG_0}\colon  \prod_i \CM_{\ul{\BL}_i} \;\longto \enspace & \nabla_n\scrH^1(C,\FG)^{\ul\nu}\whtimes_{\BF_{\ul\nu}} \bar \TTT\;=\;\nabla_n\scrH^1(C,\FG)\whtimes_{C^n} \wh \TTT,\\
\bigl(\ul\CL_i,\delta_i\colon\ul\CL_i\to\ul\BL_i\bigr)_i \;\longmapsto \enspace & \delta_n^*\circ\ldots\circ\delta_1^*\,\ul\CG_0 \nonumber
\end{align}
which is ind-proper and formally \'etale.
\end{theorem}

\begin{proof}
By rigidity of quasi-isogenies \cite[\PropRigidityLocal]{AH_Local} the functors $\CM_{\ul{\BL}_i}$ are naturally isomorphic to the functors
\begin{eqnarray*}
\SSS &\longmapsto & \big\{\text{Isomorphism classes of}~(\ul{\CL}_i,\delta_i)\colon \;\text{where }\ul{\CL}_i~\text{is a local $\BP_{\nu_i}$-shtuka over $\SSS$ }\\ 
&&~~~~~~~~~\text{and}~\delta_i\colon  \ul{\CL}_i\to \ul{\BL}_{i,\SSS}~\text{is a quasi-isogeny of local $\BP_{\nu_i}$-shtukas over $\SSS$}\big\}. 
\end{eqnarray*}
Let $\SSS\to\wh \TTT$ be a $\wh \TTT$-scheme. In particular via the morphism $\SSS\to\wh \TTT\to\bar \TTT\whtimes_{\BF_\ul\nu}\Spf\BF_{\ul\nu}\dbl\zeta_i\dbr=\bar \TTT\whtimes_{\BF_{\nu_i}}\Spf A_{\nu_i}$ the scheme $\SSS$ lies in $\Nilp_{\bar \TTT\whtimes_{\BF_{\nu_i}}\Spf A_{\nu_i}}$. Let $\charsect_i\colon \SSS\to\Spf A_{\nu_i}\to C$ be the induced characteristic morphism. Take an element $(\ul{\CL}_i,\delta_i)_i$ of $\prod_i \CM_{\ul{\BL}_i}(\SSS)$ and the global $\FG$-shtuka $\delta_1^*\,\ul\CG_{0,\SSS}$ from Proposition~\ref{PropLocalIsogeny}. Since $\nu_i\ne\nu_j$ for $i\ne j$ the local $\BP_{\nu_i}$-shtuka $\wh\Gamma_{\nu_i}(\delta_1^*\,\ul\CG_{0,\SSS})$ equals $\ul\BL_i$ for $i\ge2$. We may therefore iterate this procedure for $i=2,\ldots,n$ to obtain a global $\FG$-shtuka $\delta_n^\ast\circ\ldots\circ\delta_1^*\,\ul\CG_{0,\SSS}$ in $\nabla_n\scrH^1(C,\FG)^{\ul\nu}(\SSS)$. It satisfies $\wh\Gamma_{\nu_i}(\delta_n^\ast\circ\ldots\circ\delta_1^*\,\ul\CG_{0,\SSS})=\ul\CL_i$. Sending $(\ul{\CL}_i,\delta_i)_i$ to $\delta_n^\ast\circ\ldots\circ\delta_1^*\,\ul\CG_{0,\SSS}$ establishes the morphism $\Psi_{\ul\CG_0}$.

That $\Psi_{\ul\CG_0}$ is formally \'etale is just another way of phrasing the rigidity of quasi-isogenies. We give more details. Let $\ul\CG$ be an $\SSS$-valued point of $\nabla_n\scrH^1(C,\FG)^{\ul\nu}\whtimes_{\BF_{\ul\nu}} \bar \TTT$. Let $\bar \SSS$ be a closed subscheme of $\SSS$ defined by a locally nilpotent sheaf of ideals. Further assume that $\ul\CG_{\bar \SSS}= \bar\delta_n^\ast\circ\ldots\circ\bar\delta_1^*\,\ul\CG_{0,\bar \SSS}$ for quasi-isogenies $\bar\delta_i\colon\wh\Gamma_{\nu_i}(\ul\CG_{\bar \SSS})\to\ul\BL_{i,\bar \SSS}$ as above, defined over $\bar \SSS$ and let $\bar g\colon\ul\CG_{\bar \SSS}\to\ul\CG_{0,\bar \SSS}$ be the corresponding quasi-isogeny from Proposition~\ref{PropLocalIsogeny} with $\wh\Gamma_{\nu_i}(\bar g)=\bar\delta_i$. Now these quasi-isogenies lift uniquely to quasi-isogenies $\delta_i\colon\wh\Gamma_{\nu_i}(\ul\CG)\to\ul\BL_{i,\SSS}$ and $g\colon\ul\CG\to\ul\CG_{0,\SSS}$ over $\SSS$ by rigidity for local \cite[\PropRigidityLocal]{AH_Local}, respectively global shtukas \cite[\PropGlobalRigidity]{AH_Local}. By uniqueness we must have $\wh\Gamma_{\nu_i}(g)=\delta_i$. Therefore the $\SSS$-valued point $\delta_n^\ast\circ\ldots\circ\delta_1^*\,\ul\CG_{0,\SSS}$ of $\nabla_n\scrH^1(C,\FG)^{\ul\nu}$ is equal to $\ul\CG$ by the uniqueness in Proposition~\ref{PropLocalIsogeny}.
 
It remains to verify that $\Psi_{\ul\CG_0}$ is ind-proper. Since $\prod_i \CM_{\ul{\BL}_i}$ is of ind-finite type over $\wh \TTT$ we can test the ind-properness of $\Psi_{\ul\CG_0}$ by the valuative criterion of properness; see \cite[Theorem 7.3]{L-M}. Let $R$ be a complete discrete valuation ring with fraction field $L$ and algebraically closed residue field. Let $\ul\CG$ be an $R$ valued point of $\nabla_n\scrH^1(C,\FG)^{\ul\nu}\whtimes_{\BF_{\ul\nu}}\bar\TTT$ and set $\wh{\ul\Gamma}(\ul\CG)=(\ul\CL_i)_i$. Since $R$ is strictly henselian, we may trivialize all local $\BP_{\nu_i}$-shtukas $\ul\CL_i\cong ((L^+\BP_{\nu_i})_R,b'_i\,\hat\sigma^*_{\nu_i})$ and $(\ul\BL_i)_R\cong ((L^+\BP_{\nu_i})_R,b_i\,\hat\sigma^*_{\nu_i})$ over $R$.
Consider an $L$-valued point $\ul x$ of $\prod_i \CM_{\ul{\BL}_i}$ which maps to $\ul\CG_L$ under the morphism $\Psi_{\ul\CG_0}$ and represent it as the tuple $\bigl(((L^+\BP_{\nu_i})_L,b'_i\,\hat\sigma^*_{\nu_i}),g_i\bigr)_i$, where $g_i$ lies in $L\genericG_{\nu_i}(L)$. We may take a faithful representation $\rho\colon  \FG\to \GL(\CV)$, where $\CV$ is a vector bundle over $C$ of rank $r$; see  Proposition~\ref{PropGlobalRep}\ref{PropGlobalRep_A}. Let $\rho_\ast\ul\CG_0$ be the induced global $\GL(\CV)$-shtuka over $\bar\TTT$. Note that we want to show that there is a unique morphism $\tilde{\alpha}$ which fits into the following commutative diagram
\[
\xymatrix @C+3pc @R+1pc {
\Spec L \ar[d]\ar[r]^{\ul x}  & \prod_i \CM_{\ul{\BL}_i}\ar[d]^{\raisebox{4ex}{$\SC\Psi_{\ul\CG_0}$}} \ar[r]^{\rho_*} & \ar[d]^{\Psi_{\rho_\ast\ul\CG_0}}\prod_i \CM_{\rho_*\ul{\BL}_i}\\
\Spec R \ar@{-->}[ur]^{\tilde{\alpha}}\ar[urr]^{\quad\qquad\qquad\qquad\qquad\qquad{}_{\SC\alpha}} \ar[r] & \nabla_n\scrH^1(C,\FG)^{\ul\nu}\whtimes_{\BF_{\ul\nu}} \bar \TTT\ar[r] \ar[r]^{\rho_*}\ar[d] & \nabla_n\scrH^1(C,\GL(\CV))^{\ul\nu}\whtimes_{\BF_{\ul\nu}} \bar \TTT \ar[d]\\
& \wh\TTT \ar@{=}[r] & \wh\TTT.
}
\]
The horizontal arrows in the right commutative diagram are induced by the representation $\rho$. In addition the existence and uniqueness of the morphism $\alpha$ follows from the fact that $\prod_i \CM_{\rho_*\ul{\BL}_i}$ is ind-proper over $\wh\TTT$ by Theorem~\ref{ThmModuliSpX} because $\GL(\CV)$ is reductive, and that $\nabla_n\scrH^1(C,\GL(\CV))^{\ul\nu}\whtimes_{\BF_{\ul\nu}}\bar T$ is ind-separated over $\wh\TTT$ by Theorem~\ref{nHisArtin}. The $R$-valued point $\alpha$ is given by a tuple of the form $\bigl(\left(\rho_\ast(L^+\BP_{\nu_i})_R,\rho(b'_i)\hat\sigma^*_{\nu_i}\right),\wt{g}_i\bigr)_i$, where $\wt{g}_i$ lies in $L\!\GL_r(R)$. Since $L\genericG_{\nu_i}$ is closed in $L\!\GL_r$ and $\wt{g}_i$ extends $g_i$ over $R$ the element $\wt{g}_i$ lies in fact in $L\genericG_{\nu_i}(R)$. Thus $\wt{g}_i$ produces an extension of $\bigl(((L^+\BP_{\nu_i})_L,b'_i\hat\sigma^*_{\nu_i}),g_i\bigr)_i$ over $R$. This gives the desired morphism $\tilde{\alpha}$. Note that the commutativity of the diagram
\[
\xymatrix @C+3pc {
 & \prod_i \CM_{\ul{\BL}_i}\ar[d]^{\SC\Psi_{\ul\CG_0}}\\
\Spec R\ar[r]\ar@{-->}[ur]^{\tilde{\alpha}} & \nabla_n\scrH^1(C,\FG)^{\ul\nu}\whtimes_{\BF_{\ul\nu}} \bar \TTT\\
}
\]
follows from the ind-separatedness of the stack $\nabla_n\scrH^1(C,\FG)^{\ul\nu}\whtimes_{\BF_{\ul\nu}} \bar \TTT$; see Theorem~\ref{nHisArtin}.
\end{proof}

\begin{remark}\label{quasi-isogenylocus}
We observe that the image of $\Psi_{\ul\CG_0}$ lies inside the quasi-isogeny locus of $\ul\CG_0$ in $\nabla_n\scrH^1(C,\FG)^{\ul\nu}$. Indeed, by the above construction, starting with an $\SSS$-valued point $\ul x=(\ul{\CL}_i,\delta_i)_i$ of $\prod_i \CM_{\ul{\BL}_i}(\SSS)$ there is a unique quasi-isogeny $\delta_{\ul x}\colon \Psi_{\ul\CG_0}(\ul x)= \delta_n^\ast\circ \ldots\circ\delta_1^*\,\ul\CG_{0,\SSS} \to \ul\CG_{0,\SSS}$ which is an isomorphism outside the $\nu_i$ with $\left(\wh{\ul\Gamma}(\Psi_{\ul\CG_0}(\ul x)), \wh{\ul\Gamma}(\delta_{\ul x}) \right)=\ul x$ by Proposition~\ref{PropLocalIsogeny}.
\end{remark}

%
%

\section{Tate Modules of Global $\FG$-Shtukas}\label{GalRepSht}
\setcounter{equation}{0}

From now on we assume that the group scheme $\FG$ is smooth over $C$ and all its fibers are connected.

Fix an $n$-tuple $\ul\nu:=(\nu_1,\ldots ,\nu_n)$ of pairwise different places on $C$ and recall the notation from Definition~\ref{DefGlobalLocalFunctor}. Let $C':=C\setminus\{\nu_1,\ldots,\nu_n\}$. Let $\BO^{\ul\nu}:=\prod_{\nu\notin\ul\nu}A_\nu$ be the ring of integral adeles of $C$ outside $\ul\nu$ and $\BA^{\ul\nu}:=\BO^{\ul\nu}\otimes_{\CO_C}Q$ the ring of adeles of $C$ outside $\ul\nu$. The group $\FG(\BO^{\ul\nu})$ acts through Hecke correspondences on the tower $\nabla_n\scrH^1_D(C,\FG)^{\ul\nu}$. We want to extend this to an action of $\FG(\BA^{\ul\nu})$\lang{; see Remark~\ref{RemHeckeCorr}(b) below}. For this purpose we generalize the notion of level structures on global $\FG$-shtukas in this chapter. Let $S$ be a scheme in $\Nilp_{A_{\ul\nu}}$ and let $\ul\CG\in\nabla_n\scrH^1(C,\FG)^{\ul\nu}(S)$ be a global $\FG$-shtuka over $S$. For a finite subscheme $D$ of $C$ set $D_S:=D\times_{\BF_q}S$ and let $\ul\CG|_{D_S}:=\ul\CG\times_{C_S}D_S$ denote the pullback of $\ul\CG$ to $D_S$.
Let $\Rep_{\BO^{\ul\nu}} \FG$ be the category of representations of $\FG$ in finite free $\BO^{\ul\nu}$-modules $V$. More precisely, $\Rep_{\BO^{\ul\nu}} \FG$ is the category of $\BO^{\ul\nu}\,$-morphisms \mbox{$\rho\colon\FG\times_C\Spec\BO^{\ul\nu}\to\GL_{\BO^{\ul\nu}}(V)$}. Assume that $S$ is connected, fix a geometric base point $\bar{s}$ of $S$, and let $Funct^\otimes(\Rep_{\BO^{\ul\nu}} \FG,\; \FM od_{\BO^{\ul\nu} [\pi_1^\et(S,\bar{s})]})$ denote the category of tensor functors from $\Rep_{\BO^{\ul\nu}} \FG$ to the category $\FM od_{\BO^{\ul\nu}[\pi_1^\et(S,\bar{s})]}$ of $\BO^{\ul\nu}[\pi_1^\et(S,\bar{s})]$-modules. We define the (\emph{dual}) \emph{Tate functors} as follows
\begin{eqnarray}\label{tatefunctor}
\check{\CT}_{-}\colon \nabla_n\scrH^1(C,\FG)^{\ul\nu}(S) \;&\longto &\; Funct^\otimes (\Rep_{\BO^{\ul\nu}}\FG\,,\,\FM od_{\BO^{\ul\nu}[\pi_1^\et(S,\bar{s})]})\,\\
\ul\CG \;&\longmapsto &\; \bigl(\check{\CT}_{\ul\CG}\colon  \rho \mapsto \invlim[D\subset C'](\rho_\ast\ul\CG|_{D_{\bar{s}}})^\tau\bigr),\nonumber\\
\check{\CV}_{-}\colon \nabla_n\scrH^1(C,\FG)^{\ul\nu}(S) \;&\longto &\; Funct^\otimes (\Rep_{\BO^{\ul\nu}}\FG\,,\,\FM od_{\BA^{\ul\nu}[\pi_1^\et(S,\bar{s})]})\,\label{rationaltatefunctor}\\ 
\ul\CG \;&\longmapsto &\; \bigl(\check{\CV}_{\ul\CG}\colon  \rho \mapsto \invlim[D\subset C'](\rho_\ast\ul\CG|_{D_{\bar{s}}})^\tau\otimes_{\BO^{\ul\nu}} \BA^{\ul\nu}\bigr).\nonumber
\end{eqnarray}
Here $D_{\bar s}$ is finite over $\bar s=\Spec k$ for an algebraically closed field $k$, and $\rho_*\ul\CG|_{D_{\bar s}}$ is equivalent to $(M,\tau_M)$ where $M$ is a free $\CO_{D_{\bar s}}$-module of rank equal to $\dim\rho$, and $\tau_M\colon\sigma^*M\isoto M$ is an isomorphism of $\CO_{D_{\bar s}}$-modules. Then $(\rho_\ast\ul\CG|_{D_{\bar{s}}})^\tau:=\{m\in M\colon \tau_M(\sigma^*m)=m\}$ denotes the $\tau$-invariants. They form a free $\CO_D$-module of rank equal to $\dim\rho$ equipped with a continuous action of the \'etale fundamental group $\pi_1^\et(S,\bar{s})$. This definition is independent of $\bar{s}$ up to a change of base point. By \cite[\RemTateF]{AH_Local} there is a canonical isomorphism $\check{\CT}_{\ul\CG}(\rho)=\invlim[D\subset C'](\rho_\ast\ul\CG|_{D_{\bar{s}}})^\tau\cong\prod_{\nu\in C'} \check{\CT}_{L^+_\nu(\ul\CG)}(\rho_\nu)$ where $\rho=(\rho_\nu)$ and where $L^+_\nu(\ul\CG)$ is the \'etale local $\wt\BP_\nu$-shtuka associated with $\ul\CG=(\CG,\tauGlob)$ at a place $\nu\in C$ outside the characteristic places $\ul\nu$. The functor $\check{\CV}_{-}$ moreover transforms quasi-isogenies into isomorphisms. In fact, $\check{\CV}_{\ul\CG}$ extends to a tensor functor in $Funct^\otimes (\Rep_{\BA^{\ul\nu}}\FG\,,\,\FM od_{\BA^{\ul\nu}[\pi_1^\et(S,\bar{s})]})$ by the usual arguments, see for example \cite[Definition~7.1 and Remark~7.3]{HV3}, but we will not need this here. 

\begin{definition}\label{level structure}
Let $\omega_{\BO^{\ul\nu}}^\circ\colon \Rep_{\BO^{\ul\nu}}\FG \to \FM od_{\BO^{\ul\nu}}$ and $\omega^\circ:=\omega_{\BO^{\ul\nu}}^\circ\otimes_{\BO^{\ul\nu}}\BA^{\ul\nu}\colon \Rep_{\BO^{\ul\nu}}\FG \to \FM od_{\BA^{\ul\nu}}$ denote the forgetful functors. For a global $\FG$-shtuka $\ul\CG$ over $S$ let us consider the sets of isomorphisms of tensor functors $\Isom^{\otimes}(\omega_{\BO^{\ul\nu}}^\circ,\check{\CT}_{\ul{\CG}})$ and $\Isom^{\otimes}(\omega^\circ,\check{\CV}_{\ul{\CG}})$. These sets are non-empty by Lemma~\ref{LSisnonempty} below. Note that $\FG(\BA^{\ul\nu})=\Aut^\otimes(\omega^\circ)$ and $\FG(\BO^{\ul\nu})=\Aut^\otimes(\omega_{\BO^{\ul\nu}}^\circ)$ by the generalized Tannakian formalism \cite[Corollary~5.20]{Wed}, because $\BO^{\ul\nu}$ is a Pr\"ufer ring; see \cite[Definition~10.17]{Larsen}. By the definition of the Tate functor, $\Isom^{\otimes}(\omega_{\BO^{\ul\nu}}^\circ,\check{\CT}_{\ul{\CG}})$ admits an action of $\pi_1^\et(S,\bar{s})\times\FG(\BO^{\ul\nu})$ where $\FG(\BO^{\ul\nu})$ acts through $\omega_{\BO^{\ul\nu}}^\circ$ and $\pi_1^\et(S,\bar{s})$ acts through $\check{\CT}_{\ul\CG}$. Also $\Isom^{\otimes}(\omega^\circ,\check{\CV}_{\ul{\CG}})$ admits an action of $\pi_1^\et(S,\bar{s})\times\FG(\BA^{\ul\nu})$. 
\end{definition}

\begin{lemma}\label{LSisnonempty}
The sets $\Isom^{\otimes}(\omega_{\BO^{\ul\nu}}^\circ,\check{\CT}_{\ul{\CG}})$ and $\Isom^{\otimes}(\omega^\circ,\check{\CV}_{\ul{\CG}})$ are non-empty. $\Isom^{\otimes}(\omega_{\BO^{\ul\nu}}^\circ,\check{\CT}_{\ul{\CG}})$, respectively $\Isom^{\otimes}(\omega^\circ,\check{\CV}_{\ul{\CG}})$, is a principal homogeneous space under the group $\FG(\BO^{\ul\nu})$, respectively $\FG(\BA^{\ul\nu})$.
\end{lemma}

\begin{proof}
Using the injective map $\Isom^{\otimes}(\omega_{\BO^{\ul\nu}}^\circ,\check{\CT}_{\ul{\CG}})\longto\Isom^{\otimes}(\omega^\circ,\check{\CV}_{\ul{\CG}}),\;\gamma\mapsto\gamma_Q:=\gamma\otimes_{\BO^{\ul\nu}}\BA^{\ul\nu}$ it suffices to prove the non-emptiness of $\Isom^{\otimes}(\omega_{\BO^{\ul\nu}}^\circ,\check{\CT}_{\ul{\CG}})$. Since the question only depends on the pullback of $\ul\CG$ to the geometric base point $\bar s$ we may assume that $S=\Spec\AlgClFld$ for an algebraically closed field $\AlgClFld$. For each place $\nu\in C'$ we consider the \'etale local $\wt\BP_\nu$-shtuka $L^+_\nu(\ul\CG)$ associated with $\ul\CG$ as in \cite[\RemTateF]{AH_Local}. By our assumption that $\FG$ is smooth with connected fibers, \cite[\CorEtIsTrivial]{AH_Local} yields an isomorphism $L^+_\nu(\ul\CG)\isoto\bigl((L^+\wt\BP_\nu)_\AlgClFld,1\!\cdot\!\hat\sigma^*\bigr)=:\ul\BL_\nu$. For the local $\wt\BP_\nu$-shtuka $\ul\BL_\nu$ the Tate functor $\check{\CT}_{\ul\BL_\nu}$ equals the forgetful functor $\Rep_{A_\nu}\BP_\nu\to\FM od_{A_\nu}$. Then $\omega_{\BO^{\ul\nu}}^\circ=\prod_{\nu\in C'}\check{\CT}_{\ul\BL_\nu}\cong\prod_{\nu\in C'}\check{\CT}_{L^+_\nu(\ul\CG)}\cong\check{\CT}_{\ul\CG}$ provides an isomorphism in $\Isom^{\otimes}(\omega_{\BO^{\ul\nu}}^\circ,\check{\CT}_{\ul{\CG}})$. If $\gamma,\gamma'\in\Isom^{\otimes}(\omega_{\BO^{\ul\nu}}^\circ,\check{\CT}_{\ul{\CG}})$ then $\gamma^{-1}\circ\gamma'\in \Aut^\otimes(\omega_{\BO^{\ul\nu}}^\circ)=\FG(\BO^{\ul\nu})$. This proves the lemma.
\end{proof}

\begin{corollary}\label{CorToLSisnonempty}
Let $m\in\BN_0$ be a multiple of $[\BF_{\ul\nu}\colon\BF_q]$. For every global $\FG$-shtuka $\ul\CG\in(\nabla_n\scrH^1(C,\FG)\times_{C^n}\Spec\BF_{\ul\nu})(S)$ the map 
\[
\tau_\CG^m\;:=\;\tau_\CG\circ\sigma^*\tau_\CG\circ\ldots\circ\sigma^{(m-1)*}\tau_\CG\colon\sigma^{m*}\ul\CG\;\longto\;\ul\CG
\]
is a quasi-isogeny, called the \emph{$q^m$-Frobenius isogeny} of $\ul\CG$. If $\gamma\in\Isom^{\otimes}(\omega^\circ,\check{\CV}_{\ul{\CG}})$ then $\sigma^{m*}(\gamma)=\check\CV_{\tau_\CG^m}^{-1}\circ\gamma$ in $\Isom^{\otimes}(\omega^\circ,\check{\CV}_{\sigma^{m*}\ul{\CG}})$.
\end{corollary}

\begin{proof}
The global $\FG$-shtuka $\sigma^{m*}\ul\CG$ is obtained by pulling back $\ul\CG$ under the absolute $q^m$-Frobenius $\sigma^m=\Frob_{q^m,S}\colon S\to S$. Note that the characteristic sections $s_i$ of $\ul\CG$ satisfy $s_i\circ\Frob_{q^m,S}=\Frob_{q^m,\BF_{\ul\nu}}\circ s_i=s_i$, because $(s_1,\ldots,s_n)\colon S\to C^n$ factors through $\Spec\BF_{\ul\nu}\in C^n$ and $\Frob_{q^m,\BF_{\ul\nu}}=\id_{\BF_{\ul\nu}}$. So $\ul\CG$ and $\sigma^{m*}\ul\CG$ have the same characteristic and $\tau_\CG^m$ is a quasi-isogeny which is an isomorphism over $(C\setminus\ul\nu)_S$ and satisfies $\tau_\CG\circ\sigma^*(\tau_\CG^m)=\tau_\CG^m\circ\tau_{\sigma^{m*}\CG}$.

The last equality can be proved as in Lemma~\ref{LSisnonempty}. Namely, by the proof of the lemma $\gamma$ arises as $\gamma=(\gamma_\nu)_\nu$ from a collection of quasi-isogenies $\gamma_\nu\colon\ul\BL_\nu:=\bigl((L^+\wt\BP_\nu)_\AlgClFld,1\!\cdot\!\hat\sigma^*\bigr)\to L^+_\nu(\ul\CG)$ of local $\wt\BP_\nu$-shtukas at every place $\nu\in C\setminus\ul\nu$. Then $\sigma^{m*}\gamma_\nu\colon\ul\BL_\nu=\sigma^{m*}\ul\BL_\nu\to L^+_\nu(\sigma^{m*}\ul\CG)$ satisfies $L^+_\nu(\tau_\CG^m)\circ\sigma^{m*}\gamma_\nu=\gamma_\nu\circ\tau_{\BL_\nu}^m=\gamma_\nu$, and hence $\check\CV_{\tau_\CG^m}\circ\sigma^{m*}(\gamma)=\gamma$.
\end{proof}

\begin{definition}\label{DefRatLevelStr}
\begin{enumerate}
\item 
For a compact open subgroup $H\subseteq \FG(\BA^{\ul\nu})$ we define a \emph{rational $H$-level structure} $\bar\gamma$ on a global $\FG$-shtuka $\ul\CG$ over a connected scheme $S\in\Nilp_{A_{\ul\nu}}$ as a $\pi_1^\et(S,\bar{s})$-invariant $H$-orbit $\bar\gamma=\gamma H$ in $\Isom^{\otimes}(\omega^\circ,\check{\CV}_{\ul{\CG}})$. For a non-connected scheme $S$ we make a similar definition choosing a base point on each connected component and a rational $H$-level structure on the restriction to each connected component separately.
\item 
We denote by $\nabla_n^H\scrH^1(C,\FG)^{\ul\nu}$ the category fibered in groupoids \lang{over $\Nilp_{A_\ul\nu}$} whose $S$-valued points $\nabla_n^H\scrH^1(C,\FG)^{\ul\nu}(S)$ is the category whose objects are tuples $(\ul\CG,\gamma H)$, consisting of a global $\FG$-shtuka $\ul\CG$ in $\nabla_n\scrH^1(C,\FG)^{\ul\nu}(S)$ together with a rational $H$-level structure $\gamma H$, and whose morphisms are quasi-isogenies of global $\FG$-shtukas that are isomorphisms at the characteristic places $\nu_i$ and are compatible with the $H$-level structures.  
\end{enumerate}
\end{definition}

This definition of level structures generalizes our initial Definition~\ref{Global Sht} according to the following

\begin{theorem}\label{H_DL-Str} 
Let $D\subset C$ be a finite subscheme disjoint from $\ul\nu$, and consider the compact open subgroup $H_D:=\ker\bigl(\FG(\BO^{\ul\nu})\to\FG(\CO_D)\bigr)$ of $\FG(\BA^{\ul\nu})$. Then there is a canonical isomorphism of stacks
$$
\nabla_n\scrH^1_D(C,\FG)^{\ul\nu}\enspace \isoto \enspace\nabla_n^{H_D}\scrH^1(C,\FG)^{\ul\nu}.
$$
\end{theorem}

\begin{proof}
Let $(\ul\CG, \psi)$ be an object in $\nabla_n\scrH^1_D(C,\FG)^{\ul\nu}(S)$ where $\psi$ is a $D$-level structure; see Definition~\ref{DefD-LevelStr}. For any representation $\rho$ in $\Rep_{\BO^{\ul\nu}}\FG$ the isomorphism $\psi\colon \ul\CG|_{D_S}\isoto \FG\times_C D_S$ with $\psi\circ\tau|_{D_S}=\sigma^*\psi$ induces an isomorphism $\rho_*\psi\colon\rho_\ast\ul\CG|_{D_S}\isoto \GL_{\dim\rho,D_S}$ with $\rho_*\psi^{-1}\colon\omega^\circ_{\BO^{\ul\nu}}(\rho)\otimes_{\BO^{\ul\nu}} \CO_D\isoto\check{\CT}_{\ul\CG}(\rho)\otimes_{\BO^{\ul\nu}} \CO_D$ and consequently we obtain an isomorphism $\bar{\gamma}\colon \omega^\circ_{\BO^{\ul\nu}}\otimes_{\BO^{\ul\nu}} \CO_D\isoto\check{\CT}_{\ul\CG}\otimes_{\BO^{\ul\nu}} \CO_D$ of tensor functors with $\bar\gamma(\rho):=\rho_*\psi^{-1}$. By Lemma~\ref{LSisnonempty} there exists a $\gamma\in\Isom^{\otimes}(\omega_{\BO^{\ul\nu}}^\circ,\check{\CT}_{\ul{\CG}})$ be a lift of $\bar{\gamma}$ to $\BO^{\ul\nu}$ and we let $\gamma_Q\in\Isom^{\otimes}(\omega^\circ,\check{\CV}_{\ul{\CG}})$ be the isomorphism induced from $\gamma$. Then the coset $\gamma_Q H_D$ only depends on $\psi$ and we define the morphism 
\begin{equation}\label{EqMissingTheorem}
\nabla_n\scrH^1_D(C,\FG)^{\ul\nu}(S) \longto \nabla_n^{H_D}\scrH^1(C,\FG)^{\ul\nu}(S)
\end{equation}
by sending $(\ul\CG, \psi)$ to $(\ul\CG,\gamma_Q H_D)$.

Let us show that this functor is essentially surjective. Let $(\ul\CG,\gamma_Q H_D)$ be an object of the category $\nabla_n^{H_D}\scrH^1(C,\FG)^{\ul\nu}$. By Lemma~\ref{LSisnonempty} we may choose an isomorphism \mbox{$\beta\colon\omega_{\BO^\ul\nu}^{\circ}\isoto \check{\CT}_{\ul\CG}$}. The automorphism $\beta_Q^{-1}\gamma_Q \in Aut^\otimes (\omega^\circ)$ corresponds to an element $g\in\FG(\BA^{\ul\nu})$ by Lemma~\ref{LSisnonempty}. We may write $g:=(g_{x_1}, \ldots, g_{x_r},g^{\ul x})\in \FG(Q_{x_1}) \times \ldots \times \FG(Q_{x_r}) \times \FG(\BO^{\ul\nu, \ul x})$ for suitable $r$ and $\ul x=(x_i)_i\in (C')^r$. Then we set $g':=(1,\ldots,1,g^{\ul x})^{-1}\cdot g=(g_{x_1},\ldots, g_{x_r},1)$ and $\beta':= \beta\cdot(1,\ldots,1,g^{\ul x})\in \Isom^{\otimes}(\omega_{\BO^{\ul\nu}}^\circ,\check{\CT}_{\ul{\CG}})$. For each $x_i$ let $L^+_{x_i}(\ul\CG)$ be the \'etale local $\wt\BP_{x_i}$-shtuka associated with $\ul\CG$ by \cite[\RemTateF]{AH_Local}. We consider the automorphism $\delta:=(\delta_{x_1},\ldots,\delta_{x_r},1):=\beta'_Q\cdot(g')^{-1}\cdot(\beta'_Q)^{-1}=\beta'_Q\gamma_Q^{-1} \in\Aut^\otimes (\check{\CV}_{\ul\CG})$ and in particular its components $\delta_{x_i}\in\Aut^\otimes\bigl(\check{\CV}_{L^+_{x_i}(\ul\CG)}\bigr)$. By \cite[\PropTateEquivP]{AH_Local} there is a quasi-isogeny $ f_i\colon L^+_{x_i}(\ul\CG)\to L^+_{x_i}(\ul\CG)$ with $\check{\CV}_{L^+_{x_i}(\ul\CG)}( f_i)=\delta_{x_i}$. 

From Proposition~\ref{PropLocalIsogeny} we obtain a global $\FG$-shtuka $\ul\CG'=  f_r^*\circ\ldots\circ  f_1^\ast\, \ul\CG$ together with a quasi-isogeny $f\colon\ul\CG'\to\ul\CG$ which is an isomorphism outside the $x_i$, such that the induced quasi-isogeny of the associated \'etale local $\wt\BP_{x_i}$-shtukas is $ f_i$. Then the rational Tate functor $\check{\CV}$ from \eqref{rationaltatefunctor} takes $f$ to
$$
\check{\CV}_f\,=\,\delta\,=\, \beta'_Q\cdot(g')^{-1}\cdot(\beta'_Q)^{-1}\,=\,\beta'_Q\gamma_Q^{-1}.
$$
Consider the pair $(\ul\CG',\beta_Q' H_D)$ consisting of the global $\FG$-shtuka $\ul\CG'=(\CG',\tau')$ together with the level structure $\beta'\colon \check{\CT}_{\ul\CG'}\to \omega_{\BO^{\ul\nu}}^\circ$ defined over $\BO^{\ul\nu}$. Note that this is quasi-isogenous to $(\ul\CG,\gamma_Q H_D)$ under $f$ and that $f$ is an isomorphism outside the $x_i$. 

We now show that $(\ul\CG',\beta_Q' H_D)=((\CG',\tauGlob'),\beta_Q' H_D)$ actually comes from a pair $(\ul\CG', \psi)$ in $\nabla_n\scrH^1_D(C,\FG)^{\ul\nu}$. Consider the category $\Rep_C\FG$ of representations of $\FG$ in finite locally free $\CO_C$-modules, the category $\Vect_C$ of finite locally free $\CO_C$-modules, and the natural forgetful functor $\omega^\circ_C\colon\Rep_C\FG\to\Vect_C$. The composition of the functor $\cdot\otimes_{\CO_C}\BO^{\ul\nu}\colon\Rep_C\FG\to\Rep_{\BO^{\ul\nu}}\FG$ followed by the fiber functor $\omega^\circ_{\BO^{\ul\nu}}$ equals $\omega^\circ_C\otimes_{\CO_C}\BO^{\ul\nu}$. Also consider the category $\FF\FM od_{D_S}$ of finite locally free sheaves on $D_S$, and the functor
$$
\CN^S_{-} \colon \scrH^1(C,\FG)(S)\to Funct^\otimes(\Rep_C\FG, \FF\FM od_{D_S}),\quad \CG\longmapsto\bigl(\CN^S_\CG\colon\rho\mapsto\rho_*\CG\otimes_{\CO_{C_S}}\CO_{D_S}\bigr).
$$ 
In particular $\CN^S_{\FG\times_C C_S}=\omega^\circ_C\otimes_{\CO_C}\CO_{D_S}$. For a representation $\rho\in\Rep_C\FG$ the ``finite shtuka'' $\rho_*\ul\CG'\times_{C_S}D_S$ over $S$ consists of the locally free sheaf $N:=\rho_*\CG'\otimes_{\CO_{C_S}}\CO_{D_S}=\CN^S_{\CG'}(\rho)$ on $D_S$ together with the isomorphism $\rho_*\tauGlob'\colon\s N\isoto N$. Let $\wt S\to S$ be the finite \'etale Galois covering corresponding to the kernel of $\pi_1^\et(S,\bar s)\to\Aut_{\CO_D}(\check\CT_{\ul\CG'}(\rho)\otimes_{\BO^{\ul\nu}}\CO_D)$, $g\mapsto(\beta')^{-1}\circ g(\beta')$. Then $\pi_1^\et(S,\bar s)$ acts on $\check\CT_{\ul\CG'}(\rho)\otimes_{\BO^{\ul\nu}}\CO_D$ through its quotient $\Gal(\wt S/S)$. As in the proof of \cite[\PropTateEquiv]{AH_Local} there is an isomorphism $\CN^{\wt S}_{\CG'}(\rho):=\rho_*\CG'\otimes_{\CO_{C_S}}\CO_{D_{\wt S}}\cong(\check\CT_{\ul\CG'}(\rho)\otimes_{\BO^{\ul\nu}}\CO_D)\otimes_{\BF_q}\CO_{\wt S}$ which is compatible with the action of $\Gal(\wt S/S)$ and the action of Frobenius through $\rho_*\tauGlob'$ on the left and $\id\otimes\sigma$ on the right. Therefore the $\pi_1^\et(S,\bar s)$-invariant coset $\beta'_Q H_D\subset\Isom^\otimes(\omega^\circ_{\BO^{\ul\nu}},\check\CT_{\ul\CG'})$ induces an isomorphism $\CN^{\wt S}_{\FG\times_C C_S}(\rho)\isoto\CN^{\wt S}_{\CG'}(\rho)$ which descends to $S$. In this way $\beta'$ induces an isomorphism of tensor functors $\eta\colon\CN^S_{\FG\times_C C_S}\isoto\CN^S_{\CG'}$.

We want to show that this isomorphism comes from a level structure $\psi$. Since $D$ is finite and $\FG$ is smooth there is an \'etale covering $S'\to S$ and a trivialization $\alpha\colon(\FG\times_C D_{S'},b\cdot\s)\isoto\ul\CG'\times_{C_S}D_{S'}$ for an element $b\in\FG(D_{S'})$, compare the proof of Theorem~\ref{nHisArtin}. The composition of tensor isomorphisms $\eta_{S'}^{-1}\circ\CN^{S'}_\alpha\in\Aut^\otimes(\CN^{S'}_{\FG\times_C C_S})\,=\,\Aut^\otimes(\omega^\circ_C)(\CO_{D_{S'}})\,=\,\FG(D_{S'})$ is given by an element $\psi'\in\FG(D_{S'})$ by \cite[Corollary~5.20]{Wed}. The latter induces an isomorphism $\psi'\colon(\FG\times_C D_{S'},b\cdot\s)\isoto(\FG\times_C D_{S'},1\cdot\s)$. This isomorphism descends to the desired level structure $\psi\colon \ul\CG'\times_{C_S}D_S\isoto(\FG\times_C D_S,1\cdot\s)$.

Analyzing this construction further also shows that the functor \eqref{EqMissingTheorem} is fully faithful. This proves the theorem.
\end{proof}

\begin{remark}\label{RemIndAlgStack}
(a) Let $(\ul\CG,\gamma H)\in\nabla_n^H\scrH^1(C,\FG)^{\ul\nu}(S)$. Then every choice of a representative $\gamma\in\Isom^{\otimes}(\omega^\circ,\check{\CV}_{\ul{\CG}})$ of the $H$-level structure $\gamma H$ induces a representation of the \'etale fundamental group
\begin{equation}\label{EqRepOfpi1}
\rho_{\ul\CG,\gamma}\colon \pi_1^\et(S,\bar s)\;\longto\; H\,,\quad g\;\longmapsto\;\gamma^{-1}\circ g(\gamma)\;=:\;\rho_{\ul\CG,\gamma}(g)\,.
\end{equation}
Indeed, $\rho_{\ul\CG,\gamma}(gg')=\gamma^{-1}\circ g(\gamma)\circ g\bigl(\gamma^{-1}\circ g'(\gamma)\bigr)=\rho_{\ul\CG,\gamma}(g)\cdot\rho_{\ul\CG,\gamma}(g')$, because $\gamma^{-1}\circ g'(\gamma)$ lies in $H$ on which $\pi_1^\et(S,\bar s)$ acts trivially. Replacing $\gamma$ by $\gamma h$ for $h\in H$ yields $\rho_{\ul\CG,\gamma}=\Int_h\circ\rho_{\ul\CG,\gamma h}$, where $\Int_h$ is conjugation by $h$ on $H$.

\medskip\noindent
(b) For any compact open subgroup $H\subseteq\FG(\BA^{\ul\nu})$ and any element $h\in\FG(\BA^{\ul\nu})$ there is an isomorphism $\nabla_n^H\scrH^1(C,\FG)^{\ul\nu}\;\isoto\;\nabla_n^{h^{-1}Hh}\scrH^1(C,\FG)^{\ul\nu},\;(\ul\CG,\gamma H)\mapsto\bigl(\ul\CG,\gamma h(h^{-1}Hh)\bigr)$. 
\end{remark}

\begin{theorem}\label{ThmLSGGsht1}
\begin{enumerate}
\item \label{ThmLSGGsht1_A}
For any compact open subgroup $H\subseteq\FG(\BA^{\ul\nu})$ the stack $\nabla_n^H\scrH^1(C,\FG)^{\ul\nu}$ is an ind-DM-stack, ind-separated and locally of ind-finite type over $\Spf A_{\ul\nu}$. 
\item \label{ThmLSGGsht1_B}
If $\wt H\subset H\subseteq\FG(\BA^{\ul\nu})$ are compact open subgroups then the forgetful morphism
\[
\nabla_n^{\wt H}\scrH^1(C,\FG)^{\ul\nu}\;\longto\;\nabla_n^H\scrH^1(C,\FG)^{\ul\nu},\quad (\ul\CG,\gamma\wt H)\;\longmapsto\;(\ul\CG,\gamma H)
\]
is finite \'etale and surjective. 
\item \label{ThmLSGGsht1_C}
If $\wt H$ is a normal subgroup of $H$ then the group $H/\wt H$ acts on $\nabla_n^{\wt H}\scrH^1(C,\FG)^{\ul\nu}$ from the right via $h\wt H\colon(\ul\CG,\gamma\wt H)\mapsto(\ul\CG,\gamma h\wt H)$ for $h\wt H\in H/\wt H$. The stack $\nabla_n^H\scrH^1(C,\FG)^{\ul\nu}$ is canonically isomorphic to the stack quotient $\bigl[\nabla_n^{\wt H}\scrH^1(C,\FG)^{\ul\nu}\big/(H/\wt H)\bigr]$ and $\nabla_n^{\wt H}\scrH^1(C,\FG)^{\ul\nu}$ is a right $H \slash \wt H$-torsor over $\nabla_n^H\scrH^1(C,\FG)^{\ul\nu}$ under the forgetful morphism. 
\end{enumerate}
\end{theorem}

\begin{proof}
\ref{ThmLSGGsht1_B} Let $(\ul\CG,\gamma H)\in\nabla_n^H\scrH^1(C,\FG)^{\ul\nu}(S)$ for a connected scheme $S$. Consider the finite $\pi_1^\et(S,\bar s)$-set 
\begin{equation}\label{EqFinitePi1Set}
\{\,\text{$\wt H$-level structures }\tilde\gamma\wt H\text{ on }\ul\CG\text{ with }\tilde\gamma\in\Isom^{\otimes}(\omega^\circ,\check{\CV}_{\ul{\CG}})\text{ satisfying }\tilde\gamma H=\gamma H\,\}\,.
\end{equation}
This set is finite, because after choosing a representative $\gamma\in\Isom^{\otimes}(\omega^\circ,\check{\CV}_{\ul{\CG}})$, it becomes bijective to $H/\wt H$ under the map $h\mapsto\gamma h\wt H$ for $h\in H/\wt H$; see Lemma~\ref{LSisnonempty}. Let $S'\to S$ be the finite \'etale covering space corresponding to \eqref{EqFinitePi1Set}, that is 
\[
F^\et_{S,\bar s}(S')\;:=\;S'\times_S\bar s\;=\;\{\,\text{$\wt H$-level structures }\tilde\gamma\wt H\text{ on }\ul\CG\text{ satisfying }\tilde\gamma H=\gamma H\,\}\,,
\]
where $F^\et_{S,\bar s}\colon S'\mapsto S'\times_S\bar s$ denotes the fiber functor from finite \'etale covering spaces of $S$ to finite $\pi_1^\et(S,\bar s)$-sets. The choice of a representative $\gamma$ corresponds to an element $\gamma\wt H\in F^\et_{S,\bar s}(S')$, and hence to the choice of a base point $\bar s'$ of $S'$ lifting $\bar s$. Then $\gamma\wt H$ is $\pi_1^\et(S',\bar s')$-invariant, and hence a rational $\wt H$-level structure on $\ul\CG$ over $S'$. This defines an $S$-morphism 
\begin{equation}\label{EqEtaleCovSp}
S'\;\longto\;\nabla_n^{\wt H}\scrH^1(C,\FG)^{\ul\nu}\times_{\nabla_n^H\scrH^1(C,\FG)^{\ul\nu}}S
\end{equation}
which we will show to be an isomorphism. For this purpose let $f\colon T\to S$ be a connected scheme over $S$. A $T$-valued point of the fiber product in \eqref{EqEtaleCovSp} is given by a rational $\wt H$-level structure $\tilde\gamma\wt H$ on $f^*\ul\CG$ lifting $\gamma H$. This means that $\tilde\gamma\in\Isom^{\otimes}(\omega^\circ,\check{\CV}_{f^*\ul{\CG},\bar t})$ and the $\wt H$-orbit $\tilde\gamma\wt H$ is $\pi_1^\alg(T,\bar t)$-invariant where $\bar t$ is a geometric base point of $T$ and $\check{\CV}_{f^*\ul{\CG},\bar t}$ is the fiber over $\bar t$. We must show that $\tilde\gamma\wt H$ arises from a uniquely determined $S$-morphism $T\to S'$. Moving the base point $\bar s$ we may assume that $f(\bar t)=\bar s$, and hence $\check{\CV}_{f^*\ul{\CG},\bar t}=\check{\CV}_{\ul{\CG}}$. Consider the finite \'etale covering space $S'\times_S T\to T$. Then 
\[
F^\et_{T,\bar t}(S'\times_S T)\;=\;F^\et_{S,\bar s}(S')\;=\;\{\,\text{$\wt H$-level structures }\tilde\gamma\wt H\text{ on }\ul\CG\text{ satisfying }\tilde\gamma H=\gamma H\,\}\,.
\]
The $\wt H$-level structure $\tilde\gamma\wt H$ on $f^*\ul\CG$ is an element of this set, because $\tilde\gamma\wt H$ lifts $\gamma H$, that is $\tilde\gamma H=\gamma H$. In particular, this element $\tilde\gamma\wt H$ defines a $\pi_1^\et(T,\bar t)$-equivariant map from the one-element set $\{\bar t\}=F^\et_{T,\bar t}(T)$ to $F^\et_{T,\bar t}(S'\times_S T)$. By \cite[Expos\'e~V, \S\,7 and Th\'eor\`eme~4.1]{SGA1} this map corresponds to a uniquely determined $T$-morphism $T\to S'\times_S T$. The projection $T\to S'$ onto the first component is the desired $S$-morphism which induces the rational $\wt H$-level structure $\tilde\gamma\wt H$ over $T$.

This proves that \eqref{EqEtaleCovSp} is an isomorphism. And therefore the forgetful morphism $\nabla_n^{\wt H}\scrH^1(C,\FG)^{\ul\nu}\longto\nabla_n^H\scrH^1(C,\FG)^{\ul\nu}$ is finite \'etale and surjective.

\medskip\noindent
\ref{ThmLSGGsht1_C} Let us next assume that $\wt H\subset H$ is normal. Then Remark~\ref{RemIndAlgStack}(b) defines a right action of $H/\wt H$ on $\nabla_n^{\wt H}\scrH^1(C,\FG)^{\ul\nu}$. By definition the stack quotient $\bigl[\nabla_n^{\wt H}\scrH^1(C,\FG)^{\ul\nu}\big/(H/\wt H)\bigr]$ is the category fibered in groupoids \lang{over $\Nilp_{A_\ul\nu}$} whose $S$-valued points are pairs $\bigl(S'\to S, (\ul\CG,\gamma\wt H)\bigr)$, where $S'\to S$ is an $H/\wt H$-torsor and $(\ul\CG,\gamma\wt H)\in\nabla_n^{\wt H}\scrH^1(C,\FG)^{\ul\nu}(S')$ defines a morphism $S'\to\nabla_n^{\wt H}\scrH^1(C,\FG)^{\ul\nu}$ which is equivariant for the action of $H/\wt H$; see \cite[Tag~\href{http://stacks.math.columbia.edu/tag/04WL}{04WL}]{StacksProject}. The isomorphism
\begin{equation}\label{EqStackQuot}
\nabla_n^H\scrH^1(C,\FG)^{\ul\nu}\;\isoto\;\bigl[\nabla_n^{\wt H}\scrH^1(C,\FG)^{\ul\nu}\big/(H/\wt H)\bigr]\,,\quad (\ul\CG,\gamma H)\;\longmapsto\;\bigl(S'\to S,(\ul\CG,\gamma\wt H)\bigr)
\end{equation}
is given as follows. Let $S':=\nabla_n^{\wt H}\scrH^1(C,\FG)^{\ul\nu}\times_{\nabla_n^H\scrH^1(C,\FG)^{\ul\nu}}\,S$. By the proof of \ref{ThmLSGGsht1_B} this means that $S'\to S$ is the finite \'etale covering space corresponding to the $\pi_1^\et(S,\bar s)$-set from \eqref{EqFinitePi1Set}. This set is in bijection with $H/\wt H$ on which $\pi_1^\et(S,\bar s)$ acts via the representation $\rho_{\ul\CG,\gamma}\colon \pi_1^\et(S,\bar s)\to H\onto H/\wt H$ from \eqref{EqRepOfpi1}. In particular, $S'\to S$ is a Galois covering with Galois group $H/\wt H$. The projection $S'\to\nabla_n^{\wt H}\scrH^1(C,\FG)^{\ul\nu}$ onto the first factor is $H/\wt H$-equivariant and establishes the map \eqref{EqStackQuot}. Its inverse is defined as follows. The object $(\ul\CG,\gamma\wt H)\in\nabla_n^{\wt H}\scrH^1(C,\FG)^{\ul\nu}(S')$ defines the object $(\ul\CG,\gamma H)\in\nabla_n^H\scrH^1(C,\FG)^{\ul\nu}(S')$ which descends to $S$ by \'etale descent for $S'\to S$. By construction the canonical morphism $\nabla_n^{\wt H}\scrH^1(C,\FG)^{\ul\nu}\;\isoto\;\bigl[\nabla_n^{\wt H}\scrH^1(C,\FG)^{\ul\nu}\big/(H/\wt H)\bigr]$ is finite \'etale surjective and an $H/\wt H$-torsor.

\medskip\noindent
\ref{ThmLSGGsht1_A} The intersection $H_1:=H\cap\FG(\BO^{\ul\nu})$ has finite index in $H$, because it is open and $H$ is compact. Thus the intersection $H_2:=\bigcap_{h\in H/H_1}h H_1 h^{-1}\subset H_1\subset \FG(\BO^{\ul\nu})$ is compact open, normal in $H$, and of finite index in $\FG(\BO^{\ul\nu})$. There is a proper closed subscheme $D\subset C$ with $H_D\subset H_2$, and this is a normal subgroup because $H_D$ is normal in $\FG(\BO^{\ul\nu})$. Therefore, statement \ref{ThmLSGGsht1_A} holds for $\nabla_n^{H_D}\scrH^1(C,\FG)^{\ul\nu}$ by Theorems~\ref{H_DL-Str} and \ref{nHisArtin}, see Definition~\ref{DefGlobalLocalFunctor}. Consequently also $\nabla_n^{H_2}\scrH^1(C,\FG)^{\ul\nu}$ and $\nabla_n^H\scrH^1(C,\FG)^{\ul\nu}$ are ind-DM-stacks locally of ind-finite type over $\Spf A_{\ul\nu}$ by \ref{ThmLSGGsht1_C} because they are obtained as stack quotients by finite groups. They are ind-separated over $\Spf A_{\ul\nu}$, because the forgetful morphisms in \ref{ThmLSGGsht1_C} are finite surjective with ind-separated source.
\end{proof}

%
%

\section{The Uniformization Theorem in the Bounded Case}\label{Uniformization Theorem}
\setcounter{equation}{0}

We still assume that the group scheme $\FG$ is smooth over $C$ and all its fibers are connected. In addition we assume that the generic fiber of $\FG$ is reductive. Our aim in this chapter is to study the morphism $\Theta$ from \eqref{EqUnifIntro} in the introduction. We first describe its target.

\begin{definition}\label{bdglobal}
Fix an $n$-tuple $\ul\nu=(\nu_i)$ of places on the curve $C$ with $\nu_i\ne\nu_j$ for $i\ne j$. Let  $\ulHZ:=(\hat{Z}_i)_i$ be an $n$-tuple of bounds $\hat{Z}_i$ in $\wh{\SpaceFl}_{\BP_{\nu_i}}$ in the sense of Definition~\ref{DefBDLocal} with reflex rings $R_i:=R_{\hat{Z}_i}=\kappa_i\dbl\xi_i\dbr$. Let $\kappa$ be the compositum of the $\kappa_i$ in an algebraic closure of $\BF_q$, and set $R_{\ulHZ}:=\kappa\dbl\xi_1,\ldots,\xi_n\dbr$. It is a finite extension of $A_{\ul\nu}$. Let $\ul\CG$ be a global $\FG$-shtuka in $\nabla_n\scrH^1(C,\FG)^{\ul\nu}(S)$ over a scheme $S\in\Nilp_{R_{\ulHZ}}$. Then $S$ is a scheme in $\Nilp_{R_i}$ via the inclusion $\kappa_i\dbl\xi_i\dbr\into R_{\ulHZ}$. We say that \emph{$\ul\CG$ is bounded by $\ulHZ:=(\hat{Z}_i)_i$} if for every $i$ the associated local $\BP_{\nu_i}$-shtuka $\wh\Gamma_{\nu_i}(\ul\CG)$ is bounded by $\hat{Z}_i$. We denote by $\nabla_n^{H,\ulHZ}\scrH^1(C,\FG)^{\ul\nu}$ the substack of $\nabla_n^H\scrH^1(C,\FG)^{\ul\nu}\whtimes_{A_\ul\nu}\Spf R_{\ulHZ}$ consisting of global $\FG$-shtukas bounded by $\ulHZ$. 
\end{definition}

\begin{proposition}\label{PropSpecialFiberIsDM}
The special fiber $\nabla_n^{H,\ulHZ}\scrH^1(C,\FG)^{\ul\nu}\times_{R_{\ulHZ}}\Spec\kappa$ is a Deligne-Mumford stack locally of finite type and separated over $\Spec\kappa$. 
\end{proposition}

\begin{proof}
As in the proof of Theorem~\ref{ThmLSGGsht1}\ref{ThmLSGGsht1_A} it suffices to treat the case where $H=H_D\subset\FG(\BA^{\ul\nu})$, because taking stack quotients by finite groups preserves the $2$-category of Deligne-Mumford stacks locally of finite type and separated over $\Spec\kappa$. We fix a representation $\rho\colon\FG\into\SL(\CV_0)$ with $\rk\CV_0=:r$ as before Remark~\ref{RemRelAffineGrass}. Then we show that $\nabla_n^{H_D,\ulHZ}\scrH^1(C,\FG)^{\ul\nu}\times_{R_{\ulHZ}}\Spec\kappa$ is a closed substack of $\nabla_n^{\ul\omega}\scrH_D^1(C,\FG)\times_{(C\setminus D)^n}\Spec\kappa$ for a suitable tuple $\ul\omega=(\omega_i)_{i=1\ldots n}$ of coweights of $\SL_r$; see Definition~\ref{DefNablaOmegaH}. The representation $\rho$ induces an immersion $\SpaceFl_{\BP_{\nu_i}}\into\SpaceFl_{\SL_r}$ of the associated affine flag varieties. Let $z_i$ be a uniformizing parameter of $A_{\nu_i}$. Since $Z_i\subset\SpaceFl_{\BP_{\nu_i}}\times_{\BF_{\nu_i}}\Spec\kappa_i$ is a quasi-compact scheme, there is an integer $N_i\ge0$ such that all entries of the universal matrix over $L\SL_r\times_{\SpaceFl_{\SL_r}}Z_i$ have pole order at most $N_i$ with respect to $z_i$. We set $\omega_{i,\ell}:=-N_i$ for all $\ell=2,\ldots,r$ and $\omega_{i,1}:=(r-1)N_i$. Then every global $\FG$-shtuka $\ul\CG$ over a $\kappa$-scheme $S$ which is bounded by $\ulHZ$ belongs to $\nabla_n^{\ul\omega}\scrH^1(C,\FG)(S)$. By Proposition~\ref{PropBoundedClosed} (and Theorem~\ref{H_DL-Str}) the stack $\nabla_n^{H_D,\ulHZ}\scrH^1(C,\FG)^{\ul\nu}\times_{R_{\ulHZ}}\Spec\kappa$ is a closed substack of $\nabla_n^{\ul\omega}\scrH_D^1(C,\FG)\times_{(C\setminus D)^n}\Spec\kappa$, and hence a Deligne-Mumford stack locally of finite type and separated over $\Spec\kappa$ by Theorem~\ref{nHisArtin}.
\end{proof}

\begin{remark}\label{Rem5.2}
Likewise, by Proposition~\ref{PropBoundedClosed} the entire stack $\nabla_n^{H,\ulHZ}\scrH^1(C,\FG)^{\ul\nu}$ is a closed ind-substack of $\nabla_n^H\scrH^1(C,\FG)^{\ul\nu}$, and hence an ind-DM-stack over $\Spf R_{\ulHZ}$ which is ind-separated and locally of ind-finite type by Theorem~\ref{ThmLSGGsht1}. In contrast to Proposition~\ref{PropSpecialFiberIsDM} we do not know whether $\nabla_n^{H,\ulHZ}\scrH^1(C,\FG)^{\ul\nu}$ is the formal completion of a Deligne-Mumford stack (or at least a formal algebraic Deligne-Mumford stack over $\Spf R_{\ulHZ}$) without further conditions on the $\hat{Z}_i$. If we fix a representation $\rho\colon\FG\into\SL(\CV_0)$ with $\rk\CV_0=:r$ as before Remark~\ref{RemRelAffineGrass} and let $\ul\omega=(\omega_i)_{i=1\ldots n}$ be a tuple of coweights of $\SL_r$, see Definition~\ref{DefNablaOmegaH}, we can consider the closed substack $\nabla_n^{H,\ulHZ,\ul\omega}\scrH^1(C,\FG)^{\ul\nu}$ of $\nabla_n^{H,\ulHZ}\scrH^1(C,\FG)^{\ul\nu}$ on which $\rho_*\ul\CG$ is bounded by $\ul\omega$ as in \eqref{EqBounded1}. If $H=H_D$ then $\nabla_n^{H_D,\ulHZ,\ul\omega}\scrH^1(C,\FG)^{\ul\nu}\;=\;\nabla_n^{H_D,\ulHZ}\scrH^1(C,\FG)^{\ul\nu}\times_{\nabla_n\scrH^1(C,\FG)}\nabla_n^{\ul\omega}\scrH^1(C,\FG)$. By arguing as in Proposition~\ref{PropSpecialFiberIsDM} the stack $\nabla_n^{H,\ulHZ,\ul\omega}\scrH^1(C,\FG)^{\ul\nu}$ is a locally noetherian, adic formal algebraic Deligne-Mumford stack over $\Spf \breve R_{\ul{\hat Z}}$; use \cite[Proposition~A.14]{Har1}. In particular, $\nabla_n^{H,\ulHZ,\ul\omega}\scrH^1(C,\FG)^{\ul\nu}\whtimes_{\Spf \breve R_{\ul{\hat Z}}}\Spec\breve R_{\ul{\hat Z}}/(\xi_1,\ldots,\xi_n)^j$ is an (algebraic) Deligne-Mumford stack separated and locally of finite type over $\Spec\breve R_{\ul{\hat Z}}/(\xi_1,\ldots,\xi_n)^j$ for every $j$. So similarly to Definition~\ref{DefGlobalLocalFunctor}, the structure of ind-DM-stack on $\nabla_n^{H,\ulHZ}\scrH^1(C,\FG)^{\ul\nu}$ is given as the limit
\[
\nabla_n^{H,\ulHZ}\scrH^1(C,\FG)^{\ul\nu}\;=\;\dirlim[\ul\omega,j]\nabla_n^{H,\ulHZ,\ul\omega}\scrH^1(C,\FG)^{\ul\nu}\whtimes_{\Spf \breve R_{\ul{\hat Z}}}\Spec\breve R_{\ul{\hat Z}}/(\xi_1,\ldots,\xi_n)^j.
\]

We do not need this here, but $\nabla_n^{H,\ulHZ}\scrH^1(C,\FG)^{\ul\nu}$ will be a formal algebraic Deligne-Mumford stack over $\Spf R_{\ulHZ}$ if we impose a condition of the following form. 
\begin{itemize}
\item There is a faithful representation $\rho\colon \FG\into\SL(\CV_0)$ for a vector bundle $\CV_0$ of rank $r$ on $C$ such that on $\hat{Z}_i$ the universal matrix induced from the morphism $\SpaceFl_{\BP_{\nu_i}}\into\SpaceFl_{\SL_r}$ has poles with respect to $\Gamma_{s_i}$ bounded by some integer $N_i$ for every $\nu_i\in\ul\nu$.
\end{itemize}
Without such a condition on $\ulHZ$ the problem is the following. $\nabla_n^{H,\ulHZ}\scrH^1(C,\FG)^{\ul\nu}$ is the limit of an inductive system of algebraic stacks. One has to find a compatible inductive system of presentations for these algebraic stacks such that this inductive system forms an ind-scheme. Then one could argue as in \cite[\RemBdIsFormSch]{AH_Local} that this ind-scheme is in fact a formal scheme and consequently $\nabla_n^{H,\ulHZ}\scrH^1(C,\FG)^{\ul\nu}$ is a formal algebraic stack of DM-type in the sense of \cite[Appendix~A]{Har1}. 
\end{remark}

\begin{point}\label{Point7.4}
We keep the notation of Definition~\ref{bdglobal} and let $\BaseFldInSectUnif$ be an algebraic closure of the residue field $\BF_{\ul\nu}$ of $\ul\nu\in C^n$. Then $(\charsect_1,\ldots,\charsect_n)\colon\Spec\BaseFldInSectUnif\to C^n$ has characteristic $\ul\nu$. Let $\breve R_{\ulHZ}:=\BaseFldInSectUnif\dbl\xi_1,\ldots,\xi_n\dbr$ be the completion of the maximal unramified extension of $R_{\ulHZ}$. Let $\ul{\CG}_0$ be a global $\FG$-shtuka over $\BaseFldInSectUnif$, and let $I_{\ul\CG_0}\!(Q)$ denote the (abstract) group $\QIsog_\BaseFldInSectUnif(\ul\CG_0)$ of quasi-isogenies of $\ul\CG_0$; see Definition~\ref{quasi-isogenyGlob}.
\end{point}

\begin{proposition}\label{Prop7.1}
Let $S$ be a connected non-empty scheme in $\Nilp_{\breve R_{\ulHZ}}$. Then the natural group homomorphism $\QIsog_\BaseFldInSectUnif(\ul\CG_0)\to\QIsog_{S}(\ul\CG_{0,S})$, $g\mapsto g_{S}$ is an isomorphism. 
\end{proposition}

\begin{proof}
Since $S$ is non-empty the group homomorphism is injective. To prove surjectivity we consider a quasi-isogeny $f\in\QIsog_{S}(\ul\CG_{0,S})$. Since the question is local on $S$ we may assume that $S=\Spec B$ is affine. Let $D\subset C$ be a divisor such that $f$ is an automorphism of $\CG_{0,S}|_{C_{S}\setminus D_{S}}$. We set $C':=C\setminus(D\cup\ul\nu)$ and choose a finite flat ring homomorphism $\pi^*\colon\BF_q[t]\into\CO_C(C')$ of some degree $d$. We also choose a representation $\rho\colon\FG\into\GL(\CV)$ for a locally free sheaf $\CV$ on $C$ of rank $r$ as in Proposition~\ref{PropGlobalRep}\ref{PropGlobalRep_A}. The restriction to $C'_\BaseFldInSectUnif$ of the global $\GL_r$-shtuka $\rho_*\ul\CG_0$ corresponds to a locally free sheaf $M$ of rank $r$ on $C'_\BaseFldInSectUnif$ together with an isomorphism $\tau\colon\sigma^*M\isoto M$, because $C'\cap\ul\nu=\emptyset$. The quasi-isogeny $\rho_*f$ corresponds to an automorphism $f\colon M\otimes_\BaseFldInSectUnif B\isoto M\otimes_\BaseFldInSectUnif B$ with $f\tau=\tau\sigma^*(f)$. Via $\pi^*$ we view $M$ as a (locally) free $\BaseFldInSectUnif[t]$-module of rank $rd$. We choose a basis and write $\tau$ and $f$ as matrices $T\in\GL_{rd}(\BaseFldInSectUnif[t])$ and $F\in\GL_{rd}(B[t])$ satisfying $FT=T\sigma^*(F)$. We expand $T=\sum_iT_it^i$ and $T^{-1}=\sum_iT'_it^i$ and $F=\sum_i F_it^i$ in powers of $t$. Then $FT=T\sigma^*(F)$ is equivalent to the equations $F_m=\sum_{i+j+\ell=m}T_i\sigma^*(F_j)T'_\ell$ and $\sigma^*(F_m)=\sum_{i+j+\ell=m}T'_iF_jT_\ell$ for all $m$, where the matrix $\sigma^*(F_m)$ is obtained from $F_m$ by raising all entries to the $q$-th power. This shows that the entries of the matrices $F_m$ satisfy finite \'etale equations over $\BaseFldInSectUnif$. As $\BaseFldInSectUnif$ is algebraically closed and $S$ is connected we must have $F\in\GL_{rd}(\BaseFldInSectUnif[t])\subset\GL_{rd}(B[t])$. This implies that $f=g_{S}$ for a quasi-isogeny $g\in\QIsog_{\BaseFldInSectUnif}(\ul\CG_0)$.
\end{proof}

\begin{point}\label{Point7.6}
We next describe the source of the morphism $\Theta$ from \eqref{EqUnifIntro}. Keep the situation of \ref{Point7.4} and let $(\ul{\BL}_i)_{i=1\ldots n}:=\wh{\ul{\Gamma}}(\ul\CG_0)$ denote the associated tuple of local $\BP_{\nu_i}$-shtukas over $\BaseFldInSectUnif$, where $\wh{\ul{\Gamma}}$ is the global-local functor from Definition~\ref{DefGlobalLocalFunctor}. Let $\CM_{\ul\BL_i}^{\hat{Z}_i}$ denote the Rapoport-Zink space of $\ul\BL_i$ with underlying topological space $X_{Z_i}(\ul\BL_i)$, see Theorem~\ref{ThmRRZSp}. Let $\breve\CM_{\ul\BL_i}^{\hat{Z}_i}$ and $\breve X_{Z_i}(\ul\BL_i)$ be their base changes to $\breve R_i:=\BaseFldInSectUnif\dbl\xi_i\dbr$, respectively to $\BaseFldInSectUnif$. By Theorem~\ref{ThmRRZSp} the product $\prod_i \breve\CM_{\ul{\BL}_i}^{\hat{Z}_i}:=\breve\CM_{\ul{\BL}_1}^{\hat{Z}_1}\whtimes_\BaseFldInSectUnif\ldots\whtimes_\BaseFldInSectUnif\breve\CM_{\ul{\BL}_n}^{\hat{Z}_n}$ is a formal scheme locally formally of finite type over $\Spf\breve R_1\whtimes_\BaseFldInSectUnif\ldots\whtimes_\BaseFldInSectUnif\Spf\breve R_n=\Spf\BaseFldInSectUnif\dbl\xi_1,\ldots,\xi_n\dbr=\Spf \breve R_{\ulHZ}$ and the product $\prod_i \breve X_{Z_i}(\ul\BL_i):=\breve X_{Z_1}(\ul\BL_1)\times_\BaseFldInSectUnif\ldots\times_\BaseFldInSectUnif\breve X_{Z_n}(\ul\BL_n)$ is a scheme locally of finite type over $\Spec\BaseFldInSectUnif$. 

Recall that the group $J_{\ul{\BL}_i}(Q_{\nu_i})=\QIsog_\BaseFldInSectUnif(\ul{\BL}_i)$ of quasi-isogenies of $\ul\BL_i$ over $\BaseFldInSectUnif$ acts naturally on $\breve\CM_{\ul{\BL}_i}^{\hat{Z}_i}$ and on $\breve X_{Z_i}(\ul\BL_i)$; see Remark~\ref{RemJActsOnRZ}. Especially we see that the group $I_{\ul\CG_0}\!(Q)$ acts on $\prod_i \breve\CM_{\ul{\BL}_i}^{\hat{Z}_i}$ and $\prod_i \breve X_{Z_i}(\ul\BL_i)$ via the natural group homomorphism 
\begin{equation}\label{EqI(Q)}
I_{\ul\CG_0}\!(Q) \;\longto\; \prod_i J_{\ul{\BL}_i}(Q_{\nu_i}),\quad \eta\;\longmapsto\; (\eta_i)_i\;:=\;\ul{\wh\Gamma}(\eta)\;:=\;\bigl(\wh\Gamma_{\nu_i}(\eta)\bigr)_i
\end{equation}
by sending $(\ul\CL_i,\delta_i)_i$ to $(\ul\CL_i,\wh\Gamma_{\nu_i}(\eta)\circ\delta_i)_i$.

The group $I_{\ul\CG_0}\!(Q)$ also acts naturally on $\check{\CV}_{\ul{\CG}_0}$ and $\Isom^{\otimes}(\omega^\circ,\check{\CV}_{\ul{\CG}_0})$ by sending $\gamma\in\Isom^{\otimes}(\omega^\circ,\check{\CV}_{\ul{\CG}_0})$ to $\check\CV_\eta\circ\gamma$ for $\eta\in I_{\ul\CG_0}\!(Q)$. After choosing an element $\gamma_0\in\Isom^{\otimes}(\omega_{\BO^{\ul\nu}}^\circ,\check{\CT}_{\ul{\CG}_0})$, this defines a morphism 
\begin{equation}\label{EqEpsilon}
\epsilon\colon  I_{\ul\CG_0}\!(Q) \;\longto\; Aut^\otimes(\omega^\circ)\;\cong\;\FG(\BA^{\ul\nu}),\quad\eta\mapsto\gamma_0\circ\check{\CV}_\eta\circ\gamma_0^{-1}.
\end{equation}

To state the properties of the source of $\Theta$ we say that a formal algebraic Deligne-Mumford stack $\CX$ over $\Spf\breve R_{\ulHZ}$ (see \cite[Definition~A.5]{Har1}) is \emph{$\CJ$-adic} for a sheaf of ideals $\CJ\subset\CO_\CX$, if for some (any) presentation $X\to\CX$ the formal scheme $X$ is $\CJ\CO_X$-adic, that is, $\CJ^r\CO_X$ is an ideal of definition of $X$ for all $r$. We then call $\CJ$ an \emph{ideal of definition of $\CX$}. We say that $\CX$ is \emph{locally formally of finite type} if $\CX$ is locally noetherian, adic, and if the closed substack defined by the largest ideal of definition (see \cite[A.7]{Har1}) is an algebraic stack locally of finite type over $\Spec\BaseFldInSectUnif$.
\end{point}

\begin{proposition}\label{PropQuotientByI}
The quotient of $\prod_i \breve\CM_{\ul\BL_i}^{\hat{Z}_i}\times \Isom^{\otimes}(\omega^\circ,\check{\CV}_{\ul{\CG}_0})/H$ by the abstract group $I_{\ul\CG_0}\!(Q)$ exists as a locally noetherian, adic formal algebraic Deligne-Mumford stack locally formally of finite type over $\Spf\breve R_{\ulHZ}$ and is of the form
\begin{equation}\label{EqSourceOfTheta}
I_{\ul\CG_0}\!(Q) \big{\backslash}\bigl(\prod_i \breve\CM_{\ul\BL_i}^{\hat{Z}_i}\times \Isom^{\otimes}(\omega^\circ,\check{\CV}_{\ul{\CG}_0})/H\bigr) \enspace\cong\enspace \coprod_{\bar\gamma} \Gamma_{\bar\gamma}\big{\backslash}\prod_i \breve\CM_{\ul\BL_i}^{\hat{Z}_i}\,.
\end{equation}
Here $\bar\gamma:=\gamma H\in\Isom^{\otimes}(\omega^\circ,\check{\CV}_{\ul{\CG}_0})/H$ runs through a set of representatives for the countable double coset $I_{\ul\CG_0}\!(Q) \backslash\Isom^{\otimes}(\omega^\circ,\check{\CV}_{\ul{\CG}_0})/H\cong \epsilon\bigl(I_{\ul\CG_0}\!(Q)\bigr) \backslash \FG(\BA^{\ul\nu})/H$, and 
$$
\Gamma_{\bar\gamma}\;:=\; I_{\ul\CG_0}\!(Q)\cap \bigl(\prod_i J_{\ul{\BL}_i}(Q_{\nu_i})\times \gamma H \gamma^{-1}\bigr)\;\subset\; \prod_i J_{\ul{\BL}_i}(Q_{\nu_i})
$$
is a subgroup, which is discrete for the product of the $\nu_i$-adic topologies, and separated in the profinite topology, that is, for every $1\ne g \in \Gamma_{\bar\gamma}$ there is a normal subgroup of finite index in $\Gamma_{\bar\gamma}$ that does not contain $g$. In particular the closed substack of \eqref{EqSourceOfTheta} defined by the largest ideal of definition is the Deligne-Mumford stack locally of finite type over $\Spec\BaseFldInSectUnif$ given by
\begin{equation}\label{EqReducedSourceOfTheta}
I_{\ul\CG_0}\!(Q) \big{\backslash}\bigl(\prod_i \breve X_{Z_i}(\ul\BL_i)\times \Isom^{\otimes}(\omega^\circ,\check{\CV}_{\ul{\CG}_0})/H\bigr) \enspace\cong\enspace \coprod_{\bar\gamma} \Gamma_{\bar\gamma}\big{\backslash}\prod_i \breve X_{Z_i}(\ul\BL_i)\,.
\end{equation}
\end{proposition}

\begin{proof}
We first prove that the group homomorphisms $\wh\Gamma_{\nu_i}$ and $\epsilon$ from \eqref{EqI(Q)} and \eqref{EqEpsilon} are injective. We consider a faithful representation $\rho\colon\FG\into\GL(\CV)$ as in Proposition~\ref{PropGlobalRep}\ref{PropGlobalRep_A}. Then $\rho_*\eta$ induces a quasi-isogeny of the vector bundle $M$ associated with $\rho_*\ul\CG_0$. If $\eta$ lies in the kernel of $\epsilon$ then its restriction to $C_\BaseFldInSectUnif\setminus\ul\nu$ is the identity on $M$ because of \cite[\PropTateEquiv]{AH_Local}. Therefore $\eta$ must be the identity. This proves that $\epsilon$ is injective. On the other hand, $\wh\Gamma_{\nu_i}(M)$ is the completion of $M$ at the graph $\Gamma_{s_i}$ of $s_i\colon\Spec\BaseFldInSectUnif\to C$. Since this completion functor is faithful, also $\wh\Gamma_{\nu_i}\colon I_{\ul\CG_0}\!(Q) \to J_{\ul{\BL}_i}(Q_{\nu_i})$ is injective. Thus we get the following injective morphism
$$
(\ul{\wh\Gamma},\epsilon)\colon I_{\ul\CG_0}\!(Q)\longto \prod_i J_{\ul{\BL}_i}(Q_{\nu_i})\times \FG(\BA^{\ul\nu})
$$
and we identify $I_{\ul\CG_0}\!(Q)$ with its image. We claim that $I_{\ul\CG_0}\!(Q)$ is a discrete subgroup of $\prod_i J_{\ul\BL_i}(Q_{\nu_i})\times \FG(\BA^{\ul\nu})$. To show this we take an open subgroup $U\subset\prod_i\Aut_\BaseFldInSectUnif(\ul\BL_i)$ and consider the open subgroup $U\times\FG(\BO^{\ul\nu})\subset \prod_i J_{\ul\BL_i}(Q_{\nu_i})\times \FG(\BA^{\ul\nu})$. Since $\FG(\BO^{\ul\nu})=\gamma_0\Aut^{\otimes}(\check{\CT}_{\ul{\CG}_0})\gamma_0^{-1}$ the elements of $I_{\ul\CG_0}\!(Q)\cap \bigl(U\times\FG(\BO^{\ul\nu})\bigr)$ give automorphisms of the global $\FG$-shtuka $\ul\CG_0:=(\CG,\tau_0)$. Then the finiteness of $I_{\ul\CG_0}\!(Q)\cap \bigl(U\times\FG(\BO^{\ul\nu})\bigr)$ follows from Corollary~\ref{CorAutFinite}. 

Now choose an element $\bar\gamma:=\gamma H\in \Isom^{\otimes}(\omega^\circ,\check{\CV}_{\ul{\CG}_0})/H$, set $h:=\gamma_0^{-1}\gamma\in\FG(\BA^{\ul\nu})$ and let
\begin{align*}
\Gamma_{\bar\gamma} & \;:=\; \ul{\wh\Gamma}\Bigl(I_{\ul\CG_0}\!(Q)\cap \bigl(\ul{\wh\Gamma}^{-1}\bigl(\prod_i J_{\ul{\BL}_i}(Q_{\nu_i})\bigr)\times \gamma H \gamma^{-1}\bigr)\Bigr) \\
& \;:=\; \ul{\wh\Gamma}\Bigl(I_{\ul\CG_0}\!(Q)\cap (\ul{\wh\Gamma},\epsilon)^{-1}\bigl( \prod_i J_{\ul{\BL}_i}(Q_{\nu_i})\times hHh^{-1}) \Bigr) \;\subset\;\prod_i J_{\ul{\BL}_i}(Q_{\nu_i})
\end{align*}
be the image under the injective homomorphism $\ul{\wh\Gamma}$ from \eqref{EqI(Q)}. Since the intersection $\FG(\BO^{\ul\nu})\cap hHh^{-1}$ has finite index in $\FG(\BO^{\ul\nu})$ and in $hHh^{-1}$, the intersection $I_{\ul\CG_0}\!(Q)\cap \bigl(U\times hHh^{-1}\bigr)\cong\Gamma_{\bar\gamma}\cap U$ is also finite, and hence the subgroup $\Gamma_{\bar\gamma}\subset\prod_i J_{\ul{\BL}_i}(Q_{\nu_i})$ is discrete. The group $\Gamma_{\bar\gamma}$ is separated, because for every element $1\ne\eta\in\Gamma_{\bar\gamma}$ there is a normal subgroup $\wt H\subset H$ of finite index such that $h\wt Hh^{-1}$ does not contain the element $\epsilon(\eta)\ne1$. Therefore $\Gamma_{\bar\gamma}\big{\backslash}\prod_i \breve\CM_{\ul\BL_i}^{\hat{Z}_i}$ and $I_{\ul\CG_0}\!(Q) \big{\backslash}\bigl(\prod_i \breve\CM_{\ul\BL_i}^{\hat{Z}_i}\times \Isom^{\otimes}(\omega^\circ,\check{\CV}_{\ul{\CG}_0})/H\bigr)$ are formal algebraic Deligne-Mumford stacks by \cite[\quotisfstack]{AH_Local}. That they are Deligne-Mumford and the last assertion about \eqref{EqReducedSourceOfTheta} follow from the proof of \cite[\quotisfstack]{AH_Local} where it is shown that \eqref{EqReducedSourceOfTheta} (respectively \eqref{EqSourceOfTheta}) are locally the stack quotient of a (formal) scheme by a finite group.
\end{proof}

\begin{point}\label{Point7.8}
There is a further group acting on the source and the target of the morphism $\Theta$ from \eqref{EqUnifIntro}, namely the group $\prod_iZ(Q_{\nu_i})$ where $Z\subset\FG\times_C\Spec Q$ is the center. Indeed, $Z(Q_{\nu_i})$ lies in the center of $\FG(Q_{\nu_i})=L\genericG_{\nu_i}(\BF_{\nu_i})$. Writing $\ul\BL_i\cong\bigl((L^+\BP_{\nu_i})_\BaseFldInSectUnif,b_i\hat\sigma^*_{\nu_i}\bigr)$ with $b_i\in L\genericG_{\nu_i}(\BaseFldInSectUnif)$, there is an inclusion
\begin{align}
\label{EqActionCenter1}
Z(Q_{\nu_i}) \enspace\longinto\enspace & J_{\ul{\BL}_i}(Q_{\nu_i})\;=\;\bigl\{\,j\in L\genericG_{\nu_i}(\BaseFldInSectUnif)\colon j\, b_i=b_i\,\hat\sigma_{\nu_i}(j)\,\bigr\}\;=\;\QIsog_\BaseFldInSectUnif(\ul\BL_i)\,, \\
c_i \quad\enspace\longmapsto\enspace & \qquad\qquad\qquad\; c_i=\hat\sigma_{\nu_i}(c_i) \nonumber
\end{align}
through which $c_i\in Z(Q_{\nu_i})$ acts on $\breve\CM_{\ul{\BL}_i}^{\hat{Z}_i}$ via $c_i\colon (\ul\CL_i,\delta_i)\mapsto (\ul\CL_i,c_i\delta_i)$. This action can also be described in a different way. For any local $\BP_{\nu_i}$-shtuka $\ul\CL_i$ over a scheme $S\in\Nilp_{A_{\nu_i}}$ we claim that $c_i$ induces an element $c_i\in\QIsog_S(\ul\CL_i)$ of the quasi-isogeny group of $\ul\CL_i$. Namely, over an \'etale covering $S'\to S$ we can choose a trivialization $\alpha\colon\ul\CL_{i,S'}\isoto\bigl((L^+\BP_{\nu_i})_{S'},b'_i\hat\sigma^*_{\nu_i}\bigr)$ and then $c_i=\hat\sigma_{\nu_i}(c_i)$ implies $c_i\,b'_i=b'_i\,\hat\sigma_{\nu_i}(c_i)$, that is $\alpha^{-1}\circ c_i\circ\alpha\in\QIsog_{S'}(\ul\CL_i)$. To show that this quasi-isogeny descends to $S$ let $pr_1,pr_2\colon S'':=S'\times_S S' \to S'$ be the projections onto the first and second factor and set $g:=pr_2^*\alpha\circ pr_1^*\alpha^{-1}\in L^+\BP_{\nu_i}(S'')$. Then $g \circ pr_1^*c_i\circ g^{-1}=pr_1^*c_i = pr_2^*c_i$ implies
\[
pr_1^*(\alpha^{-1}\circ c_i\circ\alpha)\;=\;pr_2^*\alpha^{-1}\circ g \circ pr_1^*c_i\circ g^{-1}\circ pr_2^*\alpha\;=\;pr_2^*(\alpha^{-1}\circ c_i\circ\alpha)\,,
\]
and this shows that $\alpha^{-1}\circ c_i\circ\alpha$ descends to a quasi-isogeny of $\ul\CL_i$ over $S$ which we denote by $c_i\in\QIsog_S(\ul\CL_i)$. If moreover we are given a quasi-isogeny $\delta_i\colon\ul\CL_i\to\ul\BL_{i,S}$, that is if $(\ul\CL_i,\delta_i)\in\CM_{\ul\BL_i}(S)$, then $\delta_i$ is automatically compatible with the quasi-isogenies $c_i\in\QIsog_S(\ul\CL_i)$ and $c_i\in\QIsog_\BaseFldInSectUnif(\ul\BL_i)$, that is $\delta_i\circ c_i=c_i\circ\delta_i$. Indeed, using the trivialization $\alpha$ over $S'$ from above, $\delta_i$ corresponds to $g_i:=\delta_i\circ\alpha^{-1}\in L\genericG_{\nu_i}(S')$, and then
\[
\delta_i\circ c_i\;:=\;\delta_i\circ(\alpha^{-1}\circ c_i\circ\alpha)\;=\;g_i\circ c_i\circ\alpha\;=\;c_i\circ g_i\circ\alpha\;=\;c_i\circ\delta_i\,.
\]
This shows that the action of $c_i\in Z(Q_{\nu_i})$ on $\breve\CM_{\ul{\BL}_i}^{\hat{Z}_i}$ is given as
\begin{equation}\label{EqActionCenter3}
c_i\colon\breve\CM_{\ul{\BL}_i}^{\hat{Z}_i}\;\longto\;\breve\CM_{\ul{\BL}_i}^{\hat{Z}_i}\,,\quad (\ul\CL_i,\delta_i)\;\longmapsto\; (\ul\CL_i,c_i\circ\delta_i)\;=\;(\ul\CL_i,\delta_i\circ c_i)\,. 
\end{equation}
This action commutes with the action of $\eta\in I_{\ul\CG_0}\!(Q)$, because $\eta$ and $(c_i)_i$ act on $\prod_i \breve\CM_{\ul\BL_i}^{\hat{Z}_i}\times \Isom^{\otimes}(\omega^\circ,\check{\CV}_{\ul{\CG}_0})/H$ by 
\[
(\ul{\CL}_i,\delta_i)_i \times \gamma H\;\longmapsto\;(\ul{\CL}_i,\wh\Gamma_{\nu_i}(\eta)\circ\delta_i\circ c_i)_i \times \check{\CV}_\eta\gamma H\,.
\]

On the other hand $(c_i)_i\in\prod_i Z(Q_{\nu_i})$ also acts on $\nabla_n^{H,\ulHZ}\scrH^1(C,\FG)^{\ul\nu}$ as follows. Let $(\ul\CG,\gamma H)$ be in $\nabla_n^{H,\ulHZ}\scrH^1(C,\FG)^{\ul\nu}(S)$ and consider the associated local $\BP_{\nu_i}$-shtukas $\ul\CL_i:=\wh\Gamma_{\nu_i}(\ul\CG)$. We have seen above that $c_i$ induces an element $c_i\in\QIsog_S(\ul\CL_i)$ of the quasi-isogeny group of $\ul\CL_i$. By Proposition~\ref{PropLocalIsogeny} there is a uniquely determined global $\FG$-shtuka $c_n^*\circ\ldots\circ c_1^*\,\ul\CG$ and a quasi-isogeny $c\colon c_n^*\circ\ldots\circ c_1^*\,\ul\CG\to\ul\CG$ which is an isomorphism outside the $\nu_i$ and satisfies $\wh\Gamma_{\nu_i}(c)=c_i$. We can now define the action of $(c_i)_i\in\prod_i Z(Q_{\nu_i})$ on $(\ul\CG,\gamma H)\in\nabla_n^{H,\ulHZ}\scrH^1(C,\FG)^{\ul\nu}(S)$ as
\begin{equation}\label{EqActionCenter2}
(c_i)_i\colon\;(\ul\CG,\gamma H)\;\longmapsto\;(c_n^*\circ\ldots\circ c_1^*\,\ul\CG\,,\,\check\CT_{c}^{-1}\gamma H)\,.
\end{equation}
\end{point}

\begin{point}\label{Point7.9}
Source and target of $\Theta$ carry a \emph{Weil-descent datum} for the ring extension $R_{\ulHZ}\subset\breve R_{\ulHZ}=\BaseFldInSectUnif\dbl\xi_1,\ldots,\xi_n\dbr$, compare~\cite[Definition~3.45]{RZ}. We explain what this means. Consider the $R_{\ulHZ}$-automorphism $\lambda$ of $\breve R_{\ulHZ}$ given by
\[
\lambda\colon \xi_i\;\longmapsto\;\xi_i \quad\text{and}\quad \lambda|_\BaseFldInSectUnif\;=\;\Frob_{\#\kappa,\BaseFldInSectUnif}\colon x\;\longmapsto\;x^{\#\kappa}\quad\text{for }x\in\BaseFldInSectUnif\,.
\]
For a scheme $(S,\theta)\in\Nilp_{\breve R_{\ulHZ}}$ where $\theta\colon S\to\Spf\breve R_{\ulHZ}$ denotes the structure morphism of the scheme $S$ we denote by $S_{[\lambda]}\in\Nilp_{\breve R_{\ulHZ}}$ the pair ($S,\lambda\circ\theta)$. For a stack $\CH$ over $\Spf\breve R_{\ulHZ}$ we consider the stack $\CH^\lambda$ defined by 
\[
\CH^\lambda(S)\;:=\;\CH(S_{[\lambda]})\,.
\]
Then a \emph{Weil-descent datum} on $\CH$ is an isomorphism of stacks $\CH\isoto\CH^\lambda$, that is an equivalence $\CH(S)\isoto\CH(S_{[\lambda]})$ for every $S\in\Nilp_{\breve R_{\ulHZ}}$ compatible with morphisms in $\Nilp_{\breve R_{\ulHZ}}$.

On $\nabla_n^{H,\ulHZ}\scrH^1(C,\FG)^{\ul\nu}\whtimes_{R_{\ulHZ}} \Spf\breve R_{\ulHZ}$ the canonical Weil-descent datum is given by the identity
\begin{equation}\label{EqDescentDatumOnNablaH}
\id\colon\nabla_n^{H,\ulHZ}\scrH^1(C,\FG)^{\ul\nu}(S)\;\isoto\;\nabla_n^{H,\ulHZ}\scrH^1(C,\FG)^{\ul\nu}(S_{[\lambda]})\,,\quad(\ul\CG,\gamma H)\;\longmapsto\;(\ul\CG,\gamma H)
\end{equation}
because under the inclusion $\Nilp_{\breve R_{\ulHZ}}\into\Nilp_{R_{\ulHZ}}$ we have $S=S_{[\lambda]}$ in $\Nilp_{R_{\ulHZ}}$.

On $\breve\CM_{\ul\BL_i}^{\hat{Z}_i}$ we consider the Weil descent datum given by
\begin{align}\label{EqDescentDatumSource}
\breve\CM_{\ul\BL_i}^{\hat{Z}_i}(S)\;\isoto\; & \breve\CM_{\ul\BL_i}^{\hat{Z}_i}(S_{[\lambda]}),\\
(\ul\CL_i,\delta_i\colon\ul\CL_i\to\ul\BL_{i,S})\;\longmapsto\; & (\ul\CL_i,\theta^*(\tau_{\BL_i}^{-[\kappa\colon\BF_q]})\circ\delta_i\colon\ul\CL_i\to\ul\BL_{i,S_{[\lambda]}})\,.\nonumber
\end{align}
Here we observe that $\BL_{i,S}:=\theta^*\ul\BL_i$ and $\ul\BL_{i,S_{[\lambda]}}:=(\lambda\circ\theta)^*\ul\BL_i=\theta^*\lambda^*\ul\BL_i=\theta^*\sigma^{[\kappa\colon\BF_q]*}\ul\BL_i$, and that $\tau_{\BL_i}^{-[\kappa\colon\BF_q]}\colon\ul\BL_i\to\sigma^{[\kappa\colon\BF_q]*}\ul\BL_i$ is a quasi-isogeny. 

We do not discuss the question whether these Weil descent data are effective. In the analogous situation for $p$-divisible groups this is true and proved by Rapoport and Zink in \cite[Theorem~3.49]{RZ}. Their argument uses a morphism $\BG_m\to\FG$, which might not exist in our setup.

Moreover, on $\prod_i \breve\CM_{\ul\BL_i}^{\hat{Z}_i}\times \Isom^{\otimes}(\omega^\circ,\check{\CV}_{\ul{\CG}_0})/H$ we consider the product of the Weil Descent data \eqref{EqDescentDatumSource} with the identity on $\Isom^{\otimes}(\omega^\circ,\check{\CV}_{\ul{\CG}_0})/H$. This Weil descent datum commutes with the action of $\eta\in I_{\ul\CG_0}\!(Q)$ by the following diagram
\[
\xymatrix @C+2pc @R=1pc {
(\ul\CL_i,\delta_i) \ar@{|->}[r]^-{\TS\eta} \ar@{|->}[dd]^(0.45){\TS\text{descent datum}} & (\ul\CL_i,\theta^*(\wh\Gamma_{\nu_i}(\eta))\circ\delta_i) \ar@{|->}[dd]^(0.45){\TS\text{descent datum}} \\ \\
(\ul\CL_i,\theta^*(\tau_{\BL_i}^{-[\kappa\colon\BF_q]})\circ\delta_i) \ar@{|->}[rd]^-{\TS\eta} & (\ul\CL_i,\theta^*(\tau_{\BL_i}^{-[\kappa\colon\BF_q]})\circ\theta^*(\wh\Gamma_{\nu_i}(\eta))\circ\delta_i) \ar@{=}[d]\\
& \;(\ul\CL_i,(\lambda\theta)^*(\wh\Gamma_{\nu_i}(\eta))\circ\theta^*(\tau_{\BL_i}^{-[\kappa\colon\BF_q]})\circ\delta_i)\,,
}
\]
as $\theta^*(\tau_{\BL_i}^{-[\kappa\colon\BF_q]}\circ\wh\Gamma_{\nu_i}(\eta))=\theta^*(\sigma^{[\kappa\colon\BF_q]*}(\wh\Gamma_{\nu_i}(\eta))\circ\tau_{\BL_i}^{-[\kappa\colon\BF_q]})=(\lambda\theta)^*(\wh\Gamma_{\nu_i}(\eta))\circ\theta^*(\tau_{\BL_i}^{-[\kappa\colon\BF_q]})$. This defines a Weil descent datum on $\CY:=I_{\ul\CG_0}\!(Q) \big{\backslash}\bigl(\prod_i \breve\CM_{\ul\BL_i}^{\hat{Z}_i}\times \Isom^{\otimes}(\omega^\circ,\check{\CV}_{\ul{\CG}_0})/H\bigr)$ by
\begin{align}\label{EqDescentDatumSourceModI}
\CY(S) \;\isoto\; & \CY^\lambda(S)\;=\;\CY(S_{[\lambda]}) \\
(\ul\CL_i,\delta_i)_i\times \gamma H\;\longmapsto\; & (\ul\CL_i,\theta^*(\tau_{\BL_i}^{-[\kappa\colon\BF_q]})\circ\delta_i)_i\times \gamma H\,. \nonumber
\end{align}
\end{point}

\begin{point}\label{Point7.10}
For every multiple $m\in\BN_0$ of $[\kappa_i:\BF_q]$ the special fiber $\breve\CM_{\ul\BL_i}^{\hat{Z}_i}\whtimes_{R_i}\Spec\BaseFldInSectUnif$ of $\breve\CM_{\ul\BL_i}^{\hat{Z}_i}$ carries a Frobenius endomorphism $\Phi_m$ defined as follows. Consider the absolute $q^m$-Frobenius $\sigma^m:=\Frob_{q^m,S}\colon S\to S$ on a $\BaseFldInSectUnif$-scheme $S$ and a pair $(\ul\CL_i,\delta_i)\in\breve\CM_{\ul\BL_i}^{\hat{Z}_i}(S)$. It induces the left horizontal morphisms in the diagram
\[
\xymatrix @C+1pc {
S \ar[d]_{\TS\Frob_{q^m,S}} \ar[rrrrr]^-{\TS(\ul\CL_i,\delta_i)} \ar[drrrrr]_(.35){\TS(\sigma^{m*}\ul\CL_i,\sigma^{m*}\delta_i)\qquad} & & & & &  \breve\CM_{\ul\BL_i}^{\hat{Z}_i}\whtimes_{R_i}\Spec\BaseFldInSectUnif \ar[d]_{\TS\Frob_{q^m}} \ar[r] & \Spec\BaseFldInSectUnif \ar[d]_{\TS\Frob_{q^m,\BaseFldInSectUnif}} \\
S \ar[rrrrr]_-{\TS(\ul\CL_i,\delta_i)} & & & & & \breve\CM_{\ul\BL_i}^{\hat{Z}_i}\whtimes_{R_i}\Spec\BaseFldInSectUnif \ar[r] & \Spec\BaseFldInSectUnif 
}
\]
Let $\theta\colon S\to\Spec\BaseFldInSectUnif$ be the structure morphism of $S$. Then the upper $S$ viewed as a scheme over the lower $\Spec\BaseFldInSectUnif$ has structure morphism $\Frob_{q^m,\BaseFldInSectUnif}\circ\theta$. So in terms of \ref{Point7.9} with $\lambda:=\lambda_m:=\Frob_{q^m,\BaseFldInSectUnif}$ the upper $S$ is $S_{[\lambda_m]}$ over the lower $\Spec\BaseFldInSectUnif$. Since $m$ is a multiple of $[\kappa_i\colon\BF_q]$ and the reduced subscheme $Z_i$ of the bound $\hat{Z}_i$ is defined over $\kappa_i$ the Frobenius $\tau_{\sigma^{m*}\CL_i}=\sigma^{m*}\tau_{\CL_i}$ lies in $\sigma^{m*}Z_i=Z_i$. This means that also $\sigma^{m*}\ul\CL_i$ is bounded by $\hat{Z}_i$, and therefore the diagonal arrow $(\sigma^{m*}\ul\CL_i,\sigma^{m*}\delta_i)$ lies in $\breve\CM_{\ul\BL_i}^{\hat{Z}_i}(S_{[\lambda_m]})=\bigl(\breve\CM_{\ul\BL_i}^{\hat{Z}_i}\bigr)^{\lambda_m}(S)$. Via the Weil descent datum \eqref{EqDescentDatumSource} it is mapped to $(\sigma^{m*}\ul\CL_i, \theta^*(\tau_{\BL_i}^m)\circ\sigma^{m*}\delta_i)$ which lies in $\breve\CM_{\ul\BL_i}^{\hat{Z}_i}(S)$. We therefore define the $q^m$-Frobenius endomorphism of $\breve\CM_{\ul\BL_i}^{\hat{Z}_i}\whtimes_{R_i}\Spec\BaseFldInSectUnif$ as
\begin{align}\label{EqFrobOnRZ}
\Phi_m\colon\;\breve\CM_{\ul\BL_i}^{\hat{Z}_i}\whtimes_{R_i}\Spec\BaseFldInSectUnif \;\longto\; & \bigl(\breve\CM_{\ul\BL_i}^{\hat{Z}_i}\whtimes_{R_i}\Spec\BaseFldInSectUnif\bigr) \\
(\ul\CL_i,\delta_i) \;\longmapsto\; & (\sigma^{m*}\ul\CL_i, \tau_{\BL_i}^m\circ\sigma^{m*}\delta_i)\;=\;(\sigma^{m*}\ul\CL_i, \delta_i\circ\tau_{\CL_i}^m)\,. \nonumber
\end{align}
If $m$ is a multiple of $[\kappa\colon\BF_q]$ the product of the $q^m$-Frobenius morphisms $\Phi_m$ of the $\breve\CM_{\ul\BL_i}^{\hat{Z}_i}\whtimes_{R_i}\Spec\BaseFldInSectUnif$ from \eqref{EqFrobOnRZ} with the identity on $\Isom^{\otimes}(\omega^\circ,\check{\CV}_{\ul{\CG}_0})/H$ gives a $q^m$-Frobenius morphism
\begin{align}\label{EqFrobOnSource}
\Phi_m\colon\;\bigl(\prod_i\breve\CM_{\ul\BL_i}^{\hat{Z}_i}\whtimes_{R_i}\Spec\BaseFldInSectUnif\bigr) & \times \Isom^{\otimes}(\omega^\circ,\check{\CV}_{\ul{\CG}_0})/H\;\longto\; \\
& \longto\; \bigl(\prod_i\breve\CM_{\ul\BL_i}^{\hat{Z}_i}\whtimes_{R_i}\Spec\BaseFldInSectUnif\bigr)\times \Isom^{\otimes}(\omega^\circ,\check{\CV}_{\ul{\CG}_0})/H\,.\nonumber
\end{align}
It directly follows that $\Phi_m$ commutes with the action of $\eta\in I_{\ul\CG_0}\!(Q)$, because the composition is given by 
\[
(\ul{\CL}_i,\delta_i)_i \times \gamma H\;\longmapsto\;(\ul{\CL}_i,\wh\Gamma_{\nu_i}(\eta)\circ\delta_i\circ\tau_{\CL_i}^m)_i \times \check{\CV}_\eta\gamma H\,.
\]
This defines the $q^m$-Frobenius endomorphism $\Phi_m$ of 
$I_{\ul\CG_0}\!(Q) \big{\backslash}\bigl(\prod_i \breve\CM_{\ul\BL_i}^{\hat{Z}_i}\times \Isom^{\otimes}(\omega^\circ,\check{\CV}_{\ul{\CG}_0})/H\bigr)$
.

Likewise for every multiple $m\in\BN_0$ of $[\kappa\colon\BF_q]$ the stack $\nabla_n^{H,\ulHZ}\scrH^1(C,\FG)^{\ul\nu}\whtimes_{R_{\ulHZ}} \Spec\BaseFldInSectUnif$ carries the relative $q^m$-Frobenius, which is an endomorphism, because this stack arises by base change from $\Spec\kappa$. This $q^m$-Frobenius endomorphism is given 
\begin{align}\label{EqFrobOnTarget}
\Phi_m\colon \nabla_n^{H,\ulHZ}\scrH^1(C,\FG)^{\ul\nu}\whtimes_{R_{\ulHZ}} \Spec\BaseFldInSectUnif\;\longto\; & \nabla_n^{H,\ulHZ}\scrH^1(C,\FG)^{\ul\nu}\whtimes_{R_{\ulHZ}} \Spec\BaseFldInSectUnif \nonumber \\
(\ul\CG,\gamma H) \;\longmapsto\; & (\sigma^{m*}\ul\CG,\sigma^{m*}(\lambda)H)\,.
\end{align}
Since $m$ is a multiple of $[\kappa_i\colon\BF_q]$ and the reduced subscheme $Z_i$ of the bound $\hat{Z}_i$ is defined over $\kappa_i$ the Frobenius $\wh\Gamma_{\nu_i}(\tau_{\sigma^{m*}\CG})=\sigma^{m*}\wh\Gamma_{\nu_i}(\tau_{\CG})$ lies in $\sigma^{m*}Z_i=Z_i$. This means that also $\sigma^{m*}\ul\CG$ is bounded by $\ulHZ$.
\end{point}

\forget{
\begin{point}\label{Point7.10}
For every multiple $m\in\BN_0$ of $[\kappa_i:\BF_q]$ the special fiber $\breve\CM_{\ul\BL_i}^{\hat{Z}_i}\whtimes_{R_i}\Spec\BaseFldInSectUnif$ of $\breve\CM_{\ul\BL_i}^{\hat{Z}_i}$ carries a Frobenius endomorphism $\Phi_m$ given by
\begin{align}\label{EqFrobOnRZ}
\Phi_m\colon\;\breve\CM_{\ul\BL_i}^{\hat{Z}_i}\whtimes_{R_i}\Spec\BaseFldInSectUnif \;\longto\; & \sigma^{m*}\bigl(\breve\CM_{\ul\BL_i}^{\hat{Z}_i}\whtimes_{R_i}\Spec\BaseFldInSectUnif\bigr)\;:=\;\bigl(\breve\CM_{\ul\BL_i}^{\hat{Z}_i}\whtimes_{R_i}\Spec\BaseFldInSectUnif\bigr)\underset{\BaseFldInSectUnif,\sigma^m}{\times}\Spec\BaseFldInSectUnif \nonumber \\
(\ul\CL_i,\delta_i) \;\longmapsto\; & (\sigma^{m*}\ul\CL_i, \tau_{\BL_i}^m\circ\sigma^{m*}\delta_i)\;=\;(\sigma^{m*}\ul\CL_i, \delta_i\circ\tau_{\CL_i}^m)\,.
\end{align}
If $m$ is a multiple of $[\kappa\colon\BF_q]$ the product of the $q^m$-Frobenius morphisms $\Phi_m$ of $\breve\CM_{\ul\BL_i}^{\hat{Z}_i}\whtimes_{R_i}\Spec\BaseFldInSectUnif$ from \eqref{EqFrobOnRZ} with the identity on $\Isom^{\otimes}(\omega^\circ,\check{\CV}_{\ul{\CG}_0})/H$ gives a $q^m$-Frobenius morphism
\begin{align}\label{EqFrobOnSource}
\Phi_m\colon\;\bigl(\prod_i\breve\CM_{\ul\BL_i}^{\hat{Z}_i}\whtimes_{R_i}\Spec\BaseFldInSectUnif\bigr) & \times \Isom^{\otimes}(\omega^\circ,\check{\CV}_{\ul{\CG}_0})/H\;\longto\; \\
& \longto\; \bigl(\prod_i\breve\CM_{\ul\BL_i}^{\hat{Z}_i}\whtimes_{R_i}\Spec\BaseFldInSectUnif\bigr)\times \Isom^{\otimes}(\omega^\circ,\check{\CV}_{\ul{\CG}_0})/H\,.\nonumber
\end{align}
\end{point}
}

\begin{theorem}\label{Uniformization2} 
Keep the above notation and consider a compact open subgroup $H\subset \FG (\BA^{\ul\nu})$.
\begin{enumerate}
\item \label{Uniformization2_A} 
The morphism $\Psi_{\ul\CG_0}$ from Theorem~\ref{Uniformization1} induces an $I_{\ul\CG_0}\!(Q)$-invariant morphism
\begin{equation}\label{EqUnifMorphPrime}
\Theta'\colon \prod_i \breve\CM_{\ul\BL_i}^{\hat{Z}_i}\times \Isom^{\otimes}(\omega^\circ,\check{\CV}_{\ul{\CG}_0})/H\;\longto\; \nabla_n^{H,\ulHZ}\scrH^1(C,\FG)^{\ul\nu}\whtimes_{R_{\ulHZ}} \Spf\breve R_{\ulHZ}\,,
\end{equation}
where $I_{\ul\CG_0}\!(Q)$ acts trivially on the target and diagonally on the source as described in \ref{Point7.6}. Furthermore, this morphism factors through a morphism 
\begin{equation}\label{EqUnifMorph}
\Theta\colon  I_{\ul\CG_0}\!(Q) \big{\backslash}\bigl(\prod_i \breve\CM_{\ul\BL_i}^{\hat{Z}_i}\times \Isom^{\otimes}(\omega^\circ,\check{\CV}_{\ul{\CG}_0})/H\bigr) \;\longto\; \nabla_n^{H,\ulHZ}\scrH^1(C,\FG)^{\ul\nu}\whtimes_{R_{\ulHZ}} \Spf\breve R_{\ulHZ}
\end{equation}
of ind-DM-stacks over $\Spf\breve R_{\ulHZ}$, which is a monomorphism in the sense that the functor $\Theta$ is fully faithful, or equivalently that its diagonal is an isomorphism. Both morphisms are ind-proper and formally \'etale.

\item \label{Uniformization2_B}
Let $\{T_j\}$ be a set of representatives of $I_{\ul\CG_0}\!(Q)$-orbits of the irreducible components of the scheme $\prod_i\breve X_{Z_i}(\ul\BL_i)\times \Isom^{\otimes}(\omega^\circ,\check{\CV}_{\ul{\CG}_0})/H$ which is locally of finite type over $\BaseFldInSectUnif$. Then the image $\Theta'(T_j)$ of $T_j$ under $\Theta'$ is a closed substack with the reduced structure and each $\Theta'(T_j)$ intersects only finitely many others. Let $\CZ$ be the union of the $\Theta'(T_j)$. Its underlying set is the isogeny class of $\ul\CG_0$, that is the set of all $(\ul\CG,\gamma H)$ for which $\ul\CG$ is isogenous to $\ul\CG_0$. Let $\nabla_n^{H,\ulHZ}\scrH^1(C,\FG)^{\ul\nu}_{/\CZ}$ be the formal completion of $\nabla_n^{H,\ulHZ}\scrH^1(C,\FG)^{\ul\nu}\whtimes_{R_{\ulHZ}} \Spf\breve R_{\ulHZ}$ along $\CZ$; see Remark~\ref{RemHeckeCorr}(a). Then $\Theta$ induces an isomorphism of locally noetherian, adic formal algebraic Deligne-Mumford stacks locally formally of finite type over $\Spf\breve R_{\ulHZ}$
$$
\Theta_{\CZ}\colon  I_{\ul\CG_0}\!(Q) \big{\backslash}\bigl(\prod_i \breve\CM_{\ul\BL_i}^{\hat{Z}_i}\times \Isom^{\otimes}(\omega^\circ,\check{\CV}_{\ul{\CG}_0})/H\bigr)\;\isoto\; \nabla_n^{H,\ulHZ}\scrH^1(C,\FG)^{\ul\nu}_{/\CZ}\,,
$$
and in particular of the underlying Deligne-Mumford stacks
$$
\Theta_{\CZ}\colon  I_{\ul\CG_0}\!(Q) \big{\backslash}\bigl(\prod_i \breve X_{Z_i}(\ul\BL_i)\times \Isom^{\otimes}(\omega^\circ,\check{\CV}_{\ul{\CG}_0})/H\bigr)\;\isoto\; \CZ
$$
which are locally of finite type and separated over $\Spec\BaseFldInSectUnif$.

\item \label{Uniformization2_C}
The morphisms $\Theta'$, $\Theta$ and $\Theta_\CZ$ are compatible with the following actions on source and target: the action of $\prod_i Z(Q_{\nu_i})$ described in \ref{Point7.8}, the action of $\FG(\BA^{\ul\nu})$ through Hecke-corres\-ponden\-ces, see Remark~\ref{RemHeckeCorr}(b) below, and the Weil descent data described in \ref{Point7.9}. For every multiple $m\in\BN_0$ of $[\kappa\colon\BF_q]$ the base changes of $\Theta'$, $\Theta$ and $\Theta_\CZ$ to $\Spec\BaseFldInSectUnif$ are compatible with the $q^m$-Frobenius endomorphisms $\Phi_m$ from \eqref{EqFrobOnSource} and \eqref{EqFrobOnTarget}. 
\end{enumerate}
\end{theorem}

As an explanation of the theorem and a preparation for its proof we begin with a

\begin{remark}\label{RemHeckeCorr}
(a) Notice that the $T_j$ correspond bijectively to the irreducible components of the Deligne-Mumford stack \eqref{EqReducedSourceOfTheta} which is locally of finite type over $\Spec\BaseFldInSectUnif$. Since $\Theta$ is a monomorphism by part~\ref{Uniformization2_A} of the theorem, each $\Theta'(T_j)$ intersects only finitely many others. The restriction of $\Theta'$ to $T_j$ is proper, because $T_j$ is quasi-compact by \cite[\CorQC]{AH_Local} and $\Theta'$ is ind-proper by part~\ref{Uniformization2_A} of the theorem. So the $\Theta'(T_j)$ are closed substacks. Reasoning as in \cite[6.22]{RZ} we may form the formal completion along their union $\CZ$. It is defined by requiring that its category $\nabla_n^{H,\ulHZ}\scrH^1(C,\FG)^{\ul\nu}_{/\CZ}(S)$ of $S$-valued points is the full subcategory given by
\[
\bigl\{\,f\colon S\to\nabla_n^{H,\ulHZ}\scrH^1(C,\FG)^{\ul\nu}\whtimes_{R_{\ulHZ}} \Spf\breve R_{\ulHZ}, \enspace\text{such that } f|_{S_\red}\text{ factors through }\CZ\,\bigr\}\,,
\]
where $S_\red$ is the underlying reduced closed subscheme. A priory this formal completion is only an ind-DM-stack over $\Spf\breve R_{\ulHZ}$, but it will follow from Theorem~\ref{Uniformization2}\ref{Uniformization2_B} that it is a locally noetherian, adic formal algebraic Deligne-Mumford stack locally formally of finite type over over $\Spf\breve R_{\ulHZ}$. Note that it follows immediately that the natural morphism
\[
\nabla_n^{H,\ulHZ}\scrH^1(C,\FG)^{\ul\nu}_{/\CZ}\;\longto\;\nabla_n^{H,\ulHZ}\scrH^1(C,\FG)^{\ul\nu}\whtimes_{R_{\ulHZ}} \Spf\breve R_{\ulHZ}
\]
is a monomorphism and formally \'etale, because for an affine scheme $S=\Spec B\in\Nilp_{\breve R_{\ulHZ}}$ and an ideal $I\subset B$ with $I^2=(0)$ one has $S_\red=(\Spec B/I)_\red$.

\medskip\noindent
(b) The action of $h\in\FG(\BA^{\ul\nu})$ by Hecke correspondences is explicitly given as follows. Let $H,H'\subset\FG(\BA^{\ul\nu})$ be compact open subgroups. Then the Hecke correspondences $\pi(h)_{H'\!,H}$ are given by the diagrams
\begin{equation}\label{EqHeckeSource}
\xymatrix @C=-4pc {
& \prod_i \breve\CM_{\ul\BL_i}^{\hat{Z}_i}\times \Isom^{\otimes}(\omega^\circ,\check{\CV}_{\ul{\CG}_0})/(hHh^{-1}\cap H') \ar[dl] \ar[dr] \\
\prod_i \breve\CM_{\ul\BL_i}^{\hat{Z}_i}\times \Isom^{\otimes}(\omega^\circ,\check{\CV}_{\ul{\CG}_0})/H& & \prod_i \breve\CM_{\ul\BL_i}^{\hat{Z}_i}\times \Isom^{\otimes}(\omega^\circ,\check{\CV}_{\ul{\CG}_0})/H'\ar@{-->}[ll]\\
& (\ul{\CL}_i,\delta_i)_i \times \gamma(hHh^{-1}\cap H')\ar@{|->}[dl] \ar@{|->}[dr]\\
(\ul{\CL}_i,\delta_i)_i \times \gamma hH & & (\ul{\CL}_i,\delta_i)_i \times \gamma H'
}
\end{equation}
and
\begin{equation}\label{EqHeckeTarget}
\xymatrix @C=0pc {
& \nabla_n^{(hHh^{-1}\cap H'),\ulHZ}\scrH^1(C,\FG)^{\ul\nu} \ar[dl] \ar[dr] \\
\nabla_n^{H,\ulHZ}\scrH^1(C,\FG)^{\ul\nu} & & \nabla_n^{H'\!,\ulHZ}\scrH^1(C,\FG)^{\ul\nu} \ar@{-->}[ll] \\
& (\ul\CG,\gamma(hHh^{-1}\cap H')) \ar@{|->}[dl] \ar@{|->}[dr]\\
(\ul\CG,\gamma hH) & & (\ul\CG,\gamma H')
}
\end{equation}
A special case for $H'\subset H$ and $h=1$ are the forgetful morphisms 
\[
\pi(1)_{H'\!,H}\colon\prod_i \breve\CM_{\ul\BL_i}^{\hat{Z}_i}\times \Isom^{\otimes}(\omega^\circ,\check{\CV}_{\ul{\CG}_0})/H'\to\prod_i \breve\CM_{\ul\BL_i}^{\hat{Z}_i}\times \Isom^{\otimes}(\omega^\circ,\check{\CV}_{\ul{\CG}_0})/H
\]
and $\pi(1)_{H'\!,H}\colon\nabla_n^{H'\!,\ulHZ}\scrH^1(C,\FG)^{\ul\nu}\to\nabla_n^{H,\ulHZ}\scrH^1(C,\FG)^{\ul\nu}$, which are finite \'etale and surjective; see Theorem~\ref{ThmLSGGsht1}\ref{ThmLSGGsht1_B}.
\end{remark}

\smallskip

\begin{proof}[Proof of Theorem~\ref{Uniformization2}\ref{Uniformization2_A} and \ref{Uniformization2_C}.]
For the empty closed subscheme $D=\emptyset\subset C$ recall that  $H_\emptyset=\FG(\BO^{\ul\nu})$ and $\nabla_n^{H_\emptyset}\scrH^1(C,\FG)^{\ul\nu}\cong\nabla_n\scrH_\emptyset^1(C,\FG)^{\ul\nu}=\nabla_n\scrH^1(C,\FG)^{\ul\nu}$ by Theorem~\ref{H_DL-Str}. Consider the following diagram of ind-DM-stacks in which the map $\Psi_{\ul\CG_0}$ in the bottom row was introduced in Theorem~\ref{Uniformization1} 
\begin{equation}\label{Eq_Psi}
\xymatrix @C+4pc @R=0.2pc {
\prod_i \breve\CM_{\ul\BL_i}^{\hat{Z}_i}\ar[r]^{\TS\Psi_{\ul\CG_0}\qquad\qquad} \ar[dddd] & \nabla_n^{H_\emptyset,\ulHZ}\scrH^1(C,\FG)^{\ul\nu}\whtimes_{R_{\ulHZ}} \Spf\breve R_{\ulHZ} \ar[dddd]\\ \\ \\ \\
\prod_i \breve\CM_{\ul\BL_i}\ar[r]^{\TS\Psi_{\ul\CG_0}\qquad\qquad} & \nabla_n\scrH^1(C,\FG)^{\ul\nu}\whtimes_{R_{\ulHZ}} \Spf\breve R_{\ulHZ}\\
\bigl(\ul\CL_i,\delta_i\colon\ul\CL_i\to\ul\BL_i\bigr)_i \ar@{|->}[r] & \delta_n^*\circ\ldots\circ\delta_1^*\,\ul\CG_0\;.
}
\end{equation}
Since $\delta_n^*\circ\ldots\circ\delta_1^*\,\ul\CG_0$ is bounded by $\ul{\hat Z}$ if and only if every $\ul\CL_i\cong\wh\Gamma_{\nu_i}(\delta_n^*\circ\ldots\circ\delta_1^*\,\ul\CG_0)$ is bounded by $\hat{Z}_i$, the diagram \eqref{Eq_Psi} is 2-cartesian. In particular the morphism in the upper row, which we again call $\Psi_{\ul\CG_0}$, is ind-proper and formally \'etale. Consider an $S$-valued point $(\ul{\CL}_i,\delta_i)_i$ of $\prod_i \breve\CM_{\ul\BL_i}^{\hat{Z}_i}$ and let $\ul\CG:=\delta_n^*\circ\ldots\circ\delta_1^*\,\ul\CG_{0,S}$ denote its image under $\Psi_{\ul\CG_0}$. By Proposition~\ref{PropLocalIsogeny} there is a unique quasi-isogeny $\delta\colon  \ul\CG \to \ul\CG_{0,S}$ which is an isomorphism outside the $\nu_i$ and satisfies $\wh{\ul\Gamma}(\delta)=(\delta_i)_i$. This induces an isomorphism $\check{\CT}_\delta\colon  \check{\CT}_{\ul\CG} \isoto \check{\CT}_{\ul\CG_0}$ of tensor functors, see \eqref{tatefunctor}. Since $\delta$ is defined over $S$, this isomorphism $\check{\CT}_\delta$ is equivariant for the action of $\pi_1^\et(S,\bar s)$ which acts on $\check{\CT}_{\ul\CG_0}$ through the map $\pi_1^\et(S,\bar s)\to \pi_1^\et(\Spec\BaseFldInSectUnif,\bar s)=(1)$, that is trivially. In particular, the $H$-orbit $\check{\CT}_\delta^{-1}\gamma H$ of the tensor isomorphism $\check{\CT}_\delta^{-1}\gamma\colon\omega^\circ \isoto\check{\CV}_{\ul\CG}$ is invariant under $\pi_1^\et(S,\bar s)$. Now sending $(\ul{\CL}_i,\delta_i)_i \times \gamma H$ to $(\ul\CG,\check{\CT}_\delta^{-1}\gamma H)$ defines the morphism
\begin{alignat}{3}
\Theta'\colon  \prod_i \breve\CM_{\ul\BL_i}^{\hat{Z}_i} & \times \Isom^{\otimes}(\omega^\circ,\check{\CV}_{\ul{\CG}_0}) {\slash}H \enspace & \longto & \enspace \nabla_n^{H,\ulHZ}\scrH^1(C,\FG)^{\ul\nu}\whtimes_{R_{\ulHZ}} \Spf\breve R_{\ulHZ} \\
(\ul{\CL}_i,\delta_i)_i\; & \times \quad \gamma H \enspace & \longmapsto & \enspace \quad (\ul\CG,\check{\CT}_\delta^{-1}\gamma H)\;. \nonumber
\end{alignat}
It sends the $\BaseFldInSectUnif$-valued point $(\ul\BL_i,\id)_i\times \gamma H$ to $(\ul\CG_0,\gamma H)$ and is obviously equivariant for the action of $\prod_i Z(Q_{\nu_i})$ given in \eqref{EqActionCenter3} and \eqref{EqActionCenter2}, and the action of $\FG(\BA^{\ul\nu})$ through Hecke correspondences given in \eqref{EqHeckeSource} and \eqref{EqHeckeTarget}. 

The morphism $\Theta'$ is compatible with the Weil descent data \eqref{EqDescentDatumOnNablaH} and \eqref{EqDescentDatumSource}, because for $(S,\theta)\in\Nilp_{\breve R_{\ulHZ}}$ and $S_{[\lambda]}=(S,\lambda\theta)\in\Nilp_{\breve R_{\ulHZ}}$ the $S$-valued point $(\ul\CL_i,\delta_i)_i\times \gamma H$, respectively the $S_{[\lambda]}$-valued point $(\ul\CL_i,\theta^*(\tau_{\BL_i}^{-[\kappa\colon\BF_q]})\circ\delta_i)_i\times\gamma H$ of the source of $\Theta'$ are sent to $(\ul\CG,\check\CT_\delta^{-1}\theta^*(\gamma) H)$ and $(\ul\CG',\check\CT_{\delta'}^{-1}(\lambda\theta)^*(\gamma) H)$, respectively, where $\ul\CG:=\delta_n^*\circ\ldots\circ\delta_1^*\,\ul\CG_{0,S}$ and $\ul\CG':=(\theta^*\tau_{\BL_n}^{-[\kappa\colon\BF_q]}\delta_n)^*\circ\ldots\circ(\theta^*\tau_{\BL_1}^{-[\kappa\colon\BF_q]}\delta_1)^*\,\ul\CG_{0,S_{[\lambda]}}$, and $\delta\colon\ul\CG\to\ul\CG_{0,S}:=\theta^*\ul\CG_0$ and $\delta'\colon\ul\CG'\to\ul\CG_{0,S_{[\lambda]}}:=(\lambda\circ\theta)^*\ul\CG_0=\theta^*\lambda^*\ul\CG_0=\theta^*(\sigma^{[\kappa\colon\BF_q]*}\ul\CG_0)$ are the quasi-isogenies with $\wh\Gamma_{\nu_i}(\delta)=\delta_i$ and $\wh\Gamma_{\nu_i}(\delta')=\theta^*(\tau_{\BL_i}^{-[\kappa\colon\BF_q]})\circ\delta_i$, which are isomorphisms outside $\ul\nu$. Then $\phi:=\delta^{-1}\circ\theta^*(\tau_{\CG_0}^{[\kappa\colon\BF_q]})\circ\delta'\colon\ul\CG'\to\ul\CG$ is a quasi-isogeny by Corollary~\ref{CorToLSisnonempty} with $\wh\Gamma_{\nu_i}(\phi)=\id_{\ul\CL_i}$ and $\check\CV_\phi\circ\check\CT_{\delta'}^{-1}\circ\theta^*\lambda^*(\gamma) H=\check\CT_\delta^{-1}\theta^*(\gamma) H$, because $\lambda^*(\gamma)=\sigma^{[\kappa\colon\BF_q]*}(\gamma)$ and $\check\CV_{\tau_{\CG_0}^{[\kappa\colon\BF_q]}}\circ\sigma^{[\kappa\colon\BF_q]*}(\gamma)=\gamma$. Since the Weil descend datum on the source of $\Theta'$ commutes with the action of $I_{\ul\CG_0}\!(Q)$ this also proves the compatibility of $\Theta$ with the Weil descend data \eqref{EqDescentDatumOnNablaH} and \eqref{EqDescentDatumSourceModI}. Finally, the target $\nabla_n^{H,\ulHZ}\scrH^1(C,\FG)^{\ul\nu}_{/\CZ}$ carries the Weil descend datum \eqref{EqDescentDatumOnNablaH} induced from $\nabla_n^{H,\ulHZ}\scrH^1(C,\FG)^{\ul\nu}\whtimes_{R_{\ulHZ}} \Spf\breve R_{\ulHZ}$. The reason is that if a morphism $S_\red\to\nabla_n^{H,\ulHZ}\scrH^1(C,\FG)^{\ul\nu}\whtimes_{R_{\ulHZ}} \Spf\breve R_{\ulHZ}$ given by $(\ul\CG,\gamma H)$ factors through $\CZ=\im(\Theta')$, then also the morphism $(S_{[\lambda]})_\red\to\nabla_n^{H,\ulHZ}\scrH^1(C,\FG)^{\ul\nu}\whtimes_{R_{\ulHZ}} \Spf\breve R_{\ulHZ}$ given by $(\ul\CG,\gamma H)$ factors through $\CZ=\im(\Theta')$, because $\Theta'$ commutes with the Weil descend data. This shows that also $\Theta_\CZ$ is compatible with the Weil descend data \eqref{EqDescentDatumOnNablaH} and \eqref{EqDescentDatumSourceModI}.

We also prove that $\Theta'$ commutes with the $q^m$-Frobenius endomorphisms $\Phi_m$. Let $\ul y:=(\ul\CL_i,\delta_i)_i\times \gamma H$ be an $S$-valued point of $\bigl(\prod_i\breve\CM_{\ul\BL_i}^{\hat{Z}_i}\whtimes_{R_i}\Spec\BaseFldInSectUnif\bigr) \times \Isom^{\otimes}(\omega^\circ,\check{\CV}_{\ul{\CG}_0})/H$. The images of this point and of $\Phi_m(\ul y)=(\sigma^{m*}\ul\CL_i, \tau_{\BL_i}^m\circ\sigma^{m*}\delta_i)_i\times\gamma H$ in $\nabla_n^{H,\ulHZ}\scrH^1(C,\FG)^{\ul\nu}$ are given by $\Theta'(\ul y)=(\ul\CG,\check\CT_\delta^{-1}\gamma H)$ and $\Theta'\circ\Phi_m(\ul y)=(\ul\CG',\check\CT_{\delta'}^{-1}\gamma H)$, respectively, where $\ul\CG:=\delta_n^*\circ\ldots\circ\delta_1^*\,\ul\CG_{0,S}$ and $\ul\CG':=(\tau_{\BL_n}^m\sigma^{m*}\delta_n)^*\circ\ldots\circ(\tau_{\BL_1}^m\sigma^{m*}\delta_1)^*\,\ul\CG_{0,S}$, and $\delta\colon\ul\CG\to\ul\CG_{0,S}$ and $\delta'\colon\ul\CG'\to\ul\CG_{0,S}$ are the quasi-isogenies with $\wh\Gamma_{\nu_i}(\delta)=\delta_i$ and $\wh\Gamma_{\nu_i}(\delta')=\tau_{\BL_i}^m\circ\sigma^{m*}\delta_i$, which are isomorphisms outside $\ul\nu$. We obtain for the image $\Phi_m\circ\Theta'(\ul y)=(\sigma^{m*}\ul\CG,\sigma^{m*}(\check\CT_\delta^{-1}\gamma) H)$, which comes with the quasi-isogeny $\sigma^{m*}(\delta)\colon\sigma^{m*}\ul\CG\to\sigma^{m*}\ul\CG_{0,S}$. Then $\phi:=\sigma^{m*}(\delta)^{-1}\circ\tau_{\CG_0}^{-m}\circ\delta'\colon\ul\CG'\to\sigma^{m*}\ul\CG$ is a quasi-isogeny by Corollary~\ref{CorToLSisnonempty} with $\wh\Gamma_{\nu_i}(\phi)=\id_{\sigma^{m*}\ul\CL_i}$ and $\check\CV_\phi\circ\check\CT_{\delta'}^{-1}\gamma H=\check\CT_{\sigma^{m*}\delta}^{-1}\circ\check\CV_{\tau_{\CG_0}^m}^{-1}\circ\gamma H=\sigma^{m*}(\check\CT_\delta^{-1}\gamma) H$, because $\check\CV_{\tau_{\CG_0}^{m}}^{-1}\circ\gamma=\sigma^{m*}(\gamma)$. This proves $\Theta'\circ\Phi_m=\Phi_m\circ\Theta'$. Since $\Phi_m$ commutes with the action of $I_{\ul\CG_0}\!(Q)$ this also proves that $\Theta$ commutes with the $\Phi_m$. Finally, the target $\nabla_n^{H,\ulHZ}\scrH^1(C,\FG)^{\ul\nu}_{/\CZ}$ carries the $q^m$-Frobenius endomorphism $\Phi_m$ induced from \eqref{EqFrobOnTarget} on $\nabla_n^{H,\ulHZ}\scrH^1(C,\FG)^{\ul\nu}\whtimes_{R_{\ulHZ}} \Spf\breve R_{\ulHZ}$. The reason is that if a morphism $S_\red\to\nabla_n^{H,\ulHZ}\scrH^1(C,\FG)^{\ul\nu}\whtimes_{R_{\ulHZ}} \Spf\breve R_{\ulHZ}$ given by $(\ul\CG,\gamma H)$ factors through $\CZ=\im(\Theta')$, then also the morphism $S_\red\to\nabla_n^{H,\ulHZ}\scrH^1(C,\FG)^{\ul\nu}\whtimes_{R_{\ulHZ}} \Spf\breve R_{\ulHZ}$ given by $(\sigma^{m*}\ul\CG,\sigma^{m*}(\lambda)H)$ factors through $\CZ=\im(\Theta')$, because $\Theta'$ commutes with the $\Phi_m$. This shows that also $\Theta_\CZ$ is compatible with the $q^m$-Frobenius endomorphisms $\Phi_m$ \eqref{EqFrobOnTarget} and \eqref{EqFrobOnSource}. So we have already proved \ref{Uniformization2_C}.

The group $I_{\ul\CG_0}\!(Q)$ acts on the source of the morphism $\Theta'$ by sending an $S$-valued point $(\ul{\CL}_i,\delta_i)_i \times \gamma H$ to $(\ul{\CL}_i,\wh\Gamma_{\nu_i}(\eta)\delta_i)_i \times \check{\CV}_\eta\gamma H$ for $\eta\in I_{\ul\CG_0}\!(Q)$. These two $S$-valued points are mapped under $\Theta'$ to global $\FG$-shtukas with $H$-level structure $(\ul\CG,\check{\CT}_\delta^{-1}\gamma H)$ and $(\ul{\wt\CG},\check{\CT}_{\tilde\delta}^{-1}\check{\CV}_\eta\gamma H)$ over $S$, where $\delta\colon\ul\CG:=\delta_n^*\circ\ldots\circ\delta_1^*\,\ul\CG_0\to\ul\CG_0$ is the isogeny satisfying $\wh{\Gamma}_{\nu_i}(\delta)=\delta_i$ and $\tilde\delta\colon\ul{\wt\CG}:=(\wh\Gamma_{\nu_n}(\eta)\delta_n)^*\circ\ldots\circ(\wh\Gamma_{\nu_1}(\eta)\delta_1)^*\,\ul\CG_0\to\ul\CG_0$ is the isogeny satisfying $\wh{\Gamma}_{\nu_i}(\tilde\delta)=\wh\Gamma_{\nu_i}(\eta)\delta_i$. Since $\check{\CV}_{\tilde\delta^{-1}\eta\delta}\circ\check{\CT}_\delta^{-1}\gamma H=\check{\CT}_{\tilde\delta}^{-1}\check{\CV}_\eta\gamma H$ these two global $\FG$-shtukas with $H$-level structure are isomorphic via the quasi-isogeny $\tilde\delta^{-1}\eta\delta\colon\ul\CG\to\ul{\wt\CG}$, which is an isomorphism at the $\nu_i$, because $\wh{\Gamma}_{\nu_i}(\tilde\delta^{-1}\eta\delta)=(\wh\Gamma_{\nu_i}(\eta)\delta_i)^{-1}\circ\wh\Gamma_{\nu_i}(\eta)\delta_i=\id$. In other words, $\Theta'$ is invariant under the action of $I_{\ul\CG_0}\!(Q)$ and factors through the morphism $\Theta$ from \eqref{EqUnifMorph} of ind-DM-stacks. Then Theorem~\ref{Uniformization2}\ref{Uniformization2_A} follows from Lemmas~\ref{LemmaThetaismono} and \ref{LemmaFormEtale} below. 
\end{proof}

\begin{lemma} \label{LemmaThetaismono}
We use the abbreviations $Y_1:=\prod_i \breve\CM_{\ul\BL_i}^{\hat{Z}_i}\times \Isom^{\otimes}(\omega^\circ,\check{\CV}_{\ul{\CG}_0})/H$ and $Y_2:=\prod_i \breve X_{Z_i}(\ul\BL_i)\times \Isom^{\otimes}(\omega^\circ,\check{\CV}_{\ul{\CG}_0})/H$. Then for $j=1$ or $j=2$ the action of $I_{\ul\CG_0}\!(Q)$ on $Y_j$ induces an isomorphism of stacks 
\[
I_{\ul\CG_0}\!(Q) \times Y_j \enspace := \enspace \coprod_{I_{\ul\CG_0}\!(Q)} Y_j \enspace \isoto \enspace Y_j \underset{\nabla_n^{H,\ulHZ}\scrH^1(C,\FG)^{\ul\nu}\whtimes_{R_{\ulHZ}} \Spf\breve R_{\ulHZ}}{\times} Y_j\,,
\]
where the map to the first copy of $Y_j$ is the identity and the map to the second copy is given by the action of $I_{\ul\CG_0}\!(Q)$ on $Y_j$. In particular, $\Theta$ is a monomorphism in the sense stated in Theorem~\ref{Uniformization2}\ref{Uniformization2_A}. 
\end{lemma}

\begin{proof}
That the two definitions of a monomorphism given in Theorem~\ref{Uniformization2}\ref{Uniformization2_A} are equivalent follows from \cite[Tag~\href{https://stacks.math.columbia.edu/tag/04Z7}{04Z7}]{StacksProject}.
The morphism is well defined by the $I_{\ul\CG_0}\!(Q)$-equivariance of $\Theta'$. To describe its inverse, consider a connected scheme $S\in\Nilp_{\breve R_{\ul{\hat Z}}}$ and two $S$-valued points of $Y_j$ 
$$
\ul y:=\bigl((\ul{\CL}_i,\delta_i)_i,\gamma H \bigr)~~~~ \text{and}~~~~ \ul y':=\bigl((\ul{\CL}_i',\delta'_i)_i,\gamma'H \bigr) 
$$  
which under $\Theta'$ are mapped to global $\FG$-shtukas $(\ul\CG,\check{\CT}_\delta^{-1} \gamma H)$ and $(\ul\CG',\check{\CT}_{\delta'}^{-1} \gamma'H)$ with $H$-level structures, where $\delta\colon\ul\CG \to \ul\CG_{0,S}$ and $\delta'\colon  \ul\CG'\to \ul\CG_{0,S}$ are the canonical quasi-isogenies which are isomorphisms outside $\ul\nu$ with $\wh{\ul{\Gamma}}(\delta)=(\delta_i)_i$ and $\wh{\ul{\Gamma}}(\delta')=(\delta'_i)_i$. Assume that $(\ul\CG,\check{\CT}_\delta^{-1} \gamma H)$ and $(\ul\CG',\check{\CT}_{\delta'}^{-1} \gamma'H)$ are isomorphic in $\nabla_n^{H,\ulHZ}\scrH^1(C,\FG)^{\ul\nu}(S)$ via a quasi-isogeny $\phi\colon\ul\CG\to \ul\CG'$ which is an isomorphism at the $\nu_i$ and compatible with the $H$-level structures, that is $\check{\CV}_\phi\circ\check{\CT}_\delta^{-1} \gamma H=\check{\CT}_{\delta'}^{-1} \gamma'H$; see Definition~\ref{DefRatLevelStr}. Consider the quasi-isogeny $\eta:=\delta'\phi \delta^{-1}$ from $\ul\CG_{0,S}$ to itself. By Proposition~\ref{Prop7.1} we may view $\eta$ as an element of $I_{\ul\CG_0}\!(Q)$.

Between the associated local $\BP_{\nu_i}$-shtukas we consider the corresponding quasi-isogenies 
$$
\xymatrix @C+2pc {
\ul\CL_i \ar[r]^{\wh\Gamma_{\nu_i}(\phi)} \ar[d]_{\delta_i} & \ul\CL_i' \ar[d]_{\delta'_i} \\
\ul\BL_{i,S} \ar[r]^{\wh\Gamma_{\nu_i}(\eta)} & \ul\BL_{i,S}.
}
$$
Since $\phi\colon\ul\CG\to \ul\CG'$ is an isomorphism at the $\nu_i$ the quasi-isogenies $\wh\Gamma_{\nu_i}(\phi)$ are isomorphisms. This shows that $\eta\cdot(\ul\CL_i,\delta_i):=(\ul\CL_i,\wh\Gamma_{\nu_i}(\eta)\circ\delta_i)\cong(\ul\CL'_i,\delta'_i)$ in $\breve\CM_{\ul\BL_i}^{\hat{Z}_i}(S)$. Moreover, $\eta$ sends $\gamma H\in\Isom^{\otimes}(\omega^\circ,\check{\CV}_{\ul{\CG}_0})/H$ to $\check\CV_\eta\circ\gamma H=\check\CT_{\delta'}\check\CV_\phi \check\CT_\delta^{-1}\circ\gamma H=\gamma'H$. This proves that $\eta\cdot\ul y=\ul y'$ and thus $\Theta$ is a monomorphism. Moreover, $(\eta,\ul y)$ maps to $(\ul y,\ul y')$.
\end{proof}

\begin{lemma}\label{LemmaFormEtale}
The morphisms $\Theta'$ and $\Theta$ from \eqref{EqUnifMorphPrime} and \eqref{EqUnifMorph} are formally \'etale and ind-proper.
\end{lemma}

\begin{proof}
If $H'\subset H$ is a normal subgroup, then dividing out the action of $H/H'$ on source and target of the morphism $\Theta$ (respectively $\Theta'$) for $H'$ yields the morphism $\Theta$ (respectively $\Theta'$) for $H$, because $\Theta$ and $\Theta'$ are compatible with the action of $h\in H\subset\FG(\BA^{\ul\nu})$ by the Hecke correspondences $\pi_{H'\!,H'}(h)$. Using the argument of Theorem~\ref{ThmLSGGsht1}\ref{ThmLSGGsht1_A}, which produces normal subgroups $H_2\subset H$ and $H_D\subset H_2$, it suffices to prove the lemma for $H=H_D$, where $D\subset C$ is a proper closed subscheme. Fix an element $\gamma H_D\in\Isom^{\otimes}(\omega^\circ,\check{\CV}_{\ul{\CG}_0})/H_D$ and consider the component $\prod_i \breve\CM_{\ul\BL_i}^{\hat{Z}_i}\times \{\gamma H_D\}$ of $\prod_i \breve\CM_{\ul\BL_i}^{\hat{Z}_i}\times \Isom^{\otimes}(\omega^\circ,\check{\CV}_{\ul{\CG}_0})/H_D$. The morphism 
\begin{alignat}{3}\label{EqTheta'InProof}
\Theta'\colon\prod_i \breve\CM_{\ul\BL_i}^{\hat{Z}_i} & \times \{ \gamma H_D\}\enspace & \longto & \enspace\nabla_n^{H_D,\ulHZ}\scrH^1(C,\FG)^{\ul\nu}\whtimes_{R_{\ulHZ}} \Spf\breve R_{\ulHZ}\,, \\
(\ul{\CL}_i,\delta_i)_i\enspace & \times \;\; \gamma H_D\enspace & \longmapsto & \enspace \quad(\ul\CG,\check{\CT}_\delta^{-1}\gamma H_D) \nonumber
\end{alignat}
is formally \'etale, because its composition with the finite \'etale forgetful morphism $\pi(1)_{H_D,H_\emptyset}$ from Remark~\ref{RemHeckeCorr}(b) is the morphism $\Psi_{\ul\CG_0}\colon\prod_i \breve\CM_{\ul\BL_i}^{\hat{Z}_i}\longto\nabla_n^{H_\emptyset,\ulHZ}\scrH^1(C,\FG)^{\ul\nu}\whtimes_{R_{\ulHZ}} \Spf\breve R_{\ulHZ}$ from \eqref{Eq_Psi} which is formally \'etale. This proves that $\Theta'$ is formally \'etale. For the same reason $\Theta'$ satisfies the valuative criterion for properness \cite[Th\'eor\`eme~7.3]{L-M}.

To prove that $\Theta$ is formally \'etale consider a component $\Gamma_{\bar\gamma}\backslash\prod_i \breve\CM_{\ul\BL_i}^{\hat{Z}_i}$ of the source of $\Theta$ for $\ol\gamma=\gamma H_D$, see Proposition~\ref{PropQuotientByI}, and a commutative diagram of solid arrows
\begin{equation}\label{EqDiagInProof}
\xymatrix {
\olT \ar[r]\ar[d] & \Gamma_{\bar\gamma}\backslash\prod_i \breve\CM_{\ul\BL_i}^{\hat{Z}_i} \ar[d]^{\TS\Theta} \\
T \ar[r] \ar@{-->}[ru]^{\TS ?} & \nabla_n^{H_D,\ulHZ}\scrH^1(C,\FG)^{\ul\nu}\whtimes_{R_{\ulHZ}} \Spf\breve R_{\ulHZ}
}
\end{equation}
where $\olT\subset T$ is a closed subscheme defined by a sheaf of ideals $\CI\subset \CO_T$ with $\CI^2=(0)$. We use the notation of the proof of \cite[Proposition~4.27]{AH_Local}. There $\Gamma_{\bar\gamma}\backslash\prod_i \breve\CM_{\ul\BL_i}^{\hat{Z}_i}$ is the union of open substacks $(\Gamma'_x\backslash\Gamma_{\bar\gamma})\backslash V_x$ where $\Gamma'_x\subset\Gamma_{\bar\gamma}$ is a normal subgroup of finite index, $V_x=\bigcup_{\beta\in\Gamma'_x\backslash\Gamma_{\bar\gamma}}\,\beta\!\cdot\! U_x$ is an open formal subscheme of $\Gamma'_x\backslash\prod_i \breve\CM_{\ul\BL_i}^{\hat{Z}_i}$, and $\beta\!\cdot\! U_x\subset\prod_i \breve\CM_{\ul\BL_i}^{\hat{Z}_i}$ is a formal open subscheme such that $\beta\!\cdot\! U_x\to\Gamma'_x\backslash\prod_i \breve\CM_{\ul\BL_i}^{\hat{Z}_i}$ is an open immersion. Let $\olT_x$ be the preimage of $(\Gamma'_x\backslash\Gamma_{\bar\gamma})\backslash V_x$ in $\olT$ and let $T_x$ be the open subscheme of $T$ with underlying topological space $\olT_x$. Its base change
\[
\olT'_x\;:=\;\olT_x\times_{(\Gamma'_x\backslash\Gamma_{\bar\gamma})\backslash V_x} V_x\;\longto\;\olT
\]
is finite \'etale $\Gamma'_x\backslash\Gamma_{\bar\gamma}$-Galois. By \cite[Expos\'e I, Th\'eor\`eme~8.3 and Expos\'e~IX, Proposition~2.4]{SGA1} there exists a finite \'etale $\Gamma'_x\backslash\Gamma_{\bar\gamma}$-Galois cover $T'_x\to T_x$ with $T'_x\times_{T_x}\olT_x=\olT'_x$. Under the projection map $\olT'_x\to V_x$ we let $\olT'_{x,\beta}$ be the preimage of $\beta\!\cdot\!U_x\subset V_x$ and we let $T'_{x,\beta}$ be the open subscheme of $T'_x$ with underlying topological space $\olT'_{x,\beta}$. Then we obtain a morphism $\olT'_{x,\beta}\to \beta\!\cdot\!U_x\subset \prod_i \breve\CM_{\ul\BL_i}^{\hat{Z}_i}=\prod_i \breve\CM_{\ul\BL_i}^{\hat{Z}_i}\times\{\gamma H_D\}$, and since the morphism $\Theta'$ from \eqref{EqTheta'InProof} is formally \'etale, it lifts uniquely to a morphism $f_\beta\colon T'_{x,\beta}\to \beta\!\cdot\!U_x\into V_x$. For two different $\beta,\beta'$ the two maps $f_\beta$ and $f_{\beta'}$ from $T'_{x,\beta}\cap T'_{x,\beta'}$ to $\beta\!\cdot\!U_x\cap\beta'\!\cdot\!U_x\subset\beta\!\cdot\!U_x\subset V_x$ coincide by uniqueness, because their restrictions to $\olT'_{x,\beta}\cap\olT'_{x,\beta'}$ do. Thus the morphisms $f_\beta$ for all $\beta$ glue to give a uniquely determined morphism $f\colon T'_x\to V_x$. This morphism is $\Gamma'_x\backslash\Gamma_{\bar\gamma}$-equivariant because its restriction to $\olT'_x$ is. Dividing out $\Gamma'_x\backslash\Gamma_{\bar\gamma}$ yields a uniquely determined morphism $T_x\to \Gamma_{\bar\gamma}\backslash\prod_i \breve\CM_{\ul\BL_i}^{\hat{Z}_i}$. For all $x$ the latter morphisms glue to produce the dashed arrow in diagram~\eqref{EqDiagInProof}. This proves that $\Theta$ is formally \'etale.

To show that $\Gamma_{\bar\gamma}\backslash\prod_i \breve\CM_{\ul\BL_i}^{\hat{Z}_i}\to\nabla_n^{H_D,\ulHZ}\scrH^1(C,\FG)^{\ul\nu}\whtimes_{R_{\ulHZ}} \Spf\breve R_{\ulHZ}$ satisfies the valuative criterion for properness \cite[Th\'eor\`eme~7.3]{L-M} we use that $\Gamma'_x\backslash\prod_i \breve\CM_{\ul\BL_i}^{\hat{Z}_i}\to\Gamma_{\bar\gamma}\backslash\prod_i \breve\CM_{\ul\BL_i}^{\hat{Z}_i}$ is finite \'etale $\Gamma'_x\backslash\Gamma_{\bar\gamma}$-Galois. In the analogous diagram to \eqref{EqDiagInProof} with $\olT$ replaced by $\Spec K$ and $T$ replaced by $\Spec R$ for a valuation ring $R$ with fraction field $K$, there are $x,\beta$ and a finite field extension $K'$ of $K$ such that the morphism $\Spec K \to\Gamma_{\bar\gamma}\backslash\prod_i \breve\CM_{\ul\BL_i}^{\hat{Z}_i}$ lifts to $\Spec K'\to \beta\!\cdot\!U_x$. By the ind-properness of $\Theta'$ from \eqref{EqTheta'InProof} this lifts to $\Spec R'\to \prod_i \breve\CM_{\ul\BL_i}^{\hat{Z}_i}\to\Gamma_{\bar\gamma}\backslash\prod_i \breve\CM_{\ul\BL_i}^{\hat{Z}_i}$ and proves the valuative criterion for properness. 

Finally $\Theta'$ and $\Theta$ are ind-proper, because their sources $\prod_i \breve\CM_{\ul\BL_i}^{\hat{Z}_i}$ and $\Gamma_{\bar\gamma}\backslash\prod_i \breve\CM_{\ul\BL_i}$ are an ind-closed ind-subscheme of the ind-quasi-projective ind-scheme $\prod_i \breve\CM_{\ul\BL_i}^{\hat{Z}_i}$, respectively a quotient of that.
\end{proof}

\begin{proof}[Proof of Theorem~\ref{Uniformization2}\ref{Uniformization2_B}.]
We already started with the proof in Remark~\ref{RemHeckeCorr}(a). Let us next prove that the underlying set of $\CZ$ is the isogeny class of $\ul\CG_0$. Every $K$-valued point $(\ul\CG,\gamma H)$ of $\CZ$ for a field $K$ lies in the image of $\Theta'$, and hence is of the form $\ul\CG=\delta_n^*\circ\ldots\circ\delta_1^*\,\ul\CG_{0,K}$ with an isogeny $\delta\colon\ul\CG\to\ul\CG_{0,K}$. This shows that $\CZ$ is contained in the isogeny class of $\ul\CG_0$. Conversely, let $(\ul\CG',\gamma'H)$ be a $K$-valued point of $\nabla_n^{H_D,\ulHZ}\scrH^1(C,\FG)^{\ul\nu}\whtimes_{R_{\ulHZ}}\Spec\BaseFldInSectUnif$ in the isogeny class of $\ul\CG_0$, and let $\delta'\colon\ul\CG'\to\ul\CG_{0,K}$ be an isogeny. Let $\ul\CL_i:=\wh\Gamma_{\nu_i}(\ul\CG')$ and $\delta_i:=\wh\Gamma_{\nu_i}(\delta')\colon\ul\CL_i\to\ul\BL_i$ and $\gamma H:=\check\CV_{\delta'}\gamma'H\in\Isom^{\otimes}(\omega^\circ,\check{\CV}_{\ul{\CG}_0})/H$. Then $(\ul\CL_i,\delta_i)\times \gamma H$ is a $K$-valued point of the source of $\Theta'$ which is mapped under $\Theta'$ to $(\ul\CG,\check\CT_\delta^{-1}\gamma H)$, where $\ul\CG:=\delta_n^*\circ\ldots\circ\delta_1^*\,\ul\CG_{0,K}$ and $\delta\colon\ul\CG\to\ul\CG_{0,K}$ is the isogeny with $\wh\Gamma_{\nu_i}(\delta)=\delta_i$ which is an isomorphism outside $\ul\nu$. The isogeny $\delta^{-1}\delta'\colon\ul\CG'\to\ul\CG$ satisfies $\wh\Gamma_{\nu_i}(\delta^{-1}\delta')=\id$ and $\check\CV_{\delta^{-1}\delta'}\circ\gamma' H=\check\CT_\delta^{-1}\gamma H$, and so $(\ul\CG',\gamma'H)\cong(\ul\CG,\gamma H)$ in $\nabla_n^{H_D,\ulHZ}\scrH^1(C,\FG)^{\ul\nu}(K)$. The point $(\ul\CL_i,\delta_i)\times \gamma H$ lies on an irreducible component of $\prod_i \breve X_{Z_i}(\ul\BL_i)\times \Isom^{\otimes}(\omega^\circ,\check{\CV}_{\ul{\CG}_0})/H$ belonging to the $I_{\ul\CG_0}\!(Q)$-orbit of some irreducible component $T_j$. By the $I_{\ul\CG_0}\!(Q)$-equivariance of $\Theta'$ we can move the point $(\ul\CL_i,\delta_i)\times \gamma H$ to $T_j$ and then its image $(\ul\CG',\gamma'H)$ under $\Theta'$ lies in $\Theta'(T_j)\subset\CZ$ as desired.

To prove that $\CZ$ is separated over $\Spec\BaseFldInSectUnif$ we use the valuative criterion \cite[Proposition~7.8]{L-M}. Let $R$ be a valuation ring containing $\BaseFldInSectUnif$ and consider two morphisms $f_1,f_2\colon\Spec R\to \CZ$ whose restrictions $f_{i,K}$ to the fraction field $K$ of $R$ are isomorphic in $\CZ(K)$. We must show that $f_1\cong f_2$ in $\CZ(R)$. The $K$-valued point $f_{1,K}\cong f_{2,K}$ lies on $\Theta'(T_j)\subset\CZ$ for some $j$. Since $\Theta'(T_j)$ is a closed substack of $\nabla_n^{H_D,\ulHZ}\scrH^1(C,\FG)^{\ul\nu}\whtimes_{R_{\ulHZ}}\Spec\BaseFldInSectUnif$, also the two morphisms $f_1,f_2$ factor through $\Theta'(T_j)$. Since $\nabla_n^{H_D,\ulHZ}\scrH^1(C,\FG)^{\ul\nu}\whtimes_{R_{\ulHZ}}\Spec\BaseFldInSectUnif$ is separated over $\BaseFldInSectUnif$, also $\Theta'(T_j)$ is separated over $\BaseFldInSectUnif$, and so $f_1\cong f_2$. Thus $\CZ$ is separated over $\BaseFldInSectUnif$.

Now we prove that $\Theta_\CZ$ is an isomorphism. The morphism $\Theta_\CZ$ is locally of ind-finite presentation because its source $\CY:=I_{\ul\CG_0}\!(Q) \big{\backslash}\bigl(\prod_i \breve\CM_{\ul\BL_i}^{\hat{Z}_i}\times \Isom^{\otimes}(\omega^\circ,\check{\CV}_{\ul{\CG}_0})/H\bigr)$, as a locally noetherian, adic formal algebraic Deligne-Mumford stack locally formally of finite type, is locally of ind-finite presentation over $\Spf\breve R_{\ul{\hat Z}}$. Since $\nabla_n^{H,\ulHZ}\scrH^1(C,\FG)^{\ul\nu}_{/\CZ}\to\nabla_n^{H,\ulHZ}\scrH^1(C,\FG)^{\ul\nu}\whtimes_{R_{\ulHZ}} \Spf\breve R_{\ulHZ}$ is a monomorphism of ind-DM-stacks by Remark~\ref{RemHeckeCorr}(a), and $\Theta$ is formally \'etale and satisfies the valuative criterion for properness \cite[Th\'eor\`eme~7.3]{L-M} by Lemma~\ref{LemmaFormEtale}, it follows that also $\Theta_\CZ$ is formally \'etale and satisfies the valuative criterion for properness. 

We fix a representation $\rho\colon\FG\into\SL(\CV)$ as before Remark~\ref{RemRelAffineGrass} and a tuple $\ul\omega=(\omega_i)_{i=1\ldots n}$ of coweights of $\SL_r$, and consider the closed substacks $\nabla_n^{H,\ulHZ,\ul\omega}\scrH^1(C,\FG)^{\ul\nu}$ of $\nabla_n^{H,\ulHZ}\scrH^1(C,\FG)^{\ul\nu}$ from Remark~\ref{Rem5.2}, which are locally noetherian, adic formal algebraic Deligne-Mumford stacks over $\Spf \breve R_{\ul{\hat Z}}$. Their formal completion along $\CZ$
\[
\CX^{\ul\omega}\;:=\;\nabla_n^{H,\ulHZ}\scrH^1(C,\FG)^{\ul\nu}_{/\CZ}\times_{\nabla_n^{H,\ulHZ}\scrH^1(C,\FG)^{\ul\nu}}\nabla_n^{H,\ulHZ,\ul\omega}\scrH^1(C,\FG)^{\ul\nu}
\]
are closed substacks of $\nabla_n^{H,\ulHZ}\scrH^1(C,\FG)^{\ul\nu}_{/\CZ}$ and likewise locally noetherian, adic formal algebraic Deligne-Mumford stacks over $\Spf \breve R_{\ul{\hat Z}}$ by \cite[Proposition~A.14]{Har1}. We may write
\[
\CX\;:=\;\nabla_n^{H,\ulHZ}\scrH^1(C,\FG)^{\ul\nu}_{/\CZ}\;=\;\dirlim\CX^{\ul\omega}\,.
\]
The base change $\CY^{\ul\omega}:=\CY\times_\CX\CX^{\ul\omega}$ is a closed formal algebraic substack of $\CY$. Since $\CX^{\ul\omega}$ and $\CY^{\ul\omega}$ are locally noetherian, adic formal algebraic stacks they have maximal ideals of definition $\CI_{\ul\omega}$ and $\CJ_{\ul\omega}$ containing $(\xi_1,\ldots,\xi_n)$. For positive integers $m$ consider the algebraic substacks $\CX_m^{\ul\omega}:=\Var(\CI_{\ul\omega}^m)$ and $\CY_m^{\ul\omega}:=\Var(\CJ^m_{\ul\omega})\subset\CY^{\ul\omega}$. They are Deligne-Mumford stacks locally of finite type over $\Spec\breve R_{\ul{\hat Z}}/(\xi_1,\ldots,\xi_n)^m$, because $\CY_m^{\ul\omega}\subset\CY$ is closed, and $\CX_m^{\ul\omega}$ is a closed substack of the Deligne-Mumford stack $\nabla_n^{H,\ulHZ,\ul\omega}\scrH^1(C,\FG)^{\ul\nu}\times_{\breve R_{\ul{\hat Z}}}\Spec\breve R_{\ul{\hat Z}}/(\xi_1,\ldots,\xi_n)^m$ which is locally of finite type over $\Spec\breve R_{\ul{\hat Z}}/(\xi_1,\ldots,\xi_n)^m$. By definition of $\CI_{\ul\omega}$ and $\CJ_{\ul\omega}$ the stacks $\CX_1^{\ul\omega}$ and $\CY_1^{\ul\omega}$ are reduced. Moreover, $\Theta_\CZ$ induces a morphism $\Theta_\CZ\colon \CY_1^{\ul\omega} \to \CX_1^{\ul\omega}$, because if $Y^{\ul\omega}\onto\CY_1^{\ul\omega}$ is a presentation, then $Y^{\ul\omega}$ is a reduced scheme, and hence $\Theta_\CZ$ induces a morphism $Y^{\ul\omega}\to(\CZ\times_{\CX}\CX^{\ul\omega})_\red=\CX_1^{\ul\omega}$ which descends to a morphism $\Theta_\CZ\colon\CY_1^{\ul\omega}\to\CX_1^{\ul\omega}$. In particular, $\Theta_\CZ^*(\CI_{\ul\omega})\subset\CJ_{\ul\omega}$. This shows that $\Theta_\CZ$ induces a morphism $\CY_m^{\ul\omega}\to\CX_m^{\ul\omega}$ for every $m$, which is locally of finite presentation as a morphism between Deligne-Mumford stacks locally of finite type over $\Spec\breve R_{\ul{\hat Z}}/(\xi_1,\ldots,\xi_n)^m$.

Now part~\ref{Uniformization2_B} will follow from Lemmas~\ref{LemmaThetaRedQC} and \ref{Thetaisadic} below. More precisely, by Lemma~\ref{Thetaisadic} we have $\Theta_\CZ^*(\CI_{\ul\omega})=\CJ_{\ul\omega}$. Then $\CY_m^{\ul\omega}\to\CX_m^{\ul\omega}$ is obtained from $\CY\to\CX$ by base change, and hence is a formally \'etale morphism locally of finite presentation of algebraic stacks. Since $\Theta$ is a monomorphism by part~\ref{Uniformization2_A} and $\CX\to\nabla_n^{H,\ulHZ}\scrH^1(C,\FG)^{\ul\nu}$ is a monomorphism by Remark~\ref{RemHeckeCorr}(a), also $\CY\to\CX$ is a monomorphism. In particular, it is relatively representable by an \'etale monomorphism of schemes; see \cite[Corollaire~8.1.3 and Th\'eor\`eme~A.2]{L-M}. In addition $\CY_m^{\ul\omega}\to\CX_m^{\ul\omega}$ is surjective by Lemma~\ref{LemmaThetaRedQC}, hence an isomorphism by \cite[IV$_4$, Th\'eor\`eme~17.9.1]{EGA}. As this holds for all $m$ and all $\ul\omega$ we conclude that $\Theta_\CZ:\CY\to\CX$ is an isomorphism of stacks.
\end{proof}

\begin{lemma}\label{LemmaThetaRedQC}
The induced morphism $\Theta_\CZ\colon \CY_m^{\ul\omega} \to \CX_m^{\ul\omega}$ is quasi-compact and surjective.
\end{lemma}

\begin{proof}
The assertion only depends on the underlying topological spaces $|\CX_m^{\ul\omega}|=|\CX_1^{\ul\omega}|$ and $|\CY_m^{\ul\omega}|=|\CY_1^{\ul\omega}|$; see \cite[Chapitre~5]{L-M}. So we may assume that $m=1$. The morphism $\Theta_\CZ\colon \CY_1^{\ul\omega} \to \CX_1^{\ul\omega}=(\CZ\times_{\CX}\CX^{\ul\omega})_\red$ is surjective by definition of $\CZ$ as the image of $\Theta'$.

We prove that the morphism is quasi-compact. For this purpose we may pass to finer level. Namely, we choose a proper closed subscheme $D\subset C$ with $H_D\subset H$ and use on source and target of $\Theta_\CZ$ the finite \'etale and surjective morphisms $\pi(1)_{H_D,H}$ from Remark~\ref{RemHeckeCorr}(b). Then the morphism $\Theta_\CZ\colon \CY_1^{\ul\omega} \to \CX_1^{\ul\omega}$ for $H_D$ is obtained by base change from the one for $H$. So we may assume $H=H_D$ from now on. Let $S$ be a quasi-compact scheme and $f\colon S\to\CX_1^{\ul\omega}$ be a morphism given by a global $\FG$-shtuka $(\ul\CG',\beta'_Q H_D)$ with $H_D$-level structure over $S$, where $\beta'\in\Isom^\otimes(\omega^\circ_{\BO^{\ul\nu}},\check\CT_{\ul\CG'})$; see the proof of Theorem~\ref{H_DL-Str}. We must show that $S\times_{\CX_1^{\ul\omega}}\CY_1^{\ul\omega}$ is quasi-compact. Consider the topological space $|\CZ|$ underlying $\CZ$ and the set $\{T_j\}_{j\in J}$ of representatives of $I_{\ul\CG_0}\!(Q)$-orbits of the irreducible components of the scheme $\prod_i\breve X_{Z_i}(\ul\BL_i)\times \Isom^{\otimes}(\omega^\circ,\check{\CV}_{\ul{\CG}_0})/H_D$ from Theorem~\ref{Uniformization2}\ref{Uniformization2_B}. Every point $z\in|\CZ|$ lies only on finitely many $\Theta'(T_j)$, see Remark~\ref{RemHeckeCorr}(a). Let $J(z)\subset J$ be the set with $j\in J(z)$ if and only if $z\in\Theta'(T_j)$. The open substack
\[
U_z\;:=\;\bigcup_{j\in J(z)}\Theta'(T_j)\;\setminus\bigcup_{j\notin J(z)}\Theta'(T_j)\;\subset\;\CZ
\]
\renewcommand{\ell}{l}
contains $z$, and hence all the $U_z$ cover $\CZ$. Let $s\in S$ be a point and set $z=f(s)\in|\CZ|$. The preimage $f^{-1}(U_{f(s)})$ of $U_{f(s)}$ in $S$ contains $s$. We choose an affine open neighborhood $S_s$ of $s$ in $S$ which is contained in $f^{-1}(U_{f(s)})$. Then the $S_s$ cover $S$, and hence already $S=S_{s_1}\cup\ldots\cup S_{s_r}$ for finitely many points $s_\ell\in S$, because $S$ is quasi-compact. The scheme $S'$ defined as the finite disjoint union
\[
S'\;:=\;\coprod_{\ell=1}^r\;\coprod_{j\in J(f(s_\ell))} S_{s_\ell} \underset{f,\CZ,\Theta'}{\times} T_j
\]
is quasi-compact, because $\Theta'\colon T_j\to\CZ$ is proper by Remark~\ref{RemHeckeCorr}(a) and so $S_{s_\ell} \times_\CZ T_j$ is proper over the affine scheme $S_{s_\ell}$. The projection $S'\to S$ is surjective, because every point $s\in S$ lies in one $S_{s_\ell}$, and then $f(s)\in U_{f(s_\ell)}\subset\bigcup_{j\in J(f(s_\ell))}\Theta'(T_j)$ has a preimage in one of the $T_j$. On every component $S'_{\ell,j}:= S_{s_\ell} \times_\CZ T_j$ of $S'$ the projection onto $T_j$ defines a tuple $(\ul\CL_i,\delta_i)\in\breve\CM_{\ul\BL_i}^{\hat{Z}_i}(S_{\ell,j})$ for $i=1,\ldots,n$. Let $\ul\CG:=\delta_n^*\circ\ldots\circ\delta_1^*\,\ul\CG_{0,S_{\ell,j}}$ and let $\delta\colon\ul\CG\to\ul\CG_{0,S_{\ell,j}}$ be the induced quasi-isogeny. By definition of $\Theta'$ via the morphism $\Psi_{\ul\CG_0}$, there is an isomorphism $\alpha\colon\ul\CG'\isoto\ul\CG$ of global $\FG$-shtukas over $S_{\ell,j}$. Then $\gamma_{\ell,j}H_D:=\check\CT_{\delta\alpha}\circ\beta_Q'H_D\in\Isom^{\otimes}(\omega^\circ,\check{\CV}_{\ul{\CG}_0})/H_D$. The projection $S_{\ell,j}\to T_j$ together with the element $\gamma_{\ell,j}H_D$ defines the upper horizontal morphism in the commutative diagram
\[
\xymatrix @C+1pc {
S' \ar@{->>}[d] \ar[rr] & & \prod_i\breve X_{Z_i}(\ul\BL_i)\times \Isom^{\otimes}(\omega^\circ,\check{\CV}_{\ul{\CG}_0})/H_D \ar[d]^{\TS\Theta'} \\
S \ar[r]^{\TS f} & \CX_1^{\ul\omega} \ar[r] & \nabla_n^{H_D,\ulHZ}\scrH^1(C,\FG)^{\ul\nu}\whtimes_{R_{\ulHZ}} \Spf\breve R_{\ulHZ}\;.
}
\]
Then we can consider the surjective morphisms 
\[
\xymatrix @C=-4pc @R+1pc {
I_{\ul\CG_0}\!(Q) \times S' \ar[r]^{\TS\sim\qquad\qquad} &  **{!L(0.5) !U(0.5)}
\objectbox{\; S' \underset{\nabla_n^{H_D,\ulHZ}\scrH^1(C,\FG)^{\ul\nu}\whtimes_{R_{\ulHZ}} \Spf\breve R_{\ulHZ}}{\times} \prod_i\breve X_{Z_i}(\ul\BL_i)\times \Isom^{\otimes}(\omega^\circ,\check{\CV}_{\ul{\CG}_0})/H_D} \ar@{->>}[d] \\
S\times_{\CX_1^{\ul\omega}}\CY_1^{\ul\omega} & \ar@{->>}[l] S'\times_{\CX_1^{\ul\omega}}\CY_1^{\ul\omega} 
}
\]
in which the isomorphism in the upper row comes from Lemma~\ref{LemmaThetaismono}. By the $I_{\ul\CG_0}\!(Q)$-equivariance of $\Theta'$ we obtain a surjective map $S'\onto S\times_{\CX_1^{\ul\omega}}\CY_1^{\ul\omega}$, and hence $S\times_{\CX_1^{\ul\omega}}\CY_1^{\ul\omega}$ is quasi-compact by \cite[Tag~\href{https://stacks.math.columbia.edu/tag/04YC}{04YC}]{StacksProject} as desired.
\end{proof}

\begin{lemma}\label{Thetaisadic}
For every $\ul\omega$ we have $\Theta_\CZ^*(\CI_{\ul\omega})=\CJ_{\ul\omega}$. In particular the morphism $\Theta_\CZ\colon\CY^{\ul\omega}\to \CX^{\ul\omega}$ of formal algebraic stacks is adic. 
\end{lemma}

\begin{proof} 
Let $X^{\ul\omega} \onto \CX_1^{\ul\omega}=\Var(\CI_{\ul\omega})$ be a presentation. By Lemma~\ref{LemmaThetaismono} and our arguments for the proof of Theorem~\ref{Uniformization2}\ref{Uniformization2_B} given above, we see that $\Theta_{\CZ,m}\colon\CY_m^{\ul\omega}\times_{\CX^{\ul\omega}}X^{\ul\omega}=\CY_m^{\ul\omega}\times_{\CX^{\ul\omega}_m}X^{\ul\omega} \to X^{\ul\omega}$ is a monomorphism locally of finite presentation satisfying the valuative criterion for properness. Since it is quasi-compact by Lemma~\ref{LemmaThetaRedQC} it is a proper monomorphism, hence a closed immersion of schemes by \cite[Corollaire~A.2.2]{L-M}
. Since $\Theta_{\CZ,m}$ is surjective by Lemma~\ref{LemmaThetaRedQC} and $X^{\ul\omega}$ is reduced, it must be an isomorphism for all $m$. Therefore $\CY_m^{\ul\omega}\times_{\CX^{\ul\omega}} X^{\ul\omega}=X^{\ul\omega}=\CY_1^{\ul\omega}\times_{\CX^{\ul\omega}} X^{\ul\omega}$ and $\CY^{\ul\omega}\times_{\CX^{\ul\omega}} X^{\ul\omega}=\dirlim\CY_m^{\ul\omega}\times_{\CX^{\ul\omega}} X^{\ul\omega}=\CY_1^{\ul\omega}\times_{\CX^{\ul\omega}} X^{\ul\omega}$. This shows that $\Var(\Theta_\CZ^*\CI_{\ul\omega})=\CY^{\ul\omega}\times_{\CX^{\ul\omega}}\CX_1^{\ul\omega}=\CY_1^{\ul\omega}\times_{\CX^{\ul\omega}}\CX_1^{\ul\omega}\subset\CY_1^{\ul\omega}=\Var(\CJ_{\ul\omega})$ as closed substacks of $\CY^{\ul\omega}$. Therefore $\CJ_{\ul\omega}\subset\Theta_\CZ^*(\CI_{\ul\omega})$. With the opposite inclusion established above $\CJ_{\ul\omega}=\Theta_\CZ^*(\CI_{\ul\omega})$, and hence $\Theta_\CZ$ is adic. 
\end{proof}

%
%

\vfill

\begin{minipage}[t]{0.5\linewidth}
\noindent
Esmail Arasteh Rad\\
Universit\"at M\"unster\\
Mathematisches Institut \\
Einsteinstr.~62\\
D -- 48149 M\"unster
\\ Germany
\\[1mm]
\end{minipage}
\begin{minipage}[t]{0.45\linewidth}
\noindent
Urs Hartl\\
Universit\"at M\"unster\\
Mathematisches Institut \\
Einsteinstr.~62\\
D -- 48149 M\"unster
\\ Germany
\\[1mm]
\href{http://www.math.uni-muenster.de/u/urs.hartl/index.html.en}{www.math.uni-muenster.de/u/urs.hartl/}
\end{minipage}

\end{document}